\documentclass[12pt]{amsart}
\usepackage{amsmath,amssymb,mathrsfs,euscript,color,stmaryrd,mathabx}
\usepackage[colorlinks]{hyperref}
\usepackage{pdfsync}
\usepackage[backgroundcolor=green!40, linecolor=green!40]{todonotes}
\input{xypic}

\theoremstyle{plain}
\newtheorem*{Theorem}{Theorem}
\newtheorem*{Corollary}{Corollary}
\newtheorem*{Proposition}{Proposition}
\newtheorem*{Lemma}{Lemma}
\theoremstyle{definition}
\newtheorem*{Remark}{Remark}
\newtheorem*{Convention}{Convention}

\newtheorem*{Definition}{Definition}

\newtheorem*{Notation}{Notation}

\makeatletter
\def\section{\def\@secnumfont{\mdseries}%
  \@startsection{section}{1}%
  \z@{.7\linespacing\@plus\linespacing}{.5\linespacing}%
  {\normalfont\scshape\centering}}
\def\subsection{\def\@secnumfont{\normalfont\bfseries}%
 \@startsection{subsection}{2}%
 \z@{.5\linespacing\@plus.7\linespacing}{-.5em}%
 {\normalfont\bfseries}}
\makeatother

\def\labelenumi{{\rm(\roman{enumi})}}
\def\theenumi{{\rm(\roman{enumi})}}

\numberwithin{equation}{section}

\def\bdots{\mathinner{\mkern1mu\raise1pt\hbox{.}\mkern2mu\raise4pt\hbox{.}
           \mkern2mu\raise7pt\vbox{\kern7pt\hbox{.}}\mkern1mu}}

\def\oF{{\mathfrak o}_F}
\def\oFx{{\mathfrak o}_F^\times}
\def\pF{{\mathfrak p}_F}
\def\kF{{k_F}}

\def\oE{{{\mathfrak o}_E}}
\def\oEx{{\mathfrak o}_E^\times}
\def\pE{{\mathfrak p}_E}

\def\oEo{{\mathfrak o}_E^\so}

\def\oEi{{\mathfrak o}_E^i}

\def\oEj{{\mathfrak o}_E^j}
\def\oEW{{\mathfrak o}_{W}}
\def\pEW{{\mathfrak p}_{W}}

\def\til#1{\widetilde {#1}}

\def\fM{{\mathfrak M}}

\def\End{{\operatorname{End}}}
\def\Aut{{\operatorname{Aut}}}
\def\Hom{{\operatorname{Hom}\mspace{0.5mu}}}
\def\dim{\operatorname{dim}}

\def\det{\operatorname{det}}
\def\deg{\operatorname{deg}}
\def\Ad{\operatorname{Ad}}
\def\Ind{\operatorname{Ind}}
\def\cInd{\operatorname{c-Ind}}
\def\Irr{\operatorname{Irr}}
\def\Cusp{\operatorname{Cusp}}

\def\gcd{\operatorname{gcd}}
\def\lcm{\operatorname{lcm}}

\def\val{\operatorname{val}}
\def\Sym{\operatorname{Sym}}
\def\Jord{\operatorname{Jord}}
\def\Red{\operatorname{Red}}
\def\IJord{\operatorname{IJord}}

\def\Gal{\operatorname{Gal}}
\def\Ver{\operatorname{Ver}}

\def\G{{\mathbb G}}
\def\Sp{\operatorname{Sp}}

\def\GL{\operatorname{GL}}
\def\PGL{\operatorname{PGL}}
\def\SL{\operatorname{SL}}
\def\SO{\operatorname{SO}}
\def\St{\operatorname{St}}
\def\sp{{\mathfrak{sp}}}
\def\U{{\mathrm U}}
\def\HP{\mathcal G}

\def\tr{\operatorname{tr}}

\def\into{\hookrightarrow}

\def\CE{{\mathcal E}}
\def\CEE{{\CE\!\CE}}

\def\CH{{\mathcal H}}
               
\def\CH{{\mathcal H}}
\def\CK{{\mathcal K}}

\def\CT{{\mathcal T}}

\def\CV{{\mathcal V}}

\def\BZ{{\mathbb Z}}

\def\BC{{\mathbb C}}

\def\b{{\beta}}

\def\L{{\Lambda}}

\def\inv{\sigma} 

\def\bs{\boldsymbol}
\def\SpFV{G}

\def\G{{\mathscr G}}
\def\H{{\mathscr H}}
\def\V{{\mathcal V}}
\def\rP{{\mathbf r}_P}
\def\Supp{\operatorname{Supp}}
\def\diag{\operatorname{diag}}
\def\so{{\mathsf o}}
\def\Pp{{\mathscr P}}
\def\Ss{{\mathscr S}}

\begin{document}

\title{Jordan blocks of cuspidal representations of symplectic groups}

\author{Corinne Blondel}
\address{CNRS -- IMJ--PRG, Universit{\'e} Paris Diderot, Case 7012, 75205 Paris Cedex 13, France.}
\email{corinne.blondel@imj-prg.fr}

\author{Guy Henniart}
\address{Universit{\'e} de Paris-Sud, Laboratoire de Math{\'e}matiques d'Orsay, Orsay Cedex, F-91405.}
\email{Guy.Henniart@math.u-psud.fr}

\author{Shaun Stevens}
\address{School of Mathematics, University of East Anglia, Norwich Research Park, Norwich NR4 7TJ, United Kingdom}
\email{Shaun.Stevens@uea.ac.uk}

\thanks{The authors would like to thank the Schroedinger Institute in Vienna, where this work was initiated back in winter~2009. The first and second authors would like to thank the University of East Anglia for hosting them during several visits in the course of the present work. The second author would like to thank the Universit{\'e} de Paris-Sud and the Institut Universitaire de France. The third author would like to thank the Universit{\'e} Paris~7 for inviting and hosting him in spring~2012. The third author was supported by EPSRC grants EP/G001480/1 and EP/H00534X/1, and by the Heilbronn Institute for Mathematical Research for the end of the writing-up period.} 

\begin{abstract}
Let~$G$ be a symplectic group over a nonarchimedean local field of characteristic zero and odd residual characteristic. Given an irreducible cuspidal representation of~$G$, we determine its Langlands parameter (equivalently, its \emph{Jordan blocks} in the language of M\oe glin) in terms of the local data from which the representation is explicitly constructed, up to a possible unramified twist in each block of the parameter. We deduce a Ramification Theorem for~$G$, giving a bijection between the set of \emph{endo-parameters} for~$G$ and the set of restrictions to wild inertia of discrete Langlands parameters for~$G$, compatible with the local Langlands correspondence. The main tool consists in analysing the intertwining Hecke algebra of a good cover, in the sense of Bushnell--Kutzko, for parabolic induction from a cuspidal representation of~$G\times\GL_n$, seen as a maximal Levi subgroup of a bigger symplectic group, in order to determine its (ir)reducibility; a criterion of M\oe glin then relates this to Langlands parameters.
\end{abstract}

\date{\today}
\maketitle

\section*{Introduction}

\subsection{}\label{0.1new} 
Let~$F$ be a locally compact nonarchimedean local field of odd residual characteristic and denote by~$W_F$ the Weil group of~$F$. Let~$G$ be the symplectic group preserving a nondegenerate alternating form on a~$2N$-dimensional~$F$-vector space. The local Langlands conjectures for~$G$ (now a theorem of Arthur~\cite{Arthur} when~$F$ has characteristic zero) stipulate that to an irreducible (smooth, complex) representation~$\pi$ of~$G$ is attached a Langlands parameter, and the representations with a given parameter form a finite set of isomorphism classes, called an~$L$-packet for~$G$. 

Since the symplectic group is split, Langlands parameters for~$G$ are simply continuous homomorphisms~$\phi$ from~$W_F \times \SL_2(\BC)$ into the dual group~$\hat G = \SO_{2N+1}(\BC)$, taken up to conjugation, such that the~$(2N+1)$-dimensional representation~$\iota \circ \phi$ of~$W_F \times \SL_2(\BC)$, obtained from the inclusion~$\iota$ of~$\hat G$ into~$\GL_{2N+1}(\BC)$, is semisimple. If~$\pi$ is a discrete series representation of~$G$, then its parameter~$\phi$ is discrete, that is, cannot be conjugated into a proper parabolic subgroup of~$\hat G$; equivalently,~$\iota \circ \phi$ is the direct sum of inequivalent irreducible orthogonal representations of~$W_F \times \SL_2(\BC)$, and has determinant~$1$. In that case giving~$\iota \circ \phi$ up to equivalence is the same as giving~$\phi$ up to conjugation in~$\hat G$. 

On the other hand, we have an explicit description of the cuspidal representations of~$G$ via the theory of types~\cite{S5}, in the spirit of the classification of the irreducible representations of~$\GL_n(F)$ of Bushnell--Kutzko~\cite{BK}. It is our goal in this paper to describe as much as possible of the Langlands parameter of a cuspidal representation of~$G$ from its explicit construction. We will denote by~$\Phi^{\mathsf{cusp}}(G)$ the subset of discrete Langlands parameters consisting of those parameters with a cuspidal representation in the corresponding~$L$-packet (see paragraph~\ref{0.2new} below for a more detailed description).

\subsection{}\label{0.1bnew}
At the technical and arithmetic heart of the construction of cuspidal representations of~$G$ and~$\GL_n(F)$ is the theory of \emph{endo-classes} of simple characters -- families of very special characters of compact open subgroups. An irreducible cuspidal representation of~$\GL_n(F)$ contains, up to conjugacy, a unique such simple character and thus determines an endo-class. By considering the endo-classes in its cuspidal support, an arbitrary irreducible representation of~$\GL_n(F)$ then determines a formal sum of endo-classes (with multiplicities), which we call an \emph{endo-parameter of degree~$n$} (see paragraph~\ref{2.7new}). We write~$\CEE_{n}(F)$ for the set of endo-parameters of degree~$n$.

Similarly, an irreducible cuspidal representation of~$G$ is constructed from a \emph{semi}simple character, and thus also comes from an endo-parameter, the weighted formal sum of the endo-classes of its simple components; moreover, the semisimple character is self-dual so that every endo-class appearing must also be self-dual. Thus the construction of an irreducible cuspidal representation of~$G$ gives rise to a self-dual endo-parameter of degree~$2N$. We write~$\CEE_{2N}^{\mathsf{sd}}(F)$ for the set of these self-dual endo-parameters.
 
\subsection{}\label{0.1cnew}
The notions of endo-class and endo-parameter admit an instructive interpretation via the local Langlands correspondence. Denote by~$\Pp_F$ the wild ramification subgroup of the Weil group~$W_F$. Then the (First) Ramification Theorem~\cite[8.2~Theorem]{BH2} says that there is a unique bijection between the set of endo-classes over~$F$ and the set of~$W_F$-orbits of irreducible complex representations of~$\Pp_F$, which is compatible with the local Langlands correspondence for general linear groups. This then induces a bijection, again compatible with the Langlands correspondence, between the set of endo-parameters of degree~$n$ and the set of equivalence classes of~$n$-dimensional complex representations of~$\Pp_F$ which are invariant under conjugation by~$W_F$ (see~\ref{thm:ramGLSS}~Theorem for a precise statement). We call these representations of~$\Pp_F$ \emph{wild parameters}.

Our first main result (or, rather, the last in the scheme of proof) is an analogous ramification theorem for the symplectic group~$G$. First we see that the bijection above restricts to a bijection between self-dual endo-classes and self-dual~$W_F$-orbits of irreducible complex representations of~$\Pp_F$. (Note that we really mean that the \emph{orbit} is self-dual: the only self-dual irreducible complex representation of~$\Pp_F$ is the trivial representation, since~$p$ is odd.) We say that a~$(2N+1)$-dimensional wild parameter is \emph{discrete self-dual} if it is a sum of self-dual~$W_F$-orbits of irreducible complex representations of~$\Pp_F$, and write~$\Psi_{2N+1}^{\mathsf{sd}}(F)$ for the set of such wild parameters. These are precisely the restrictions to wild inertia of discrete Langlands parameters. We prove the following Ramification Theorem for~$G$ (see the end of the introduction for remarks on the characteristic).

\begin{Theorem}[\ref{thm:ramG}~Theorem]
Suppose~$F$ is of characteristic zero. There is a unique bijection~$\CEE_{2N}^{\mathsf{sd}}(F)\to\Psi_{2N+1}^{\mathsf{sd}}(F)$ which is compatible with the Langlands correspondence for cuspidal representations of~$\SpFV$:
\[
\xymatrix{
\Cusp(\SpFV) \ar@{->>}[r]\ar@{->>}[d] & \Phi^{\mathsf{cusp}}(G) \ar@{->>}[d] \\
\CEE_{2N}^{\mathsf{sd}}(F) \ar^{\sim\ }[r]& \Psi_{2N+1}^{\mathsf{sd}}(F)
}
\]
\end{Theorem}

The bijection here is not just that in the case of general linear groups (indeed, the degree has changed): one must first take the square of every endo-class in the support of the endo-parameter, then map across using the bijection for general linear groups, and finally add the trivial representation of~$\Pp_F$.

\subsection{}\label{0.1dnew}
The Ramification Theorem for~$G$ is in fact a consequence of rather more precise results, proved on the automorphic side of the Langlands correspondence. To explain the connection, we recall in more detail the structure of discrete Langlands parameters, and the results of M\oe glin.

There is, up to isomorphism, exactly one irreducible~$m$-dimensional representation~$\St_m$ of~$\SL_2(\BC)$, for each~$m\ge 1$. Thus an irreducible representation of~$W_F \times \SL_2(\BC)$ is a tensor product~$\sigma \otimes \St_m$, where~$\sigma$ is an irreducible representation of~$W_F$; moreover it is orthogonal if and only if either~$\sigma$ is self-dual symplectic and~$m$ is even, or~$\sigma$ is self-dual orthogonal and~$m$ is odd. By the Langlands correspondence for~$\GL_n$~\cite{LRS, HT, H}, such a~$\sigma$ is the Langlands parameter of a (single) cuspidal representation~$\rho$ of~$\GL_n(F)$, where~$n=\dim \sigma$. Saying that~$\sigma$ is self-dual is saying that~$\rho$ is self-dual (i.e. isomorphic to its contragredient), and~$\sigma$ is then symplectic (resp. orthogonal) if the Langlands--Shahidi~$L$-function~$L(s, \Lambda^2, \rho)$ (resp.~$L(s, \Sym^2, \rho)$) has a pole at~$s=1$~\cite{HeIMRN}, in which case we say that~$\rho$ is of symplectic (resp. orthogonal) type.

In conclusion, a discrete parameter~$\phi$ for~$G$ can be given by a set of (distinct) pairs~$(\rho_i, m_i)$, where~$\rho_i$ is an isomorphism class of irreducible cuspidal representations of~$\GL_{n_i}(F)$, with~$n_i$ and~$m_i$ positive integers, and 
\begin{itemize}
\item $\sum_i n_i m_i = 2N+1$, 
\item each~$\rho_i$ is self-dual, of symplectic type if~$m_i$ is even and of orthogonal type if~$m_i$ is odd, 
\item if~$\omega_i$ is the central character of~$\rho_i$ then~$\prod_i \omega_i^{m_i}=1$. 
\end{itemize}

\subsection{}\label{0.2new} 
If~$\pi$ is an irreducible cuspidal representation of~$G$ and~$\phi$ is its parameter, M\oe glin~\cite{Mo} has given a criterion to determine the set attached to~$\phi$ as above, i.e. the pairs~$(\rho_i,m_i)$ that she calls the ``Jordan blocks'' of~$\pi$; we write~$\Jord(\pi)$ for this set of pairs. Let us explain her results.

For any positive integer~$n$, the group~$\GL_n(F) \times G$ appears naturally as a standard maximal Levi subgroup of~$\Sp_{2(N+n)}(F)$. If~$\rho$ is a cuspidal representation of~$\GL_n(F)$ we can form the parabolically induced representation~$\rho \nu^s \rtimes \pi$ (we use normalized induction and induce via the standard parabolic), where~$s$ is here a \emph{real} parameter and~$\nu$ is the character~$g \mapsto |\det g|_F$ of~$\GL_n(F)$. If no unramified twist of~$\rho$ is self-dual then~$\rho \nu^s \rtimes \pi$ is always irreducible. On the other hand, if~$\rho$ is self-dual, there is a unique~$s_\pi(\rho)\ge 0$ such that~$\rho \nu^s \rtimes \pi$ is reducible if and only if~$s = \pm s_\pi(\rho)$.

We define the \emph{reducibility set}~$\Red (\pi)$ to be the set of isomorphism classes of cuspidal representations~$\rho$ of some~$\GL_n(F)$, with~$n \ge 1$, for which~$2s_\pi(\rho)-1$ is a positive integer. Indeed, it is known that~$2s_\pi(\rho)$ is an integer~\cite{MoTa}, so the condition for~$\rho$ to lie in~$\Red(\pi)$ is that~$s_\pi(\rho)$ is neither~$0$ nor~$1/2$. The \emph{Jordan set}~$\Jord(\pi)$ is then the set of pairs~$(\rho, m)$, where~$\rho \in \Red(\pi)$ and~$2s_\pi(\rho)-1= m + 2k$, for some integer~$k \ge 0$.

From its construction,~$\Jord(\pi)$ is ``without holes'' in the sense that, if it contains~$(\rho, m)$ then it also contains~$(\rho, m-2)$ whenever~$m-2 > 0$. However there may be discrete series non-cuspidal representations of~$G$ with the same parameter as~$\pi$; this happens as soon as~$\Jord(\pi)$ contains a pair~$(\rho,m)$ with~$m>1$. For the number of cuspidal representations of~$G$ with a given parameter (without holes), see~\cite{Mo5} (recalled in paragraph~\ref{7.4new} below).

\subsection{}\label{0.3new}
The results of M\oe glin described in the previous paragraph now say that, in order to determine the Langlands parameter of an irreducible cuspidal representation~$\pi$ of~$G$, we need only compute the reducibility points~$s_\pi(\rho)$, for~$\rho$ an irreducible self-dual representation of some~$\GL_n(F)$. Moreover, we need only find \emph{enough} reducibility points~$s_\pi(\rho)\ge 1$ to fill the parameter.

In order to compute these reducibility points, we use Bushnell--Kutzko's theory of types and covers~\cite{BK1}. The representation~$\pi$ takes the form~$\cInd_{J_\pi}^G \lambda_\pi$, for some irreducible representation~$\lambda_\pi$ of a compact open subgroup~$J_\pi$; this pair~$(J_\pi,\lambda_\pi)$ is a \emph{type} for~$\pi$. Similarly, we have a Bushnell--Kutzko type~$(\tilde J_\rho,\tilde \lambda_\rho)$ for~$\rho$. Moreover, from~\cite{MS} we have a \emph{cover}~$(J,\lambda)$ in~$\Sp_{2(N+n)}(F)$ of~$(\tilde J_\rho\times J_\pi,\tilde\lambda_\rho\otimes\lambda_\pi)$. 

The reducibility of the parabolically induced representation~$\rho\nu^s\rtimes\pi$ for \emph{complex}~$s$ is translated, via category equivalence, to the reducibility of induction from modules over the spherical Hecke algebra~$\CH(GL_n(F)\times G,\tilde\lambda_\rho\otimes\lambda_\pi)$ to~$\CH(\Sp_{2(N+n)}(F),\lambda)$. The former algebra is isomorphic to~$\mathbb C[Z^{\pm 1}]$, while the latter is a Hecke algebra on an infinite dihedral group, with two generators each satisfying a quadratic relation of the form~$(T+1)(T-q^r)$, with~$r\ge 0$ and integer and~$q$ the cardinality of the residue field of~$F$. The results of~\cite{Bl} then translate the values of the parameters~$r$ for the two generators into the \emph{real parts} of those~$s\in\mathbb C$ for which~$\rho\nu^s\rtimes \pi$ is reducible.

In the inertial class~$[\rho]=\{\rho\nu^s\mid s\in\mathbb C\}$, there are precisely two inequivalent self-dual representations, and we write~$\rho'$ for the other one. Thus the method described above allows one to compute the set~$\{s_\pi(\rho),s_\pi(\rho')\}$ but not to distinguish between the two values if they are distinct. Thus our method computes the \emph{inertial Jordan set}~$\IJord(\pi)$, which is the \emph{multiset} of pairs~$([\rho],m)$, such that~$(\rho,m)\in\Jord(\pi)$.

\subsection{}\label{0.4new}
According to the previous paragraph, computing~$\IJord(\pi)$ explicitly comes down to computing the parameters in the quadratic relations for the spherical Hecke algebra of the cover. We do this in two steps.

First, we consider the special case when the semisimple character~$\theta_\pi$ in~$\pi$, from which the type~$(J_\pi,\lambda_\pi)$ is built, is in fact \emph{simple}. In this case, it determines a self-dual endo-class~$\bs\Theta$ and we consider only those irreducible cuspidal representations of some~$\GL_n(F)$ which have endo-class~$\bs\Theta^2$. We prove that just these representations already give us enough to fill the Jordan set (see~\ref{thm:simple}~Theorem) and describe an algorithm to determine~$\IJord(\pi)$ (see paragraph~\ref{6.8new}). 

Here the computation of the parameters can be done using results of Lusztig on finite reductive groups: if~$\bs\Theta$ is the trivial endo-class~$\bs\Theta_0$, so that we are in depth zero, this was done already in~\cite{LS}; otherwise, the groups in question are the reductive quotients of maximal parahoric subgroups in a unitary group (ramified or unramified). There is also an added subtlety which does not arise in the depth zero case: two signature characters of certain permutations (coming from a comparison of so-called \emph{beta-extensions}) cause an extra twist which must be taken care of in the algorithm and counting.

In the second step, we consider an arbitrary irreducible cuspidal representation~$\pi$ and reduce to the first case. More precisely, the semisimple character~$\theta_\pi$ determines by restriction its simple components~$\theta_i$, for~$0\le i\le l$, whence endo-classes~$\bs\Theta_i$. From the construction of the type~$(J_\pi,\lambda_\pi)$, we define types~$(J_i,\lambda_i)$ in symplectic groups~$\Sp_{2N_i}(F)$, with~$\sum_{i=0}^l N_i=N$, which induce to irreducible cuspidal representations~$\pi_i$ containing a simple character of endo-class~$\bs\Theta_i$. (See paragraph~\ref{2.5new} for details.) 

The reduction is obtained by showing that elements of~$\IJord(\pi)$ with endo-class~$\bs\Theta_i$ can be obtained from those of~$\IJord(\pi_i)$ by a simple twisting process, by a character of order one or two (see~\ref{thm:reduction}~Theorem). This character arises as the comparison of pairs of signature characters as in the first case, for~$\pi$ and for~$\pi_i$; the point that is both crucial and subtle is that, although we need to make two comparisons, they turn out to be equal. Now the first case, together with a dimension count, ensures that we have filled the expected size of~$\IJord(\pi)$. If~$F$ is of characteristic zero then, by the results of M\oe glin, this is indeed the entire inertial Jordan set (see~\ref{cor:together}~Corollary).

\subsection{}\label{0.6new}
From our explicit description of the set~$\IJord(\pi)$, we know the endo-class of every self-dual irreducible cuspidal representation of some~$\GL_n(F)$ which appears in~$\Jord(\pi)$. From this we deduce the following result, which gives the compatibility of taking endo-parameters with the endoscopic transfer from~$G$ to~$\GL_{2N+1}(F)$ and from which, via the results of Arthur, we deduce compatibility with the local Langlands correspondence. In the following, the map~$\iota_{2N}$ sends a (self-dual) endo-parameter~$\sum m_{\bs\Theta}\bs\Theta$ of degree~$2N$ to the endo-parameter~$\sum m_{\bs\Theta}\bs\Theta^2+\bs\Theta_0$ of degree~$2N+1$, where~$\bs\Theta_0$ denotes the trivial endo-class.

\begin{Theorem}[\ref{thm:endoparameter}~Theorem]
Suppose~$F$ has characteristic~$0$. Then the following diagram commutes.
\[
\xymatrix{
\Cusp(\SpFV) \ar[d] 
\ar[rr]^{\text{\rm transfer}\quad} && \Irr(\GL_{2N+1}(F)) \ar[d]\\
\CEE^{\mathsf{sd}}_{2N}(F) \ar@{^{(}->}[rr]_{\iota_{2N}} && \CEE_{2N+1}(F)
}
\]
\end{Theorem}

It is very tempting to think that this result could be an instance of a general theory of endo-parameters for arbitrary reductive groups, which would be in bijection with suitably-defined wild parameters and would be compatible with (twisted) endoscopy. 

\subsection{}\label{0.5new}
Let~$\pi$ be an irreducible cuspidal representation of~$G$. Having given an explicit description of~$\IJord(\pi)$, we can ask whether we can then determine~$\Jord(\pi)$ precisely; that is, given~$([\rho],m)\in\IJord(\pi)$, can we tell whether it is~$(\rho,m)$ or~$(\rho',m)$ in~$\Jord(\pi)$, where~$\rho'$ is the self-dual unramified twist of~$\rho$ which is inequivalent to~$\rho$. In certain cases the answer is yes: often the representations~$\rho,\rho'$ have opposite parities (that is, one is symplectic and the other orthogonal) and then we know that we must have the representation of symplectic type if~$m$ is even, and the one of orthogonal type if~$m$ is odd. In the exceptional case where~$\rho,\rho'$ have the same parity, we can only recover~$\Jord(\pi)$ if it happens that \emph{both} appear (that is,~$([\rho],m)$ appears in~$\IJord(\pi)$ with multiplicity two); otherwise, we are left with an ambiguity. (See~\ref{rem:bothappear}~Remark for more on this.)

In Section~\ref{JBN3new}, we explore this exceptional case on the Galois side -- that is, we look at the self-dual irreducible representations of~$W_F$ which have the same parity as their self-dual unramified twist. It turns out that they have quite a special structure and that one can determine their parity (see~\ref{prop:indparity}~Proposition). This also translates to a criterion for determining the parity of a self-dual cuspidal representation~$\rho$ (such that~$\rho$ and its self-dual unramified twist~$\rho'$ have the same parity), in terms of the type it contains (see paragraph~\ref{3.11new}).

It is also possible, at least in certain cases, to be more precise in the analysis of the category equivalences and reducibility, in order to elucidate the ambiguity and recover~$\Jord(\pi)$ completely. We hope to come back to this in the case of~$\Sp_4(F)$, in a sequel to this paper.

\subsection*{A remark on characteristic.} 
The bulk of our work is on the representation theory of symplectic groups; for this, while we require that the residual characteristic be odd, we have no further conditions on the characteristic -- that is, we do \emph{not} require~$F$ to be of characteristic zero. In particular, our description of the inertial Jordan set in~\ref{thm:simple}~Theorem and~\ref{thm:reduction}~Theorem does not require characteristic zero. It is only when interpreting these results in terms of the Langlands correspondence (or the endoscopic transfer map) where, until these results have been proved with~$F$ of positive characteristic, we require characteristic zero.

\subsection*{Structure of the paper.} 
In Section~\ref{JBN1new}, we recall the basic structure of types for cuspidal representations, in particular semisimple characters and beta-extensions, including the choice of a base point for beta-extensions. Section~\ref{JBN2new} contains the statements of the main results on (inertial) Jordan sets, remaining entirely on the automorphic side, while the following three sections are devoted to their proofs: in Section~\ref{JBN4new}, we recall the theory of covers and the results of~\cite{Bl,MS} on their Hecke algebras and reducibility of parabolic induction; in Section~\ref{JBN5new} we prove the reduction to the simple case which is at the heart of our method; and in Section~\ref{JBN6new} we prove the result in the simple case. The exploration of self-dual irreducible representations of~$W_F$ is given in Section~\ref{JBN3new} and finally, in Section~\ref{JBN7new}, we interpret our results via the local Langlands correspondence.

\section*{Notation}
Throughout the paper,~$F$ will be a locally compact nonarchimedean local field, with ring of integers~$\oF$, maximal ideal~$\pF$, and residue field~$k_F=\oF/\pF$ of cardinality~$q=q_F$ and \emph{odd} characteristic~$p$; similar notation will be used for extensions of~$F$. The absolute value~$|\cdot|_F$ on~$F$ is normalized to have image~$q^{\mathbb Z}$ and we write~$\nu$ for the character~$g \mapsto |\det g|_F$ of~$\GL_n(F)$.

All representations we consider here will be smooth and complex. By a \emph{cuspidal} representation of the group of rational points of a connected reductive group over~$F$, we mean a representation which is smooth, irreducible and cuspidal (i.e. killed by all proper Jacquet functors).

\section{Cuspidal types and primary beta-extensions}\label{JBN1new}
 
In this section we fix notation following mostly~\cite{S5}. We recall, in the first paragraphs, the main features of the construction of cuspidal representations of symplectic groups achieved in~\cite{S5}, to which we refer for relevant definitions. We do not give references for the by now classical definitions and constructions previously made for linear groups by Bushnell and Kutzko. One of the key steps in the construction is the existence of a so-called \emph{beta-extension}. We will have to compare such beta-extensions across different groups but, unfortunately, they are not uniquely defined. Here, following~\cite{BH}, we explain one way of picking out a particular beta-extension (which we call \emph{$p$-primary}, see~\ref{1.8new}~Definition) in each case, giving a base point to make comparisons.

\subsection{}\label{1.1new} 
In this first paragraph, we recall the notation for skew semisimple strata and related objects. Let~$V$ be a finite dimensional symplectic space over~$F$ of dimension~$2N$. We denote by~$h$ the symplectic form on~$V$, by~$x \mapsto \bar x$ the corresponding adjoint (anti-)involution on~$\End_F(V)$ and by~$\inv$ the corresponding involution on~$\GL_F(V)$. We put~$G=\Sp_F(V)\simeq\Sp_{2N}(F)$, the isometry group of~$h$, which is the group of fixed points of~$\inv$ in~$\GL_F(V)$.

Let~$[\L, n, 0, \beta]$ be a skew semisimple stratum in~$\End_F(V)$~\cite[Definition 2.4, 2.5]{S5}. In particular~$\L$ is a self-dual~$\oF$-lattice sequence and~$\beta=-\bar \beta$ belongs to the Lie algebra~$\sp_F(V)$. We write~$B$ for the commuting algebra of~$\beta$ in~$\End_F(V)$. 

\begin{Remark}\label{rem:normalizesequence}
Following~\cite{S5} we always normalize self-dual lattice sequences such that their period over any relevant field is even and their duality invariant~$d$ is~$1$. With this convention, for any self-dual lattice sequence~$\Lambda$ and any multiple~$s$ of the period~$e$ of~$\L$, there is a unique self-dual lattice sequence of period~$s$ having the form~$t \mapsto \L(\frac{t + a}{s/e})$. There is thus a well defined way of summing two self-dual lattice sequences, by first transforming both into having the same period (see~\cite{BK2}). When performing such transformations, the valuation~$n$ of~$\beta$ relative to the lattice sequence~$\L$ undergoes changes that are of no importance to us, since the associated groups~$H^1,J^1,J$ and characters (see paragraphs~\ref{1.2new},~\ref{1.4new} below) are left unchanged; we will thus ignore this parameter and write the stratum in the form~$[\L, -, 0, \beta]$. 
\end{Remark}

The characteristic spaces of~$\beta$ determine a canonical orthogonal splitting~$V = \perp_{i=0}^l V^i$ for the stratum~$[\L, -, 0, \beta]$ such that, letting~$\L^i = \L \cap V^i$ (that is,~$\L^i(t) = \L(t) \cap V^i$ for any~$t \in \mathbb Z$) and~$\beta^i = \beta_{|V^i}$, the strata~$[\L^i, -, 0, \beta^i]$,~$0\le i \le l$, are skew simple strata which are ``sufficiently distant'' in the sense of~\cite[Definition 2.4]{S5}. We put~$E=F[\beta]= \bigoplus_{i=1}^l E^i$, where~$E^i=F[\beta^i]$, and write~$\oEi$ for the ring of integers of~$E^i$. We recall that~$\L$ is an~\emph{$\oE$-lattice sequence}, by which we mean that each~$\L^i$ is an~$\oEi$-lattice sequence in~$V^i$. 

\begin{Convention}\label{conv}
In this paper we also take the convention that, for any skew semisimple stratum~$[\L, n, 0, \beta]$ with splitting~$V = \perp_{i=0}^l V^i$, we have~$\beta^0=0$. When~$0$ is not an eigenvalue of~$\beta$, this can be achieved by taking~$V^0$ to be the zero-dimensional space over~$F$; since, in that case,~$\dim_F V^0=0$, it does not affect any of the following constructions. The reason for this convention will become apparent later.
\end{Convention}

\subsection{}\label{1.2new} 
From the datum~$[\L, -, 0, \beta]$ are built open compact subrings: 
\begin{itemize}
\item $\tilde{\mathfrak H}^1( \beta, \L ) \subseteq \tilde{\mathfrak J}^1(\beta, \L)$ of~$\End_F(V)$, 
\item $\mathfrak H^1( \beta, \L ) \subseteq \mathfrak J^1(\beta, \L)$ of~$\sp_F(V)$, the fixed points of the former ones under the adjoint involution on~$\End_F(V)$; 
\end{itemize}
and open compact subgroups: 
\begin{itemize}
\item $\tilde H^1(\beta, \L) \subseteq \tilde J^1(\beta, \L) \subset \tilde J(\beta, \L)$ of~$\GL_F(V)$,
\item $H^1(\beta, \L) \subseteq J^1(\beta, \L) \subset J(\beta, \L)$ of~$\SpFV$, the subgroups of fixed points of the former ones under the adjoint involution on~$\GL_F(V)$. 
\end{itemize}
We will frequently write~$H^1_\L= H^1( \beta, \L )$ and so on. 

\subsection{}\label{1.3new} 
We introduce more notation relative to~$\L$. For~$n \in \mathbb Z$ we write: 
\[
\mathfrak a_n(\L)= \{ x \in \End_F(V) \mid \forall t \in \mathbb Z , \ \ x\L(t) \subseteq \L(t+n)\}, \qquad\qquad \mathfrak b_n(\L) = \mathfrak a_n(\L) \cap B . 
\]
In particular~$\mathfrak a_0(\L)$ is a hereditary~$\oF$-order in~$\End_F(V)$ with Jacobson radical~$\mathfrak a_1(\L)$. Let~$\tilde P ( \L ) = \mathfrak a_0(\L)^\times$ and~$\tilde P_1 ( \L) = 1 + \mathfrak a_1(\L)$. Then~$P_1 ( \L) = \tilde P_1 ( \L) \cap \SpFV$ is the pro-$p$-radical of~$P ( \L) = \tilde P ( \L) \cap \SpFV$. The quotient groups 
\[
\tilde\HP(\L) = \tilde P(\L)/\tilde P_1(\L) \quad\text{ and }\quad \HP(\L) = P(\L)/P_1(\L) 
\]
are (the groups of rational points of) finite reductive groups over~$k_F$. The latter may be disconnected so we let~$\HP^0 ( \L )$ be (the group of rational points of) its neutral component and call~$P^0(\L)$ the inverse image of~$\HP^0 ( \L )$ in~$P ( \L )$; this is a parahoric subgroup of~$\SpFV$. 

Actually we will mainly work with~$\mathfrak b_0(\L) = \mathfrak a_0(\L) \cap B$ and with~$P ( \L_\oE) := P ( \L) \cap B$, with~$\HP ( \L_\oE ) = P ( \L_\oE )/P_1 ( \L_\oE)$ and its neutral component~$\HP^0 ( \L_\oE )$, and with the parahoric subgroup~$P^0 ( \L_\oE )$ of~$G_E:= B \cap \SpFV$, inverse image of~$\HP^0 ( \L_\oE )$ in~$P ( \L_\oE )$. Indeed we have the following: 
\[ 
J(\beta, \L) = P ( \L_\oE) J^1(\beta, \L) \quad \text{ and } \quad 
J(\beta, \L) / J^1(\beta, \L) \simeq P ( \L_\oE )/P_1 ( \L_\oE) = \HP ( \L_\oE ) . 
\]
Moreover, we have natural isomorphisms
\[
\HP(\L_\oE) \simeq \prod_{i=0}^l \HP(\L^i_{\oEi}) 
\qquad\text{and}\qquad
\HP^0(\L_\oE) \simeq \prod_{i=0}^l \HP^0(\L^i_{\oEi}).
\]
Note that, writing~$E^i_\so$ for the field of fixed points of~$E^i$ under the adjoint involution~$x\mapsto\overline x$ and~$k^i_\so$ for its residue field, the groups on the right hand side here are reductive groups over~$k^i_\so$. We also have similar decompositions and isomorphisms for the group~$\tilde J(\beta,\L)$.

\subsection{}\label{1.4new} 
On the group~$\tilde H^1(\beta, \L)$ lives a family of one-dimensional representations endowed with very strong properties, called \emph{semisimple characters} ({\it loc. cit. \S3.1}), that restricts to a family of \emph{skew semisimple characters} on~$H^1(\beta, \L)$. In particular, a skew semisimple character of~$H^1(\beta, \L)$, say~$\theta$, restricts to a skew \emph{simple} character~$\theta_i$ of~$H^1(\beta, \L) \cap \Sp_F(V^i) = H^1(\beta^i, \L^i)$, for~$0\le i \le l$. Among the properties of these families the ``transfer property'' is specially important. It asserts that if~$[\L^\prime, -, 0, \beta]$ is another skew semisimple stratum in~$\End_F(V)$, then there is a canonical bijection between the sets of skew semisimple characters on~$H^1(\beta, \L)$ and~$H^1(\beta, \L^\prime)$ ({\it loc. cit. Proposition 3.2}). The image of~$\theta$ under this bijection is called the \emph{transfer} of~$\theta$. 

To any semisimple character~$\tilde\theta$ of~$\tilde H^1(\beta, \L)$ is associated the unique (up to equivalence) irreducible representation~$\tilde \eta$ of~$\tilde J^1(\beta, \L)$ that contains~$\tilde \theta$ upon restriction, actually~$\tilde \eta$ restricts to a multiple of~$\tilde\theta$ on~$\tilde H^1(\beta, \L)$. Now~$\tilde H^1_\L$ and~$\tilde J^1_\L$ are pro-$p$-groups with~$p$ odd, on which the adjoint involution~$\inv$ acts. The Glauberman correspondence hence relates their representations to those of the fixed point subgroups~$H^1_\L$ and~$J^1_\L$. Indeed if~$\tilde\theta$ is fixed under the involution~$\inv$ so is~$\tilde \eta$ and its image~$\eta$ under the Glauberman correspondence is the unique (up to equivalence) irreducible representation of~$J^1(\beta, \L)$ that contains~$\theta$; it actually restricts to a multiple of~$\theta$ on~$H^1(\beta, \L)$. 

\subsection{}\label{1.5new} 
In turn the representation~$\tilde \eta$ has special extensions to~$\tilde J (\beta, \L)$ called~\emph{beta-extensions} and denoted by~$\tilde \kappa$. These beta-extensions in~$\GL_F(V)$ are characterized by the fact that they are intertwined by~$B^\times$ (\cite[(5.2.1)]{BK}). 

\begin{Remark}
In the literature, these extensions are usually called~$\beta$-extensions. However, the simple stratum~$[\L,-,0,\beta]$ giving rise to a particular simple character~$\theta$ is not unique, while the notion of beta-extension turns out to be independent of the choice of~$\beta$. It is thus convenient to write \emph{beta-extension}, especially since we also have strata indexed by~$i$ so we would otherwise need to talk about~$\beta_i$-extensions etc. 
\end{Remark}

The definition of beta-extensions in classical groups is more delicate~\cite[\S4]{S5}. A skew semisimple stratum as above is called \emph{maximal} if~$\mathfrak b_0(\L)$ is a maximal self-dual~$\oE$-order in~$B$. If~$[\L, -, 0, \beta]$ is a maximal skew semisimple stratum, a beta-extension of~$\eta$ is an extension~$\kappa$ of~$\eta$ to~$J (\beta, \L)$ such that the restriction of~$\kappa$ to any pro-$p$-Sylow is intertwined by~$G_E$ (\cite[3.11,~4.1]{S5}). In the general case, the notion of beta-extension is a relative one. Given a maximal skew semisimple stratum~$[\fM, -, 0, \beta]$ in~$\End_F(V)$ such that~$\mathfrak b_0(\fM) \supset \mathfrak b_0(\L)$, given the transfer~$\theta_{\fM}$ of~$\theta$ to~$H^1_{\fM}$ and the representation~$\eta_{\fM}$ of~$J^1_{\fM}$ determined by~$\theta_{\fM}$, there is a canonical way to associate to a beta-extension~$\kappa_{\fM}$ of~$\eta_{\fM}$, an extension~$\kappa$ of~$\eta$, called \emph{the beta-extension of~$\eta$ to~$J_\L$ relative to~$\fM$, compatible with~$\kappa_{\fM}$} \cite[4.3,~4.5]{S5}. (We can also call~$\kappa$ a beta-extension of~$\theta$.) 

Note that the groups~$\tilde J (\beta, \L)$ and~$J (\beta, \L)$ are not pro-$p$-groups: the notation~$\kappa$ here should not call to mind a Glauberman-like connection with the former~$\tilde \kappa$. 

\subsection{}\label{1.6new} 
Let~$J= J (\beta, \L)$, for a skew semisimple stratum~$[\L, -, 0, \beta]$ as above, let~$\theta$ be a skew semisimple character of~$H^1 (\beta, \L)$ and let~$\lambda$ be an irreducible representation of~$J$ of the form~$\lambda = \kappa \otimes \tau$, with~$\kappa$ some beta-extension of~$\theta$, and~$\tau$ the inflation of a cuspidal representation of~$J/J^1 \simeq \HP ( \L_\oE )$. Under the additional assumptions that the group~$G_E$ has compact centre and that~$P^0 ( \L_\oE )$ is a maximal parahoric subgroup of~$G_E$, the pair~$(J, \lambda)$ is called a \emph{cuspidal type} for~$\SpFV$. Recall from~\cite{S5} (see also~\cite{MS} for complements): 

 \begin{Theorem}[{\cite[Corollary 6.19, Theorem 7.14]{S5}}]\label{thm:cusptypeG}
A cuspidal type in~$\SpFV$ induces to a cuspidal representation of~$\SpFV$ and any cuspidal representation of~$\SpFV$ is thus obtained. 
\end{Theorem} 

\subsection{}\label{1.7new} 
There is of course a similar result for the group~$\GL_F(V)$. Here we let~$\tilde J=\tilde J(\beta,\L)$ for a \emph{simple} stratum~$[\L, -, 0, \beta]$ (so that~$E=F[\beta]$ is a field) and let~$\tilde\lambda$ be an irreducible representation of~$\tilde J$ of the form~$\tilde\lambda = \tilde\kappa \otimes \tilde\tau$, with~$\tilde\kappa$ some beta-extension of~$\tilde\theta$, and~$\tilde\tau$ the inflation of a cuspidal representation of~$\tilde J/\tilde J^1$. Under the additional assumptions that~$\tilde P(\L)\cap B$ is a maximal parahoric subgroup of~$B^\times$, the pair~$(\tilde J, \tilde\lambda)$ is called a \emph{maximal simple type} for~$\GL_F(V)$.

\begin{Theorem}[{\cite[5.5.10, 6.2.4, 8.4.1]{BK}}]\label{thm:cusptypeGL}
A maximal simple type in~$\GL_F(V)$ extends to an irreducible representation of its normaliser, which then induces to a cuspidal representation of~$\GL_F(V)$; any cuspidal representation~$\rho$ of~$\GL_F(V)$ is thus obtained and the maximal simple type yielding~$\rho$ is unique up to conjugacy in~$\GL_F(V)$.
\end{Theorem} 

\begin{Remark}
This theorem includes depth zero representations, by formally considering the null stratum~$[\L,-,0,0]$ to be simple.
\end{Remark}

\subsection{}\label{1.8new}
In order to compare representations across different groups, we need a way to compare beta-extensions. (The transfer of semisimple characters already allows a comparison.) Two beta-extensions only differ by a character (of a specific shape), however we will need to choose beta-extensions in a unique way as in~\cite[\S2.3~Lemma~1]{BH},which amounts to the~$\GL$-case in the following lemma. 

\begin{Lemma}\label{powerofp}
\begin{enumerate}
\item\label{powerofp.i} Let~$[\L, n, 0, \beta]$ be a simple stratum in~$\End_F(V)$, let~$\tilde\theta$ be a simple character of~$\tilde H^1(\beta, \L)$, and let~$\tilde \eta$ be the irreducible representation of~$\tilde J^1(\beta, \L)$ containing~$\tilde \theta$. There exists one and only one beta-extension~$\tilde \kappa$ of~$\tilde \eta$ to $\tilde J(\beta, \L)$ whose determinant has order a power of~$p$. 
\item\label{powerofp.ii} With notation as in~\ref{powerofp.i}, assume the stratum and the simple character are skew so that the involution~$\inv$ on~$\GL_F(V)$ stabilizes~$\tilde H^1(\beta, \L)$, $\tilde J^1(\beta, \L)$, $\tilde J(\beta, \L)$ and~$\tilde \theta$. The beta-extension~$\tilde \kappa$ in~\ref{powerofp.i} satisfies~$\tilde \kappa \simeq \tilde \kappa \circ \inv$. 
\item\label{powerofp.iii} Let~$[\L, n, 0, \beta]$ be a maximal skew semisimple stratum in~$\End_F(V)$, let~$\theta$ be a skew semisimple character of~$H^1(\beta, \Lambda)$, and let~$\eta$ be the irreducible representation of~$J^1(\beta, \Lambda)$ containing~$\theta$. There exists one and only one beta-extension of~$\eta$ to~$J( \beta, \L )$ whose determinant has order a power of~$p$. 
\end{enumerate}
\end{Lemma}

\begin{proof} 
\ref{powerofp.i} The reference is \cite[(5.2.2)]{BK} which we imitate below to conclude the proof of~\ref{powerofp.iii}. 

In~\ref{powerofp.ii}, self-duality with respect to~$\inv$ follows from uniqueness. Indeed~$\tilde \eta \circ \inv$ is equivalent to~$\eta$ so there is an intertwining operator~$T$ such that~$\tilde \eta (x) =T \left( \tilde \eta \circ \inv (x) \right) T^{-1}$, for~$x \in \tilde J^1( \beta, \L)$. Since $\inv$ stabilizes $\GL_{F[\b]}(V)$, the representation~$T ( \tilde \kappa \circ \inv(x) ) T^{-1}$, for~$x \in \tilde J( \beta, \L)$, is a beta-extension of~$\tilde \eta$ by~\cite[Definition 5.2.1]{BK}; its determinant is a power of~$p$, so it is equal to~$\tilde \kappa$. 

\ref{powerofp.iii} Let~$\kappa$ be a beta-extension of~$\eta$ and let~$\Phi = \det (\kappa_{|P(\L_{\oE})})$. The main point is to prove that the character~$\Phi$ of~$P(\L_{\oE})$ factors through the determinant~$\det_E$. By this we mean, as usual, that~$\Phi_{|P(\L_{\oEi})}$ factors through~$\det_{E^i}$, for~$0 \le i \le l$; the remainder of the proof uses this convention. 

Since~$\theta$ is equal to~$\chi\circ \det_E$ on~$P^1(\L_{\oE})$, for some character~$\chi$ of~$1 + \pE$, we have that~$\kappa_{|P^1(\L_{\oE})}$ is the sum of~$\dim(\eta)$ copies of~$\chi\circ \det_E$. Now~$\chi$ extends to a character~$\tilde \chi $ of~$\oEx$ and~$\Phi' = \left(\tilde \chi\circ \det_E\right)^{-\dim(\eta)} \Phi$ is then a character of~$P(\L_{\oE})/P^1(\L_{\oE})$. From~\cite[Lemma~3.10, Corollary~3.11 and Theorem~4.1]{S5}, the character~$\Phi'$ is trivial on all~$p$-Sylow subgroups of~$P(\L_{\oE})/P^1(\L_{\oE})$ so factors as~$\Phi'= \psi \circ \det_E$ where~$\psi $ is a character of~$\oEx $ trivial on~$1+ \pE$ (and depends on the choice of extension~$\tilde\chi$). 

Let us write~$\oEx= \mu'_{E}(1+\pE)$, where~$\mu'_{E}$ is the group of roots of unity in~$E^\times$ of order prime to~$p$, and, in the above, let us choose~$\tilde \chi$ trivial on~$\mu'_{E}$ so that the order of~$\tilde \chi$ is a power of~$p$. The corresponding character~$\psi$ has order prime to~$p$, so prime to~$\dim(\eta)$, and there is a character~$\alpha$ of~$\oEx$ (trivial on~$1+\pE$) such that~$\psi = \alpha^{\dim(\eta)}$. 

The representation~$\underline\kappa = \left(\alpha\circ \det_E\right)^{-1} \kappa$ satisfies the required condition. It is unique since any other beta-extension has the form~$(\psi \circ \det_E ) \underline\kappa$, with~$\psi$ as above, and if~$\psi $ is non-trivial then no~$p^i$-th power of~$\psi$ can be trivial. 
\end{proof}

\begin{Definition}
With the notations of~\ref{powerofp.i} above, we denote by $\underline{\tilde \kappa}$ the unique beta-extension of~$\tilde \eta$ whose determinant has order a power of~$p$. We call~$\underline{\tilde \kappa}$ the \emph{$p$-primary beta-extension} of~$\tilde \eta$. 

With the notations of~\ref{powerofp.iii} above, we denote by~$\underline \kappa$ the unique beta-extension of~$\eta$ whose determinant has order a power of~$p$. We call~$\underline \kappa$ the \emph{$p$-primary beta-extension} of~$\eta$. 
\end{Definition}

We remark that, while the~$p$-primary beta-extensions give a useful way of picking a base point amongst the beta-extensions, sufficient for our needs here, it is not clear whether this is the best choice of base point.

\section{Inertial Jordan blocks}\label{JBN2new}

In this section, we state the main results on Jordan blocks and the consequences for the endoscopic transfer map. We continue with the notation from the previous section.

\subsection{}\label{2.1new}
Let~$\pi$ be a cuspidal representation of~$\SpFV\simeq\Sp_{2N}(F)$. We recall the \emph{reducibility set}~$\Red (\pi)$ and the \emph{Jordan set}~$\Jord(\pi)$ from the introduction. For any positive integer~$n$, the group~$\GL_n(F) \times \SpFV$ appears naturally as a standard maximal Levi subgroup of~$\Sp_{2(N+n)}(F)$. If~$\rho$ is a cuspidal representation of~$\GL_n(F)$ we can form the normalized parabolically induced representation~$\rho \nu^s \rtimes \pi$ (we use normalized induction and induce via the standard parabolic), where~$s$ is here a \emph{real} parameter and~$\nu$ is the character~$g \mapsto |\det g|_F$ of~$\GL_n(F)$.  If no unramified twist of~$\rho$ is  self-dual (i.e. isomorphic to its contragredient) then~$\rho \nu^s \rtimes \pi$ is always irreducible. On the other hand, if~$\rho$ is self-dual, there is a unique~$s_\pi(\rho)\ge 0$ such that~$\rho \nu^s \rtimes \pi$ is reducible if and only if~$s = \pm s_\pi(\rho)$.

\begin{Definition}
Let~$\pi$ be a cuspidal representation of~$\SpFV$.
\begin{itemize}
\item The \emph{reducibility set}~$\Red (\pi)$ is the set of isomorphism classes of self-dual cuspidal representations~$\rho$ of some~$\GL_n(F)$, with~$n \ge 1$, for which~$s_\pi(\rho)\ge 1$. 
\item The \emph{Jordan set}~$\Jord(\pi)$ is the set of pairs~$(\rho, m)$, where~$\rho \in \Red(\pi)$ and~$m$ is a positive integer such that~$2s_\pi(\rho)-1-m$ is a non-negative even integer.
\end{itemize}
\end{Definition}

Note that, if~$\rho\in\Red(\pi)$ then~$2s_\pi(\rho)-1$ is a positive integer by~\cite{MoTa}, so that there is a positive integer~$m$ such that~$(\rho,m)\in\Jord(\pi)$.

\subsection{}\label{2.2new}
For~$\rho$ an irreducible representation of some~$\GL_n(F)$, we write~$n=\deg(\rho)$. Recall that the \emph{inertial class}~$[\rho]$ of a cuspidal representation~$\rho$ of~$\GL_n(F)$ is the equivalence class of~$\rho$ under the equivalence relation defined by twisting by an unramified character (that is, twisting by~$\omega \circ \det$ where~$\omega$ is a character of~$F^\times$ trivial on~$\oFx$). If~$\rho$ is self-dual then the inertial class~$[\rho]$ contains precisely two self-dual representations: if~$t(\rho)$ denotes the number of unramified characters~$\chi$ of~$\GL_n(F)$ such that~$\rho\otimes\chi\simeq\rho$, and if~$\chi'$ is an unramified character of order~$2t(\rho)$, then~$\rho'=\rho\otimes\chi'$ is the other self-dual representation in~$[\rho]$.

\begin{Definition}
Let~$\pi$ be a cuspidal representation of~$\SpFV$. The \emph{inertial Jordan set} of~$\pi$ is the \emph{multiset}~$\IJord(\pi)$~consisting of all pairs~$([\rho], m)$ with~$(\rho, m) \in \Jord(\pi)$. 
\end{Definition}

Note that, if~$([\rho],m)\in\IJord(\pi)$, with~$\rho$ a self-dual cuspidal representation of~$\GL_n(F)$, then either~$(\rho,m)\in\Jord(\pi)$ or~$(\rho',m)\in\Jord(\pi)$, where~$\rho'$ as above is the second self-dual representation in the inertial class~$[\rho]$. As discussed in the introduction, if one of~$\rho,\rho'$ is of symplectic type and the other of orthogonal type, then which occurs in~$\Jord(\pi)$ is determined by the parity of~$m$. On the other hand, if~$\rho,\rho'$ are both of the same parity then the inertial Jordan set~$\IJord(\pi)$ does \emph{not} distinguish them; of course, if~$([\rho],m)$ occurs with multiplicity two in~$\IJord(\pi)$, then both~$(\rho,m)$ and~$(\rho',m)$ occur in~$\Jord(\pi)$ and there is no ambiguity; see~\ref{rem:bothappear}~Remark for more on this.

\subsection{}\label{2.3new}
In order to refine further the (inertial) Jordan set, we need to use the notion of the \emph{endo-class} of a simple character (for linear groups), as defined in~\cite{BH1}. To any cuspidal representation~$\rho$ of $\GL_n(F)$ is attached in~\cite[\S1.4]{BH2} an endo-class of simple characters, denoted by~$\bs\Theta(\rho)$, as follows. As recalled in~\ref{thm:cusptypeGL}~Theorem, there is a maximal simple type~$(\tilde J, \tilde \lambda)$ in~$\GL_n(F)$ which occurs in~$\rho$ and~$\rho$ determines the~$\GL_n(F)$-conjugacy class of~$(\tilde J, \tilde \lambda)$. This maximal simple type is built from a simple character~$\tilde\theta$ and we define $\bs\Theta(\rho)$ to be the endo-class of $\tilde \theta$. (In fact, this is also the endo-class of \emph{any} simple character contained in~$\rho$.) Note that we are allowing here the case of depth zero representations (where~$\rho$ contains the trivial character of~$\tilde P^1(\L)$, for some lattice sequence~$\L$), in which case~$\bs\Theta(\rho)= \bs\Theta_0^F$ is the trivial endo-class over $F$. 

\begin{Definition}
Let~$\pi$ be a cuspidal representation of~$\SpFV$ and let~$\bs\Theta$ be an endo-class of simple characters over~$F$. The \emph{inertial Jordan set of~$\pi$ relative to~$\bs\Theta$} is the \emph{multiset}~$\IJord(\pi,\bs\Theta)$~consisting of all pairs~$([\rho], m)$ with~$(\rho, m) \in \Jord(\pi)$ and~$\bs\Theta(\rho)=\bs\Theta$.
\end{Definition}

\subsection{}\label{2.4new}
We will also need to twist inertial Jordan blocks as follows. With notation as in the previous paragraph, the~$\GL_n(F)$-conjugacy class of~$(\tilde J, \tilde \lambda)$ depends only on the inertial class~$[\rho]$; it also determines~$[\rho]$ by~\cite[(5.5)]{BK1}. The quotient group~$\tilde J/ \tilde J^1$ is a linear group over a finite field, say~$\GL (m_{[\rho]},k_{[\rho]})$. We define the twist of the inertial class~$[\rho]$ by a character~$\chi$ of~$k_{[\rho]}^\times$ to be the inertial class~$[\rho]_\chi$ determined by the maximal simple type~$(\tilde J, \tilde \lambda\otimes \chi\circ \det)$ -- that is, in the decomposition~$\tilde\lambda=\tilde\kappa\otimes\tilde\tau$ with~$\tilde\kappa$ a beta-extension, we replace the cuspidal representation~$\tilde \tau$ by~$\tilde \tau\otimes \chi\circ \det$. 

Let~$\bs\Theta$ be an endo-class of simple characters. By~\cite[Proposition 8.11]{BH1}, it determines a finite extension~$k_{\bs\Theta}$ of~$k_F$ such that, for any cuspidal representation~$\rho$ of some $\GL_n(F)$ satisfying~$\bs\Theta(\rho) = \bs\Theta$, if~$(\tilde J,\tilde\lambda)$ is a maximal simple type in~$\rho$ then the quotient group~$\tilde J/ \tilde J^1$ is a linear group over~$k_{\bs\Theta}$ (that is,~$k_{[\rho]}=k_{\bs\Theta}$ in the notation above). It is thus meaningful to give the following definition:

\begin{Definition}
Let~$\pi$ be a cuspidal representation of~$\SpFV$, let~$\bs\Theta$ be an endo-class of simple characters, and let~$\chi$ be a character of $k_{\bs\Theta}^\times$. The \emph{$\chi$-twisted inertial Jordan set of $\pi$ relative to $\bs\Theta$} is the multiset~$\IJord(\pi,\bs\Theta)_\chi $ consisting of all pairs~$([\rho]_\chi, m)$ with~$(\rho, m) \in \Jord(\pi)$ and~$\bs\Theta(\rho)=\bs\Theta$. 
\end{Definition}
The relevant case for us will be the case where~$\chi$ is quadratic or trivial. 

\begin{Remark}
Since~$p$ is odd, we have a squaring map~$\bs\Theta\mapsto\bs\Theta^2$ on endo-classes: if~$\theta$ is a simple character with endo-class~$\bs\Theta$, associated to a simple stratum~$[\Lambda,-,0,\beta]$, then the character~$\theta^2$ is a simple character for the stratum~$[\Lambda,-,0,2\beta]$ and~$\bs\Theta^2$ is the endo-class corresponding to~$\theta^2$. This is well-defined and moreover gives a bijection on the set of endo-classes (again, since~$p$ is odd). We note also that the fields~$k_{\bs\Theta}$ and~$k_{\bs\Theta^2}$ coincide.
\end{Remark}

\subsection{}\label{2.6new}
We begin the computation of the inertial Jordan set with a special case, to which we will reduce in the next paragraph. We call a cuspidal representation of~$\SpFV$ \emph{simple} if it contains a simple character; that is, it contains a semisimple character~$\theta$ of~$H^1(\beta,\Lambda)$, associated to a skew semisimple stratum~$[\Lambda,-,0,\beta]$, such that~$E=F[\beta]$ is a field. We allow the degenerate case~$\beta=0$, in which case~$\pi$ is of depth zero (and every depth zero representation is simple with~$\beta=0$); we also allow, in the case~$\beta=0$, the degenerate case that~$G$ is the trivial group, so that the trivial representation of the trivial group is regarded as being simple of depth zero.

\begin{Remark} 
Our use of the word \emph{simple} here is consistent with, but not the same as, the use in~\cite{GR} where, for symplectic groups, it means of minimal positive depth~$1/2N$. More precisely, all cuspidal representations of depth~$1/2N$ are simple in our sense, but the converse is false.
\end{Remark}

The following theorem tells us that, in the case of simple cuspidals, the Jordan set is filled by representations with the expected endo-class. 

\begin{Theorem}\label{thm:simple}
Let~$\pi$ be a simple cuspidal representation of~$\SpFV$ and let~$\tilde\theta$ be a self-dual simple character whose restriction to~$\SpFV$ is contained in~$\pi$. Let~$\bs\Theta$ be the endo-class of the simple character~$\tilde\theta$. Then
\[
\sum_{([\rho],m)\in\IJord(\pi,\Theta^2)} m\deg(\rho) = \begin{cases}
2N+1 &\text{ if~$\bs\Theta=\bs\Theta_0^F$, the trivial endo-class;} \\[3pt]
2N &\text{ otherwise.} \end{cases}
\]
\end{Theorem}

Note that we have~$\bs\Theta=\bs\Theta_0^F$ if and only if~$\pi$ is of depth zero (which includes the degenerate case~$N=0$ where~$G$ is the trivial group). In this case, the theorem is a special case of the main result of~\cite{LS}. We will prove this theorem in Section~\ref{JBN6new} by computing the real parts of the complex reducibility points of parabolically induced representations of the form~$\rho\nu^s\rtimes\pi$, with~$\rho$ a self-dual cuspidal representation of some general linear group with endo-class~$\bs\Theta^2$, using the theory of types and covers to reduce the calculation to computations of Lusztig for finite reductive groups. We note also that the proof not only gives the equality above but also gives an algorithm to compute the multiset~$\IJord(\pi,\Theta^2)$  (see paragraph~\ref{6.8new} for more detail).

\subsection{}\label{2.5new}
Now let~$\pi$ be an arbitrary cuspidal representation of~$\SpFV$. Recall from~\ref{thm:cusptypeG}~Theorem that~$\pi$ can be constructed by induction, starting with a maximal skew semisimple stratum~$[\L,-,0,\beta]$ and a skew semisimple character~$\theta$ of~$H^1(\beta,\L)$, which decomposes into a family of skew simple characters~$\theta_i$ of~$H^1(\beta^i,\L^i)$, for $i\in \{0, \ldots, l\}$. Let~$\underline\kappa$ be the~$p$-primary beta-extension of~$\theta$~to $J_{\L}$ and, similarly, let~$\underline\kappa_i$~be the~$p$-primary beta-extension of~$\theta_i$ to~$J_{\L^i}$ (in~$\Sp_F(V^i)$), for $0\le i \le l$.

Let~$\tau$ be the cuspidal representation of~$\HP(\Lambda_\oE) = P(\Lambda_\oE)/ P^1(\Lambda_\oE)$ such that~$\pi$ is induced from~$\lambda = \underline\kappa \otimes \tau$. Then we can uniquely decompose~$\tau$ as~$\tau=\otimes_{i=0}^l \tau_i$, with~$\tau_i$ an irreducible (cuspidal) representation of $\HP(\Lambda_{\oEi})$. We may then define, for each~$i$, the cuspidal representation~$\pi_i$ of~$\Sp_F(V^i)$ by
\[
\pi_i = \cInd_{J_{\L^i}}^{\text{Sp}_F(V^i)} \underline\kappa_i \otimes \tau_i .
\]
Note that this representation is simple, in the sense of the previous paragraph.

\begin{Remark}
Recall that we are using the notation of section~\ref{1.1new}, in particular~\ref{conv}~Convention so that we are assuming~$\beta^0=0$. If the space~$V^0$ is trivial then the representation~$\pi_0$ is the trivial representation of the trivial group.
\end{Remark}

We can now state the crucial reduction theorem, which allows us to determine the inertial Jordan set of~$\pi$ from those of the simple cuspidals~$\pi_i$.  
 
\begin{Theorem}\label{thm:reduction}
With notation as above, for~$0\le i \le l$, let~$\tilde\theta_i$ be the unique self-dual simple character of~$\tilde H^1(\beta^i,\L^i)$ restricting to~$\theta_i$ on~$H^1(\beta^i,\L^i)$. Let~$\bs\Theta_i$ be the endo-class of the simple character~$\tilde\theta_i$ and let~$k_{\bs\Theta_i}$ be the corresponding extension of~$k_F$. Then there is a character~$\chi_i$ of~$k_{\bs\Theta_i}^\times$ of order at most two such that we have an equality of multisets
\[
\IJord(\pi,\bs\Theta_i^2) = \IJord(\pi_i,\bs\Theta_i^2)_{\chi_i}.
\] 
\end{Theorem}

The character~$\chi_i$ appearing here is in some sense explicit, coming from certain permutation characters (see~\ref{thm:reductiondetail}~Theorem,~\ref{prop:mainprop}~Proposition and~\ref{prop:Phi}~Proposition for more details). The proof of the theorem will be given in Section~\ref{JBN5new}, following preparation in Section~\ref{JBN4new} (which is also needed for the proof of~\ref{thm:simple}~Theorem). Again, the principle is to use the theory of types and covers to compare the real parts of the complex reducibility points of~$\rho\nu^s\rtimes\pi$ with those of~$\rho_i\nu^s\rtimes\pi_i$, for~$\rho$ a self-dual cuspidal representation of some general linear group with endo-class~$\bs\Theta_i^2$ and~$\rho_i$ self-dual in the inertial class~$[\rho]_{\chi_i}$.

For now, we put together the two previous theorems to get:

\begin{Corollary}\label{cor:together}
Suppose~$F$ is of characteristic zero. With the notation of the Theorem, we have
\[
\IJord(\pi) = \bigsqcup_{i=0}^l \IJord(\pi_i,\Theta_i^2)_{\chi_i}.
\]
\end{Corollary}

Since the proof of~\ref{thm:simple}~Theorem gives us an algorithm to compute the multisets~$\IJord(\pi_i,\Theta_i^2)$, we can then use this also to compute~$\IJord(\pi)$, for any cuspidal representation~$\pi$.

\begin{proof}
The Theorem says that~$\IJord(\pi)$ contains the right hand side. On the other hand, by~\cite[Theorem~3.2.1]{Mo} the multiset~$\IJord(\pi)$ is finite and we have
\[
\sum_{([\rho],m)\in\IJord(\pi)} m\deg(\rho) =\sum_{(\rho,m)\in\Jord(\pi)} m\deg(\rho) = 2N+1. 
\]
However, writing~$\dim_F(V^i)=2N_i$, we get from~\ref{thm:simple}~Theorem that
\[
\sum_{i=0}^l\sum_{([\rho],m)\in\IJord(\pi_i,\Theta_i^2)} m\deg(\rho) = (2N_0+1) + \sum_{i=1}^l 2N_i = 2N+1.
\]
Thus we have equality, as required.
\end{proof}

We remark that the proof does not require the full strength of~\cite[Theorem~3.2.1]{Mo}; indeed, it only uses the inequality
\[
\sum_{(\rho,m)\in\Jord(\pi)} m\deg(\rho) \le  2N+1,
\]
which was proved previously in~\cite[\S4~Corollaire]{Mo2}. Thus it does not in fact depend on Arthur's endoscopic classification of discrete series representations of~$\SpFV$. One could also prove it (without the restriction on the characteristic of~$F$) by checking that~$\IJord(\pi,\bs\Theta)$ is empty for any self-dual endo-class~$\bs\Theta\ne\bs\Theta_i^2$; indeed, the methods of Section~\ref{JBN5new} together with results from~\cite{KSS} would allow this.

\subsection{}\label{2.7new}
In this and the following paragraph, we interpret our results in terms of the endoscopic transfer map from cuspidal representations of~$\SpFV$ to~$\GL_{2N+1}(F)$.

For~$\bs\Theta$ an endo-class over~$F$, we recall that the \emph{degree}~$\deg(\bs\Theta)$ of~$\bs\Theta$ is the degree of an extension~$F[\beta]/F$ for which there are a simple stratum~$[\Lambda,-,0,\beta]$ with a simple character of endo-class~$\bs\Theta$. Although the stratum and the field extension are not uniquely determined by~$\bs\Theta$, this degree is (see~\cite[Proposition~8.11]{BH1}). 

Let~$N'$ be a positive integer and write~$\CE(F)$ for the set of endo-classes of simple characters over~$F$. An \emph{endo-parameter of degree~$N'$ over~$F$} is a formal sum
\[
\sum_{\bs\Theta\in\CE(F)} m_{\bs\Theta}\bs\Theta, \qquad m_{\bs\Theta}\in\BZ_{\ge 0},
\]
such that
\[
\sum_{\bs\Theta\in\CE(F)} m_{\bs\Theta}\deg(\bs\Theta) = N'.
\]
In particular, such a formal sum has finite support~$\{\bs\Theta\in\CE(F)\mid m_{\bs\Theta}\ne 0\}$. (In~\cite{SSJL}, these formal sums are called \emph{semisimple endo-classes}; the nomenclature \emph{endo-parameter} comes from~\cite{KSS}.) We write~$\CEE_{N'}(F)$ for the set of endo-parameters of degree~$N'$ over~$F$. We then have, for each positive integer~$N'$, a well-defined map
\[
\mathsf{e}_{N'}:\Irr(\GL_{N'}(F)) \to \CEE_{N'}(F)
\]
given by mapping a cuspidal representation~$\rho$ to~$\frac{N'}{\deg(\bs\Theta(\rho)}\bs\Theta(\rho)$, and mapping an arbitrary representation to the sum of the endo-parameters of its cuspidal support.

\subsection{}\label{2.8new}
We call an endo-class~$\bs\Theta$ over~$F$ \emph{self-dual} if there is a self-dual simple character~$\tilde\theta$ with endo-class~$\bs\Theta$. We write~$\CE^{\mathsf{sd}}(F)$ for the set of self-dual endo-classes over~$F$. An endo-parameter of degree~$N'$ over~$F$ is called \emph{self-dual} if its support is contained in~$\CE^{\mathsf{sd}}(F)$, and we write~$\CEE^{\mathsf{sd}}_{N'}(F)$ for the set of self-dual endo-parameters of degree~$N'$ over~$F$.

Since~$p$ is odd, the only self-dual endo-class over~$F$ of odd degree is the trivial endo-class~$\bs\Theta^F_0$, which has degree~$1$. Indeed, if~$\tilde\theta$ is a self-dual simple character which is not the trivial character, then~\cite[Theorem~6.3]{S1} implies that~$\tilde\theta$ is associated to a skew simple stratum, whose associated field extension~$E/F$ is therefore of even degree. This implies, in particular, that there is a canonical bijection
\[
\CEE^{\mathsf{sd}}_{2N}(F)\to \CEE^{\mathsf{sd}}_{2N+1}(F),\qquad
\sum_{\bs\Theta\in\CE^{\mathsf{sd}}} m_{\bs\Theta}\bs\Theta \mapsto \sum_{\bs\Theta\in\CE^{\mathsf{sd}}} m_{\bs\Theta}\bs\Theta + \bs\Theta^F_0.
\]
For any~$N'$, there is also the natural squaring map
\[
\CEE_{N'}(F)\to \CEE_{N'}(F),\qquad
\sum_{\bs\Theta\in\CE} m_{\bs\Theta}\bs\Theta \mapsto \sum_{\bs\Theta\in\CE} m_{\bs\Theta}\bs\Theta^2,
\]
which is a bijection since~$p$ is odd. Combining these, we get a natural inclusion map
\[
\iota_{2N}:\CEE^{\mathsf{sd}}_{2N}(F)\into \CEE_{2N+1}(F),\qquad
\sum_{\bs\Theta\in\CE^{\mathsf{sd}}} m_{\bs\Theta}\bs\Theta \mapsto \sum_{\bs\Theta\in\CE^{\mathsf{sd}}} m_{\bs\Theta}\bs\Theta^2 + \bs\Theta^F_0.
\]

Given a maximal skew semisimple stratum~$[\L, n, 0, \beta]$ and a skew semisimple character~$\theta$ of~$H^1(\beta,\L)$, which decomposes into a family of skew simple characters~$\theta_i$ of~$H^1(\beta^i,\L^i)$, for $i\in \{0, \ldots, l\}$, we define the self-dual endo-parameter of~$\theta$ to be
\[
\sum_{i=0}^l \frac{\dim_F(V^i)}{\deg(\bs\Theta_i)}\bs\Theta_i,
\]
where~$\bs\Theta_i$ is the endo-class of the unique self-dual simple character~$\tilde\theta_i$ which restricts to~$\theta_i$. This is a self-dual endo-parameter of degree~$2N$.

We write~$\Cusp(G)$ for the set of equivalence classes of cuspidal representations of~$G$. From~\ref{thm:reduction}~Theorem and~\ref{thm:simple}~Theorem, we derive the following result.

\begin{Theorem}\label{thm:endoparameter}
Suppose that~$F$ is of characteristic zero. 
Let~$\pi$ be a cuspidal representation of~$\SpFV$ and let~$\theta$ be a skew semisimple character contained in~$\pi$. Then the self-dual endo-parameter of~$\theta$ depends only on~$\pi$. Moreover, the diagram
\[
\xymatrix{
\Cusp(\SpFV) \ar[d]_{\mathsf{e}_G} \ar[rr]^{\text{\rm transfer}\quad} && \Irr(\GL_{2N+1}(F)) \ar[d]^{\mathsf{e}_{2N+1}} \\
\CEE^{\mathsf{sd}}_{2N}(F) \ar@{^{(}->}[rr]_{\iota_{2N}} && \CEE_{2N+1}(F)
}
\]
commutes, where~${\mathsf e}_G(\pi)$ denotes the endo-parameter of any skew semisimple character contained in~$\pi$.
\end{Theorem}

We remark that the fact that the map~${\mathsf e}_G$ is well-defined is also proved, in much greater generality and without the assumption that~$F$ has characteristic zero, in~\cite{KSS}; the proof here is quite different and long predates that in~\cite{KSS}. We also remark that we will see later (\ref{thm:ramG}~Theorem) that the map~${\mathsf e}_G$ is in fact surjective.

\begin{proof} 
Let~$\pi$ be a cuspidal representation of~$\SpFV$ and let~$\theta$ be a skew semisimple character contained in~$\pi$, with all the notation from above. In particular, we have a family of skew simple characters~$\theta_i$, for~$0\le i\le l$, and, for each~$i$, the unique self-dual simple character~$\tilde\theta_i$ restricting to~$\theta_i$ and the self-dual endo-class~$\bs\Theta_i$ of~$\tilde\theta_i$. 

For~$(\rho,m)\in\Jord(\pi)$, we write~$\bs\Theta_\rho$ for the endo-class of any simple character in~$\rho$. Then~\ref{cor:together}~Corollary implies that~$\bs\Theta_\rho=\bs\Theta_i^2$, for some~$0\le i\le l$; moreover, together with~\ref{thm:simple}~Theorem it implies
\begin{equation}\label{eqn:commute}
\sum_{(\rho,m)\in\Jord(\pi)} \frac{m\deg(\rho)}{\deg(\bs\Theta_\rho)} \bs\Theta_{\rho} 
=
\sum_{i=0}^l \frac{\dim_F(V^i)}{\deg(\bs\Theta_i)}\bs\Theta_i^2 + \bs\Theta_0^F.
\end{equation}
In particular, the right hand side here is~$\bs\Theta_0^F$ plus the square of the endo-parameter of~$\theta$; since the squaring map is a bijection, this endo-parameter is therefore independent of the choice of~$\theta$ in~$\pi$ since the left hand side is.

Now, according to the results of M\oe glin~\cite[Theorem~3.2.1]{Mo}, the Jordan set exactly determines the endoscopic transfer of~$\pi$ to~$\GL_{2N+1}(F)$; more precisely, the transfer of~$\pi$ is
\[
\prod_{(\rho,m)\in\Jord(\pi)} \St(\rho,m)
\]
where~$\St(\rho,m)$ denotes the unique irreducible quotient of the normalised parabolically induced representation
\[
\rho\nu^{\frac{1-m}2} \times \rho\nu^{\frac{3-m}2} \times\cdots\times \rho\nu^{\frac{m-1}2}
\]
of~$\GL_{m\deg(\rho)}(F)$. The endo-parameter of the transfer of~$\pi$ is thus
\[
\sum_{(\rho,m)\in\Jord(\pi)} \frac{m\deg(\rho)}{\deg(\bs\Theta_\rho)} \bs\Theta_{\rho},
\]
where~$\bs\Theta_\rho$ is the endo-class of (any simple character in)~$\rho$. In particular, this lies in~$ \CEE^{\mathsf{sd}}_{2N+1}(F)$ and~\eqref{eqn:commute} now implies that the diagram commutes.
\end{proof}

\section{Types, covers and reducibility}\label{JBN4new}

In the following subsections we recall the main results about covers and their Hecke algebras, from~\cite{BK1} in the general situation and from~\cite{MS} in the particular situation of interest to us: induction from a maximal parabolic subgroup of a symplectic group. One of the key features in~\cite{MS} is the presence of quadratic characters arising from the comparison of beta-extensions. Using the notion of~$p$-primary beta-extension, together with results from~\cite{Bl}, we describe these characters as signatures of permutations and recall the implications of the structure of the Hecke algebra (including its parameters) for the reducibility of parabolic induction from~\cite{Bl}.

\subsection{}\label{4.1new}
We briefly recall the general notion of a \emph{type} as defined by Bushnell and Kutzko in~\cite{BK1}. Let for a moment~$G$ be the group of~$F$-points of an arbitrary connected reductive group defined over~$F$, let~$L$ be a Levi subgroup of~$G$ and let~$\sigma$ be a cuspidal representation of~$L$. The pair~$(L,\rho)$ determines, through~$G$-conjugacy and twist by unramified characters of~$L$, an inertial class~$\mathfrak s = [L,\rho]_G$ in~$G$. This class~$\mathfrak s$ indexes the Bernstein block~$\mathcal R^\mathfrak s (G)$ (in the category~$\mathcal R (G)$ of smooth representations of~$G$) which is the direct factor of~$\mathcal R(G)$ consisting of representations all of whose irreducible subquotients are subquotients of a representation parabolically induced from an element of~$\mathfrak s$.

Let~$(J, \lambda)$ be a pair made of an open compact subgroup~$J$ of~$G$ and an irreducible smooth representation~$\lambda$ of~$J$, acting on the finite dimensional space~$V_\lambda$. The Hecke algebra of the pair~$(J,\lambda)$ is the intertwining algebra of the representation~$\cInd_J^G \lambda$, traditionally viewed as:
\[
\begin{aligned}
\mathcal H(G , \lambda ) = \{f: G \rightarrow \End(V_\lambda) \mid f &\text{ compactly supported and } \\
&\forall g \in G, \ \forall j,k \in J, \ f(jg k) = \lambda(j) f(g) \lambda(k)
\}
\end{aligned}
\]
The pair~$(J,\lambda)$ is an~$\mathfrak s$-type if the irreducible objects of~$\mathcal R^\mathfrak s (G)$ are exactly the irreducible representations of~$G$ that contain~$\lambda$ upon restriction to~$J$. In such a case there is an equivalence of categories
\[
\mathcal M_\lambda : \mathcal R^\mathfrak s (G) \longrightarrow \operatorname{Mod--}\mathcal H(G , \lambda ), \qquad \mathcal M_\lambda (\pi) = \Hom_J(\lambda, \pi).
\]

\subsection{}\label{4.2new}
There is a counterpart of parabolic induction for types: the notion of~$G$-cover, also defined in~\cite{BK1} by Bushnell and Kutzko. Let~$M$ be a Levi subgroup of~$G$, let~$J_M$ be a compact open subgroup of~$M$ and let~$\lambda_M$ be a smooth irreducible representation of~$J_M$. A {\it~$G$-cover} of the pair~$(J_M, \lambda_M)$ is an analogous pair~$(J, \lambda)$ in~$G$ satisfying the following conditions, for any parabolic subgroup~$P$ of~$G$ of Levi~$M$, where we write~$N$ for the unipotent radical of~$P$, and~$P^-$ for the parabolic subgroup opposite to~$P$ with respect to~$M$, with unipotent radical~$N^-$:
\begin{enumerate}
\item\label{4.2new.i} $J$ has an Iwahori decomposition with respect to~$(M; P)$, i.e.
\[
J= (J\cap N^-) (J \cap M) (J \cap N), \text{ and }J \cap M = J_M;
\]
\item\label{4.2new.ii} $\lambda$ restricts to~$\lambda_M$ on~$J_M$ and to a multiple of the trivial representation on~$J\cap N^-$ and~$J \cap N$;
\item\label{4.2new.iii} the Hecke algebra~$\mathcal H (G, \lambda )$ contains an invertible element supported on the double coset of a strongly positive element of the centre of~$M$~\cite[\S7]{BK1}.
\end{enumerate}
If the pair~$(J_M, \lambda_M)$ is an~$\mathfrak s_M$-type in~$M$ for an inertial class~$\mathfrak s_M = [L, \sigma]_M$ (so that~$L$ is a Levi subgroup of~$M$) and if~$(J, \lambda)$ is a~$G$-cover of~$(J_M, \lambda_M)$, then the pair~$(J, \lambda)$ is an~$\mathfrak s_G$-type in~$G$ for the inertial class~$s_G= [L, \sigma]_G$~\cite[\S8]{BK1}. Furthermore, the third condition above provides us with an injective morphism of algebras~$t : \mathcal H (M, \lambda_M ) \hookrightarrow \mathcal H (G, \lambda )$ that induces on modules a morphism~$t_\ast$ yielding a commutative diagram :
\[
\xymatrix{
\mathcal R^{\mathfrak s_G} (G ) \ar[r]^{\mathcal M_\lambda\ \ }&\operatorname{Mod--}\mathcal H(G , \lambda ) \\
\mathcal R^{\mathfrak s_M} (M ) \ar[u]^{\Ind_{P}^{G}}
\ar[r]^{\mathcal M_{\lambda_M}\ \ \ \ } &\operatorname{Mod--}\mathcal H(M , \lambda_M )\ar[u]_{t_\ast}
}
\]
The reducibility of parabolically induced representations from~$P$ to~$G$, on the left side, can thus be studied in terms of Hecke algebra modules, on the right side. 

\subsection{}\label{4.3new}
This is the tool we use in this paper, whereas the types will be cuspidal types as in paragraphs~\ref{1.6new},~\ref{1.7new}, simple types and semisimple types. As for the relevant Levi and parabolic subgroups, they will come in most cases as follows -- and now we come back to the symplectic group~$\SpFV=\Sp_F(V)$ and the setting of paragraph~\ref{1.1new}. Thus we have a skew semisimple stratum~$[\L,-,0,\beta]$ with associated orthogonal decomposition~$V = \perp_{i=0}^l V^i$, as well as all the other notation from~\S\ref{JBN1new}.

Let~$V = \oplus_{j=-m}^m W^{j}$ be a \emph{self-dual} decomposition of~$V$ (i.e. for which the orthogonal space of~$W^{j}$ is~$\oplus_{k\ne -j}W^{k}$) such that:
{\begin{enumerate}
\def\labelenumi{{\rm(\alph{enumi})}}
\def\theenumi{{\rm(\alph{enumi})}}
\item\label{sub.i} $W^{j} = \oplus_{i=0}^l W^{j}\cap V^i$ and~$W^{j}\cap V^i$ is an~$E_i$-subspace of~$V^i$;
\item\label{sub.ii} $\L(t) = \oplus_{j=-m}^m \L(t) \cap W^{(j)}$, for all~$t \in \mathbb Z$;
\item\label{sub.iii} for any~$r \in \mathbb Z$ and~$i$ with~$0\le i \le l$, there is at most one~$j$, with~$-m \le j \le m$, such that~$\L(r) \cap V^i \cap W^{(j)} \supsetneq \L(r+1) \cap V^i \cap W^{(j)}$;
\item for~$j\ne 0$ there exists~$0\le i\le l$ such that~$W^j\subset V^i$, and~$\tilde P((\Lambda\cap W^j)_{\oEi})$ is a maximal parahoric subgroup of~$\GL_{E_i}(W^j)$;
\item $P^\so((\Lambda\cap W^{0})_{\oE})$ is a maximal parahoric subgroup of~$G\cap \prod_{i=0}^l \GL_{E_i}(W^0\cap V^i)$, which is a group with compact centre.
\end{enumerate}}
Such a decomposition is called \emph{exactly subordinate} to the stratum~$[\L, -, 0, \beta]$ (compare to~\cite[Definition~6.5]{S5}).

Let then~$V = \oplus_{j=-m}^m W^{j}$ be a self-dual decomposition of~$V$ exactly subordinate to the stratum~$[\L, -, 0, \beta]$, let~$M$ be the Levi subgroup of~$\SpFV$ stabilizing this decomposition and let~$P$ be a parabolic subgroup of~$\SpFV$ with Levi component~$M$. Then the pairs~$(H^1(\beta, \L), \theta)$, $(J^1(\beta, \L), \eta)$ and~$(J (\beta, \L), \kappa)$ all satisfy conditions~\ref{4.2new.i} and~\ref{4.2new.ii} of paragraph~\ref{4.2new} above. In fact, for the first two pairs we need only conditions~\ref{sub.i}--\ref{sub.ii} and a self-dual decomposition satisfying these will be called {\it subordinate} to the stratum; for the final pair we need only~\ref{sub.i}--\ref{sub.iii}.

\subsection{}\label{4.4new}
In the next few paragraphs, we subsume the results of~\cite{S5}, in the form easier to refer to taken from~\cite{MS} and in the case that we will focus on, that is, parabolic induction of self-dual cuspidal representations of a maximal Levi subgroup in a symplectic group. We thus continue with the notation of~\S\ref{JBN1new} and fix a cuspidal representation~$\pi$ of~$\SpFV=\Sp_F(V)$. We also fix a finite-dimensional vector space~$W$ over~$F$ and a self-dual cuspidal representation~$\rho$ of~$\GL_F(W)$. We consider the symplectic space~$X = V \perp (W \oplus W^\ast)$ over~$F$, with form
\[
h_X(v_1+w_1+w_1^\ast,v_2+w_2+w_2^\ast) = h(v_1,v_2) + \langle w_1,w_2^\ast \rangle - \langle w_2, w_1^\ast \rangle,
\]
where~$h$ is the symplectic form on~$V$ and~$\langle\cdot,\cdot\rangle$ is the pairing~$W\times W^\ast\to F$.
We put~$M= \GL_F(W) \times \SpFV$, a maximal Levi subgroup of~$\Sp_F(X)$. According to~\cite[Proposition 8.4]{D} and~\cite[\S4.1]{MS}, one can find a type~$(\tilde J_W \times J_V, \tilde \lambda_W \otimes \lambda_V)$ in~$M$ for the cuspidal representation~$\rho \otimes \pi$ of~$M$ and a~$G$-cover of this~$M$-type as follows. 

\subsection{}\label{4.5new}
There exist a skew semisimple stratum~$[\L, -, 0, \beta]$ in~$\End_F(X)$ and a skew semisimple character~$\theta$ of~$H^1(\beta, \L)$ with the following properties.
\begin{itemize}
\item The decomposition~$X = V \perp (W \oplus W^\ast)$ is exactly subordinate to the stratum~$[\L, -, 0, \beta]$. In particular, letting
\[
\L \cap V = \L_V,\quad \beta_{|V} = \beta_V,\quad \L \cap W = \L_W\quad\text{ and }\quad\beta_{|W} = \beta_W,
\]
the stratum~$[\L_V, -, 0, \beta_V]$ in~$\End_F(V)$ is skew semisimple maximal and the stratum~$[\L_W, -, 0, \beta_W]$ in~$\End_F(W)$ is simple maximal. Moreover, the self-duality of~$\rho$ is reflected in the fact that the restriction of~$\beta$ to~$W\oplus W^\ast$ generates a field (equivalently, the restricted stratum~$[\Lambda\cap(W\oplus W^*),-,0,\beta_{|W\oplus W^\ast}]$ is skew simple). We also have
\[
H^1 ( \beta, \L) \cap M \simeq \tilde H^1( \beta_W, \L_W) \times H^1 ( \beta_V, \L_V),
\]
where the isomorphism is given by restriction, and similarly for~$J^1( \beta, \L)$ and for~$J( \beta, \L)$. We will abbreviate~$H^1(\beta,\L)=H^1_{\L}$, $H^1(\beta_V,\L_V)=H^1_V$ and~$\tilde H^1(\beta_W,\L_W)=\tilde H^1_W$, and similarly for~$J^1$ and~$J$.
\item Let~$\tilde \vartheta_W$ be the restriction of~$\theta$ to~$\tilde H^1_W$; this is a self-dual simple character. There are the~$p$-primary beta-extension~$\underline{\tilde \kappa}_W$ of~$\tilde \vartheta_W$ and a self-dual cuspidal representation~$\tilde \tau_W$ of~$\tilde J_W / \tilde J^1_W$ such that~$\rho$ is induced by an extension of~$\tilde \lambda_W = \underline{\tilde \kappa}_W \otimes \tilde \tau_W$ to the normalizer of~$\tilde J_W$.
\item Let~$\theta_V$ be the restriction of~$\theta$ to~$H^1_V$; this is a skew semisimple character. There are the~$p$-primary beta-extension~$\underline{\kappa}_V$ of~$\theta_V$ and a cuspidal representation~$\tau_V$ of~$J_V / J^1_V$ such that~$\pi$ is induced by~$\lambda_V = \underline{\kappa}_V \otimes \tau_V$.
\end{itemize}

\subsection{}\label{4.6new}
Let~$P$ be the parabolic subgroup of~$\Sp_F(X)$ which is the stabiliser of the subspace~$W$ (so stabilises the flag~$W \subset W \perp V \subset X$), let~$U$ be the unipotent radical of~$P$ and let~$P^-$ be the parabolic subgroup opposite to~$P$ with respect to~$M$ (the stabiliser of~$W^\ast$). Also set~$J_P = H^1_{\L} (J_{\L} \cap P)$ and~$J_P^1 = H^1_\L (J^1_{\L} \cap P)$. 

For any extension~$\kappa$ of~$\eta$ to~$J_\L$ we denote by~$\kappa_P$ the natural representation of~$J_P$ in the space of~$(J_\Lambda \cap U)$-fixed vectors under~$\kappa$. In particular, there is a beta-extension~$\kappa_\L$ of~$\theta$ such that~${\kappa_{\L,P}}_{|J\cap M} = \underline{\tilde \kappa}_W \otimes \underline{\kappa}_V$. We can view~$\tau = \tilde \tau_W \otimes \tau_V$ as a cuspidal representation of~$J_P / J_P^1 \simeq \tilde J_W/\tilde J_W^1\times J_V/J_V^1$. Then, letting~$\lambda_P = \kappa_{\L,P} \otimes \tau$, we have:

\begin{Theorem}[{\cite[\S4.1]{MS}}]\label{thm:cover}
$(J_P, \lambda_P)$ is an~$\Sp_F(X)$-cover of~$(\tilde J_W \times J_V, \tilde \lambda_W \otimes \lambda_V)$.
\end{Theorem}

\subsection{}\label{4.6bnew}
Furthermore precise information about the Hecke algebra of this cover is given in {\it loc.cit.}:
\begin{Theorem}[{\cite[Theorem B]{MS}}]\label{thm:coverH}
The Hecke algebra~$\mathscr H(\Sp_F(X), \lambda_P)$ is a Hecke algebra on a dihedral group: it is generated by~$T_0$ and~$T_1$, each invertible and supported on a single double coset, with relations :
\[
(T_i- q^{r_i}) \ (T_i + 1) = 0 , \quad i=0,1, \quad r_0, r_1 \in \mathbb Z.
\]
\end{Theorem}

\subsection{}\label{4.7new}
In fact, the parameters come from rank two Hecke algebras of finite reductive groups as follows~\cite[(7.3) and~\S7.2.2]{S5}. There are two self-dual~$\oE$-lattice sequences~$\mathfrak M_0$ and~$\mathfrak M_1$ in~$X$ such that~$[\mathfrak M_t,- ,0,\beta]$, for~$t=0,1$, are semisimple strata and:
\begin{itemize}
\item the hereditary orders~$\mathfrak b_{0}(\mathfrak M_0)$ and~$\mathfrak b_{0}(\mathfrak M_1)$ are maximal self-dual~$\oE$-orders containing~$\mathfrak b_{0}(\L)$;
\item the decomposition~$X = V \perp (W \oplus W^\ast)$ is subordinate to the strata~$[\mathfrak M_t,- ,0,\beta ]$, for~$t=0,1$;
\item we have~$P ( \L_{\oE} ) = \left(P ( \mathfrak M_{1 ,\oE} )\cap P^- \right) P^1( \mathfrak M_{1 ,\oE} ) = \left(P ( \mathfrak M_{0 ,\oE} )\cap P \right) P^1( \mathfrak M_{0 ,\oE} )$.
\end{itemize}
The representation~$\tau = \tilde \tau_W \otimes \tau_V$ is a cuspidal representation of the Levi subgroup~$\HP(\L_{\oE})=P ( \L_{\oE} ) / P^1 ( \L_{\oE})$ of~$\HP(\mathfrak M_{t ,\oE})=P ( \mathfrak M_{t ,\oE} )/ P^1 ( \mathfrak M_{t ,\oE} )$, for~$t=0,1$, that can be inflated to the parabolic subgroup~$P ( \L_{\oE} ) / P^1 ( \mathfrak M_{t ,\oE} )$, then induced to the full group~$\HP(\mathfrak M_{t ,\oE})$. A specific use of the notion of beta-extension relative to~$\mathfrak M_{t ,\oE}$ leads to self-dual characters~$\chi_t$ of~$\HP(\L_{\oE})$, for~$t=0,1$, giving rise to injective homomorphisms of algebras:
\begin{equation}\label{eqn:injchi}
\mathscr H(\HP(\mathfrak M_{t ,\oE}), \chi_t \otimes \tau) \ \hookrightarrow\ \mathscr H(G, \lambda_P) \qquad (t=0,1).
\end{equation}
We will elaborate on this in paragraph~\ref{4.10new} below.

\subsection{}\label{4.8new}
In order to make use of this, we need some control on the characters~$\chi_t$ and it is here that we really need to use the notion of~$p$-primary beta-extension. We continue with the notation of the previous paragraphs but, for the moment, drop the subscript~$t$ on~$\mathfrak M_t$. We will assume that~$P ( \L_{\oE} ) = \left(P ( \mathfrak M_{\oE} )\cap P \right) P^1( \mathfrak M_{\oE})$ so that, in the notation above, we are doing the case~$t=0$; the case~$t=1$ is obtained by exchanging the parabolic~$P$ with its opposite~$P^-$. Denote by~$\theta_{\mathfrak M}$ the transfer of~$\theta_\L=\theta$ to~$H^1_{\mathfrak M}=H^1(\beta,\mathfrak M)$, and denote by~$\eta_\L,\eta_{\mathfrak M}$ the unique irreducible representations of~$J^1_\L,J^1_{\mathfrak M}$ which contain~$\theta_\L,\theta_{\mathfrak M}$ respectively. Similarly, we have the representations~$\tilde\eta_W,\eta_V$ of~$\tilde J_W^1,J_V^1$ which contain~$\tilde\vartheta_W,\theta_V$ respectively.

For a moment, let~$(J, J^1 , \eta)$ be either~$(J_\mathfrak M , J^1_\mathfrak M , \eta_\mathfrak M)$ or~$(J_\L, J^1_\L, \eta_\L)$, and let~$\kappa$ be any extension of~$\eta$ to~$J$. We define~$\rP (\kappa)$, the \emph{Jacquet restriction} of~$\kappa$, as the natural representation of~$J\cap M$ on the space of~$J^1 \cap U$-invariants of~$\kappa$~\cite[Corollaire 1.12, Lemme 1.18]{Bl}; that is,~$\rP(\kappa)$ is the restriction to~$J\cap M$ of~$\kappa_P$, in the notation of paragraph~\ref{4.6new}.

\subsection{}\label{4.9new}
In order to compute the character~$\chi$ from~\eqref{eqn:injchi}, we need to compare the following two representations of~$J_\L$:
\begin{itemize}
\item the beta-extension~$\underline \kappa_{\L, \mathfrak M}$ of~$\eta_\L$ to~$J_\L$ which is compatible with the~$p$-primary beta-extension~$\underline \kappa_{\mathfrak M}$ of~$\eta_{\mathfrak M}$ to~$J_{\mathfrak M}$ (in the sense of~\cite[Definition 4.5]{S5});
\item the extension~$\kappa_\L=\kappa_\L^P$ of~$\eta_\L$ to~$J_\L$ characterized by the property
\[
\rP( \kappa^P_\L) \simeq \underline{\tilde \kappa}_W \otimes \underline \kappa_V,
\]
where~$\underline{\tilde \kappa}_W$ and~$\underline \kappa_V$ are the~$p$-primary beta-extensions of~$\tilde\eta_W$ and~$\eta_V$ respectively, as above.
(See~\cite[Lemme 1.16]{Bl}.)
\end{itemize}
We apply Jacquet restriction to~$\underline \kappa_{\L, \mathfrak M}$. The groups~$J_\mathfrak M \cap M$ and~$J_\L \cap M$ are both equal to~$\tilde J_W \times J_V$ and the representations~$\rP ( \underline \kappa_{\mathfrak M})$ and~$\rP ( \underline \kappa_{\L, \mathfrak M})$ both extend~$\tilde \eta_W \otimes \eta_V$. From~\cite[Proposition 1.20]{Bl}, the beta-extension~$\underline \kappa_{\L, \mathfrak M}$ is characterized by
\begin{equation}\label{eqn:fait2}
\rP ( \underline \kappa_{\L, \mathfrak M}) = \rP ( \underline \kappa_{\mathfrak M} ).
\end{equation}

\begin{Proposition}\label{prop:Phi}
For~$m \in J_\mathfrak M \cap M$, define~$\epsilon_\mathfrak M (m)$ as the signature of the permutation:
\[
\Ad m : \ u \mapsto m^{-1} u m , \quad u \in J^1_\mathfrak M \cap U^- / H^1_\mathfrak M \cap U^-.
\]
The~$p$-primary beta-extension~$\underline \kappa_{\mathfrak M}$ of~$\eta_{\mathfrak M}$ to~$J_\mathfrak M$ satisfies:
\begin{equation}\label{eqn:fait3}
\rP(\underline \kappa_{\mathfrak M}) \simeq \epsilon_\mathfrak M (\underline{\tilde \kappa}_W \otimes \underline \kappa_V).
\end{equation}
\end{Proposition}

This proposition and~\eqref{eqn:fait2} immediately imply:

\begin{Corollary}
The extensions~$\underline \kappa_{\L, \mathfrak M}$ and~$\kappa^P_\L$ of~$\eta_\L$ to~$J_\L$ are related by:
\[
\rP(\kappa^P_\L) = \epsilon_\mathfrak M \rP(\underline \kappa_{\L, \mathfrak M}).
\]
\end{Corollary}

\begin{proof}[Proof of Proposition]
Let~$\phi$ be an arbitrary extension of~$\eta_\mathfrak M$ to~$J_\mathfrak M$. By~\cite[Lemme~1.10]{Bl} the restriction of~$\phi$ to~$(J_\mathfrak M \cap P) J^1_\mathfrak M$ is induced from the natural representation~$\phi_P$ of~$(J_\mathfrak M \cap P) H^1_\mathfrak M$ on the space of~$J^1_\mathfrak M \cap U$-invariants of~$\phi$. Hence we can realize~$\phi_{|J_\mathfrak M \cap M}$ as the action by right translation on functions
taking values in the space of~$\phi_P$. We also have the representation~$\eta_{\mathfrak M,P}$ of~$(J^1_\mathfrak M \cap P) H^1_\mathfrak M$ on the space of~$J^1_\mathfrak M \cap U$-invariants of~$\eta_{\mathfrak M}$.

Let~$\widetilde{\mathcal S}$ be the space of~$\tilde \eta_W$ and let~$\mathcal S$ be the space of~$\eta_V$, so that~$\widetilde{\mathcal S}\otimes \mathcal S$ is the space of~$\eta_{\mathfrak M,P}$. The representation~$\phi_P$ itself extends~$\eta_{\mathfrak M,P}$, so our representation~$\phi_{|J_\mathfrak M \cap M}$ acts by right translation on the space of functions
\[
f : \ ( J_\mathfrak M \cap P) J^1_\mathfrak M \ \longrightarrow\ \widetilde{\mathcal S} \otimes \mathcal S
\]
satisfying, for all~$x \in ( J_\mathfrak M \cap P) H^1_\mathfrak M$ and all~$g \in ( J_\mathfrak M \cap P) J^1_\mathfrak M$:
\[
f(xg) = \phi_P (x) f(g) .
\]
Using Iwahori decompositions as in~\cite[\S1.3]{Bl} we identify this space with the space~$\mathcal T$ of functions on~$J^1_\mathfrak M \cap U^- / H^1_\mathfrak M \cap U^-$ with values in~$\widetilde{\mathcal S} \otimes \mathcal S$. The action of~$m \in J_\mathfrak M \cap M$ on~$f \in \mathcal T$ is now given by:
\[
\phi (m) f(u) = f(um) = f(m. m^{-1}um)= \rP ( \phi )(m) f(m^{-1}um),
\]
for~$u \in J^1_\mathfrak M \cap U^- / H^1_\mathfrak M \cap U^-$.

Let~$\mathcal T_0$ be the space of complex functions on~$J^1_\mathfrak M \cap U^- / H^1_\mathfrak M \cap U^-$ and~$\mathcal E$ the permutation representation of~$J_\mathfrak M \cap M$ on~$\mathcal T_0$:
\[
\mathcal E(m) f (u) = f( m^{-1} u m), \qquad\text{for }f \in \mathcal T_0,~m \in J_\mathfrak M \cap M,~u \in J^1_\mathfrak M \cap U^- / H^1_\mathfrak M \cap U^-.
\]
We can further identify~$\mathcal T$ with~$\mathcal T_0 \otimes ( \widetilde{\mathcal S} \otimes \mathcal S)$ to obtain~$\phi_{|J_\mathfrak M \cap M} \ \simeq\ \mathcal E \otimes \rP ( \phi )$. 

All of this applies to~$\underline \kappa_\mathfrak M$, so
\[
{\underline \kappa_{\mathfrak M}}_{|J_\mathfrak M \cap M} \quad \simeq \mathcal E \otimes \rP ( \underline \kappa_{\mathfrak M}) .
\]
The determinant of this representation has order a power of~$p$, a property that is unchanged by taking~$p^k$-th powers. Recall that the determinant of some~$x \otimes y$ acting on~$X \otimes Y$ is~$(\det x)^{\dim Y} (\det y)^{\dim X}$. The two spaces here,~$\mathcal T_0$ and~$\widetilde{\mathcal S} \otimes \mathcal S$, have dimension a power of~$p$, which is odd, and the determinant of~$\mathcal E$ acting on~$\mathcal T_0$ is~$\epsilon_\mathfrak M$. 

We now write~$\rP ( \underline \kappa_{\mathfrak M})= \tilde \kappa_W \otimes \kappa_V$ where~$\tilde \kappa_W$ is a beta-extension of~$\tilde \eta_W$ and~$\kappa_V$ is a beta-extension of~$\eta_V$ (see~\eqref{eqn:fait2} and~\cite[Proposition~6.3]{S5}). It is enough to prove
\[
\epsilon_\mathfrak M \ \det \left(\tilde \kappa_W \otimes \kappa_V \right) \text{ has order a power of } p .
\]
Writing~$\epsilon_\mathfrak M$ for the restrictions of~$\epsilon_\mathfrak M$ to~$\tilde J_W$ and to~$J_V$, this condition transforms into
\[
\det \left( \epsilon_\mathfrak M \otimes \tilde \kappa_W \right) \text{ and } \det \left( \epsilon_\mathfrak M \otimes \kappa_V \right) \text{ have order a power of } p.
\]
The character~$\epsilon_\mathfrak M$ is trivial on pro-$p$-subgroups so~$\epsilon_\mathfrak M \otimes \tilde \kappa_W$ and~$\epsilon_\mathfrak M \otimes \kappa_V$ are beta-extensions of~$\tilde\eta_W$ and~$\eta_V$ respectively~\cite[Theorem~4.1]{S5}. This last condition actually means that they are the~$p$-primary beta-extensions of~$\tilde\eta_W$ and~$\eta_V$ respectively, and \eqref{eqn:fait3} follows.
\end{proof}

\subsection{}\label{4.9bnew}
Before returning to the implications on reducibility, we examine the character~$\epsilon_{\mathfrak M}$ a little further. We begin with a general lemma.

\begin{Lemma}\label{lem:sign} 
Let~$Z$ be a finite dimensional vector space over a finite field~$\mathbb F_q$ with odd cardinality~$q$ and let~$g \in \GL_{\mathbb F_q}(Z)$. The signature of the permutation~$g$ of~$Z$ is equal to~$(\det_{\mathbb F_q} g)^\frac{q-1}{2}$. 
\end{Lemma} 

\begin{proof} 
As a character of~$\GL_{\mathbb F_q}(Z)$, the signature is trivial on the derived subgroup, which is~$\SL_{\mathbb F_q}(Z)$, as~$q>2$, hence factors through a character~$\chi$ of the determinant over~$\mathbb F_q$. We know~$\chi^2$ is trivial and it remains to show that~$\chi$ is not identically trivial on~$\GL_{\mathbb F_q}(Z)$. 

We define the following permutation of~$\mathbb F_q$: the multiplication by an element~$\zeta$ of~$\mathbb F_q^\times$ of order~$2^t$ such that~$\frac{q-1}{2^t}$ is an odd integer. This permutation fixes~$0$ and has~$\frac{q-1}{2^t}$ cycles, of length~$2^t$, in~$\mathbb F_q^\times$ so has odd signature. Then the element~$g=\diag(\zeta,1,\ldots,1)$ has odd signature so~$\chi$ is non-trivial.
\end{proof} 

\begin{Proposition}\label{prop:lemma3.8}
For~$m \in J_\mathfrak M \cap M$, the permutation~$\Ad m$ of~$J^1_\mathfrak M \cap U^- / H^1_\mathfrak M \cap U^-$ is an~$\mathbb F_p$-linear transformation of this~$\mathbb F_p$-vector space and 
\[
\epsilon_\mathfrak M (m) = [\det_{\mathbb F_p} \Ad m ]^{\frac{p-1}{2}}. 
\] 
Moreover, the permutation~$u \mapsto m^{-1} u m$ of the space~$J^1_\mathfrak M \cap U / H^1_ \mathfrak M \cap U$ also has signature~$\epsilon_\mathfrak M (m)$. 
\end{Proposition}
 
\begin{proof} 
The first part follows from the previous lemma. Since the decomposition~$X = V \perp (W \oplus W^\ast)$ is subordinate to~$[\mathfrak M, -, 0, \beta]$, the pairing 
\[
\langle x,y\rangle = \theta_{\mathfrak M} ( [x,y]), \qquad\text{for } x \in J^1_\mathfrak M \cap U^- / H^1_\mathfrak M \cap U^-, \quad y \in J^1_\mathfrak M \cap U / H^1_ \mathfrak M \cap U, 
\] 
identifies each of those~$\mathbb F_p$-vector spaces to the dual of the other~\cite[Lemma~5.6]{S5}, in such a way that, for~$m \in J_\mathfrak M \cap M$, the transpose of the map~$x \mapsto m^{-1} x m$, for~$x \in J^1_\mathfrak M \cap U^- / H^1_\mathfrak M \cap U^-$, is~$y \mapsto m y m^{-1}$, for~$y \in J^1_\mathfrak M \cap U / H^1_ \mathfrak M \cap U$. The result follows. 
\end{proof} 

\subsection{}\label{4.10new}
We return to the notation of paragraphs~\ref{4.4new}--\ref{4.7new} and now put together the Hecke algebra homomorphisms~\eqref{eqn:injchi} with~\ref{prop:Phi}~Proposition. Let~$t=0$ or~$1$. We recall from \cite[(7.3)]{S5} (rephrased in the present framework in~\cite[Proposition~3.6]{Bl}) that if~$\kappa = \cInd_{J_{P}}^{J_{\L} } \kappa_P$ is a beta-extension of~$\eta_\L= \cInd_{ J_{P}^1 }^{J^1_{\L} } \eta_P$ \textbf{relative to~$\mathfrak M_{t}$}, then there is an injective morphism of algebras 
\[
\mathscr H(\HP (\mathfrak M_{t ,\oE} ), \tilde\tau_W \otimes\tau_V) \hookrightarrow \mathscr H(\Sp_F(X), \kappa_P \otimes (\tilde\tau_W \otimes\tau_V))
\]
that preserves support. We want to express this with the fixed representation~$\lambda_{P}=\kappa_{\Lambda,P}$ on the right, where~${\kappa_{\L,P}}_{|J\cap M} = \underline{\tilde \kappa}_W \otimes \underline{\kappa}_V$, as in paragraph~\ref{4.6new}. We thus plug in~\ref{prop:Phi}~Proposition above and get: 

\begin{Theorem}\label{thm:jt} 
Let~$t=0$ or~$1$. There is an injective morphism of algebras 
\[
j_t : \mathscr H(\HP ( \mathfrak M_{t ,\oE} ), \epsilon_{\mathfrak M_{t}} (\tilde\tau_W \otimes\tau_V )) \hookrightarrow \mathscr H(\Sp_F(X), \lambda_{P} )
\]
that preserves support, i.e.~$\Supp(j_t(\phi)) = J_P \Supp (\phi) J_P$, for all~$\phi\in\mathscr H(\HP ( \mathfrak M_{t ,\oE} ), \epsilon_{\mathfrak M_{t}} (\tilde\tau_W \otimes\tau_V ))$. 
\end{Theorem}

\subsection{}\label{4.11new}
We now focus on the finite-dimensional algebra~$\mathscr H(\HP ( \mathfrak M_{t ,\oE} ), \epsilon_{\mathfrak M_{t}} (\tilde\tau_W \otimes\tau_V))$, a Hecke algebra on the finite reductive group~$\HP ( \mathfrak M_{t ,\oE} )$ relative to a cuspidal representation of the parabolic subgroup~$P(\L_{\oE}) / P^1(\mathfrak M_{t ,\oE})$. 

Let~$X=\perp_{j=0}^l X^j$ be the splitting associated to the skew semisimple stratum~$[\L,-,0,\beta]$. Since the stratum~$[\L_W,-,0,\beta_W]$ is simple, there is a unique index~$i$ such that~$W\subseteq X^i$, and then~$W^\ast\subseteq X^i$ also. \textbf{This index~$i$ will be fixed until the end of the section.}

Writing~$V^j=V\cap X^j$, the skew semisimple stratum~$[\L_V,-,0,\beta_V]$ then has splitting consisting of the non-zero spaces in~$V=\perp_{j=0}^l V^j$; the only spaces which may be zero here are~$V^0$ (since we have the convention that~$\beta^0=0$) and~$V^i$ (which is zero if and only if~$X^i=W\oplus W^\ast$).

The ambient finite group~$\HP( \mathfrak M_{t ,\oE})$ is a product over~$j$, for~$0 \le j \le l$, of analogous groups relative to~$X^j$, but in all of them except~$X^i$ the parabolic subgroup considered is the full group:
\begin{align*}
P ( \mathfrak M_{t ,\oE} )/ P^1 ( \mathfrak M_{t ,\oE} ) &\simeq P ( \mathfrak M^i_{t ,\oEi} )/P^1 ( \mathfrak M^i_{t ,\oEi} ) \times \prod_{j \ne i} P (\L^j_{\oEj} )/P^1 (\L^j_{\oEj} ) \\
P ( \L_{\oE} ) / P^1 ( \mathfrak M_{t ,\oE} ) &\simeq P ( \L^i_{\oEi} ) / P^1 ( \mathfrak M^i_{t ,\oEi} ) \times \prod_{j \ne i} P ( \L^j_{\oEj} )/P^1 ( \L^j_{\oEj} ).
\end{align*}
The representation~$\tilde \tau_W \otimes\tau_V$ decomposes accordingly using~$\tau_V = \otimes_{j=0}^l \tau_j$ and we finally get an isomorphism of algebras: 
\begin{equation}\label{eqn:thealgebra}
\mathscr H(\HP ( \mathfrak M_{t ,\oE} ), \epsilon_{\mathfrak M_{t}} (\tilde\tau_W \otimes\tau_V )) \simeq \mathscr H(\HP ( \mathfrak M^i_{t ,\oEi} ), \epsilon_{\mathfrak M_{t}} (\tilde\tau_W \otimes \tau_i) ) 
\end{equation}
where~$\tilde \tau_W\otimes \tau_i$ is a cuspidal representation of~$\tilde \HP(\L_{W, \oEi})\times \HP(\L^i_{V,\oEi})$, identified with a maximal Levi subgroup of each finite reductive group~$\HP ( \mathfrak M^i_{t ,\oEi} )$, for~$t=0,1$. 
 
It follows from Lusztig's work~\cite{Lbook}, as recalled in section~\ref{JBN6new}, that this algebra is two-dimensional, because~$\tilde \tau_W$ and~$\epsilon_{\mathfrak M_{t}}$ are self-dual. It has basis given by the identity element and an element~$\mathcal T_t$ supported on the double coset of a certain Weyl group element, called~$s_i$, if~$t=0$, or~$s_i^\varpi$~ if~$t=1$, in~\cite[\S7.2.2]{S5}; this only defines~$\mathcal T_t$ up to a non-zero scalar, which will not matter to us at first. Lusztig gives an algorithm permitting the actual computation of the quadratic relation satisfied by~$\mathcal T_t$. This relation always has the following shape, for some non-zero complex number~$\omega_t$: 
 \begin{equation}\label{eqn:emphasis}
(\mathcal T_t - \ q^{r_t} \ \omega_t) (\mathcal T_t + \ \omega_t ) = 0 , \qquad\text{where } r_t = r_t ( \epsilon_{\mathfrak M_{t}} (\tilde \tau_W\otimes \tau_i) ) \ge 0. 
 \end{equation}
We emphasise the dependency in the inducing cuspidal representation~$\epsilon_{\mathfrak M_{t}} (\tilde \tau_W\otimes \tau_i)$.

\subsection{}\label{4.12new}
Finally, we can restate~\cite[Proposition~3.12]{Bl}, describing the real parts of the reducibility points we wish to compute, in our notation. Recall that, for~$\rho$ a cuspidal representation of~$\GL_F(W)$ as above, we write~$t(\rho)$ for the number of unramified characters~$\chi$ of~$\GL_F(W)$ such that~$\rho \otimes\chi \simeq \rho$. Recall also that, if~$\rho$ is self-dual, then there are precisely two representations~$\rho,\rho'$ in the inertial class of~$\rho$ which are self-dual. 

Let~$\pi$ be a cuspidal representation of~$\SpFV$. Recall that there is a real number~$s_\pi(\rho)\ge 0$ such that, for \emph{real~$s$}, the normalised induced representation~$\nu^s \rho\times\pi$ of~$\Sp_F(X)$ is reducible if and only if~$s=\pm s_\pi(\rho)$, and similarly we have~$s_\pi(\rho')$. Then, for \emph{complex~$s$}, if~$\nu^s\rho\times\pi$ is reducible then the real part of~$s$ must be~$\pm s_\pi(\rho)$ or~$\pm s_\pi(\rho')$; we say that these are the real parts of the reducibility points of~$\nu^s \rho \times \pi$.

\begin{Proposition}[{\cite[Proposition~3.12]{Bl}}]\label{prop:realpartsprop}
Let~$\pi$ be a a cuspidal representation of~$\SpFV$, let~$\rho$ be an irreducible self-dual cuspidal representation of~$\GL_F(W)$, and take all the notation of the previous paragraphs. Then the real parts of the reducibility points of the normalized induced representation~$\nu^s \rho \times \pi$ are the elements of the set 
\begin{equation}\label{eqn:realparts}
\left\{ \pm \frac{r_0 + r_1}{2t(\rho)} , \pm \frac{r_0 - r_1}{2t(\rho)}\right\} , \quad\text{where } r_0= r_0 ( \epsilon_{\mathfrak M_{0}} (\tilde \tau_W\otimes \tau_i)) , \ r_1 = r_1 ( \epsilon_{\mathfrak M_{1}} (\tilde \tau_W\otimes \tau_i) ) . 
\end{equation}
\end{Proposition}

Note that, by~\cite[Lemma 6.2.5]{BK}, the unramified twist number~$t(\rho)$ can also be computed from the formula~$t(\rho)= {\frac{\dim_F W}{e(F[\beta_W]/F)}}$. 

\subsection{}\label{4.13new}
We can also apply the discussion of the previous paragraphs in the space~$X^i=W\oplus V^{i} \oplus W^\ast$. From the splitting of our strata, we have the lattice sequence~$\L_V^i=\L\cap V^i$ and the simple stratum~$[\L_V^i,-,0,\b_V^i]$ in~$V^i$. We write~$J_i=J(\b_V^i,\L_V^i)$, and similarly for~$J_i^1$ and~$H_i^1$, and let~$\underline\kappa_i$ be the~$p$-primary beta-extension of the simple character~$\theta_{|H_i^1}$. Then~$\tau_i$ is a representation of the reductive quotient~$J_i/J_i^1$ and, putting~$G_i=\Sp_F(V^i)$ we can define the cuspidal representation~$\pi_i = \cInd_{J_{\L^i}}^{G_i} \underline\kappa_i \otimes \tau_i$ of~$G_i$. (Note that, if~$V^i=\{0\}$, then~$\pi_i$ is the trivial representation of the trivial group.)

Applying the discussion above to the representation~$\pi_i$ and the space~$X^{i}$, we find that the real parts of the reducibility points of the normalized induced representation~$\nu^s \rho \times \pi_i$ of~$\Sp_F(X^i)$ are the elements of the set 
\begin{equation}\label{eqn:realparts2}
\left\{ \pm \frac{r'_0 + r'_1}{2t(\rho)} , \pm \frac{r'_0 - r'_1}{2t(\rho)}\right\} , \quad\text{where } r'_0= r_0 ( \epsilon_{\mathfrak M_{0}^i} (\tilde \tau_W\otimes \tau_i)) , \ r'_1 = r_1 ( \epsilon_{\mathfrak M_{1}^i} (\tilde \tau_W\otimes \tau_i) ) . 
 \end{equation}
The comparison between~\eqref{eqn:realparts} and~\eqref{eqn:realparts2} will be crucial.

\subsection{}\label{4.14new}
We end this section with the simplest example of the computation of the parameters~$r_0,r_1$ in~\eqref{eqn:realparts}, for positive depth representations. Continuing in the notation above, we assume that~$i>0$ and that~$\b^i$ is \emph{maximal} in the following sense: we have~$[F[\beta^i]:F] = \dim_F V^i$, so that (the image of)~$F[\beta^i]$ is a maximal extension of~$F$ in~$\End_F(V^i)$. In particular, this implies that~$V^i\ne \{0\}$. We assume moreover that~$\dim_F W=\dim_F V^i$, the smallest example of the situation above. (It will turn out that this is in fact the only situation of interest, in this context.) 

Let~$E^i_0$ be the fixed field of the adjoint involution acting on~$E^i=F[\b^i]$. The centralizer of~$\beta^i$ in~$\Sp_F(X^i)$ is thus isomorphic to the unitary group~$\U(2,1)(E^i/E^i_0)$. In the latter group, there are two conjugacy classes of maximal compact subgroups, the reductive quotients of which are, for some~$a$ and~$b$ with~$\{a,b\} = \{0,1\}$ depending on the initial lattice sequence~$\L^i$: 
\begin{itemize}
\item $\HP ( \mathfrak M_{a ,\oE} ) \simeq U(2,1)( k_{E^i} / k_{E^i_0})$ and~$\HP ( \mathfrak M_{b ,\oE} ) \simeq U(1,1)\times U(1)( k_{E^i} / k_{E^i_0})$, if~$E^i/ E^i_0$ is unramified;
\item $\HP ( \mathfrak M_{a ,\oE} ) \simeq \SL(2, k_{E^i}) \times \{\pm 1 \}$ and~$\HP ( \mathfrak M_{b ,\oE} )\simeq \text{O}(2,1)(k_{E^i})$, if~$E^i/ E^i_0$ is ramified. 
 \end{itemize}
We set~$f = f(E^i_0/F)$. From the calculations in section~\ref{JBN6new} the possible values for~$r_a$,~$r_b$ and the sets of real parts of reducibility points in~\eqref{eqn:realparts} are: 
\begin{itemize}
\item if~$E^i/ E^i_0$ is unramified:~$r_a = 3f$ or~$f$, and~$r_b = f$;~real parts~$ \{\pm 1, \pm \frac 12\}$ or~$\{\pm \frac 12, 0 \}$; 
\item if~$E^i/ E^i_0$ is ramified:~$r_a = f$ or~$0$, and~$r_b= f$;~real parts~$\{ \pm 1, 0\}$ or~$\{\pm \frac 12, \pm \frac 12\}$. 
\end{itemize}
In both cases the value of~$r_b= r_b (\epsilon_{\mathfrak M_{b}} \tilde \tau_W \otimes \tau_i)$ is independent of the representation. We choose~$\tilde \tau_W$ such that~$~r_a ( \epsilon_{\mathfrak M_{a}} (\tilde \tau_W \otimes \tau_i))= 3f$ if~$E^i/E^i_0$ is unramified, or~$r_a (\epsilon_{\mathfrak M_{a}} (\tilde \tau_W \otimes \tau_i))= f$ if~$E^i/E^i_0$ is ramified. This choice, which we denote by~$\underline{\tilde\tau}_W$, is unique and provides us with a reducibility with~real part~$1$. 

We conclude that there exists one and only one self-dual cuspidal representation~$\rho$ of~$\GL_F(W)$ containing the simple character~$\tilde \vartheta_W$ such that the parabolically induced representation~$\nu^1 \rho \otimes \pi$ is reducible. The representation~$\rho$ contains the type~$(\tilde J_W,\underline{\tilde \kappa}_W\otimes \underline{\tilde \tau}_W)$. However, as discussed previously, this does not give us a full description of the self-dual representation~$\rho$: we know its inertial class but this still leaves two possibilities. This situation is explored more fully in Section~\ref{JBN3new}.
 
\subsection{}\label{4.15new}
Applying the previous paragraph again to the representation~$\pi_i$ of~$G_i=\Sp_F(V^i)$ and comparing~\eqref{eqn:realparts} and~\eqref{eqn:realparts2}, we remark that the relevant choice of~$\tilde \rho_W$ for the situation in~$X$, with the cuspidal representation~$\pi$ of~$G=\Sp_F(V)$, differs from the analogous choice relative to the situation in~$X^{i}$, with the cuspidal representation~$\pi_i$ of~$G_i$, by a simple twist by the character~$\epsilon_{\mathfrak M_{a}} \epsilon_{\mathfrak M_{a}^i}$. Indeed, in our example, the value of~$r_b$ is independent of the representation. In the next section we will study the general case, when~$r_a$ and~$r_b$ may both depend on the representation.

\section{Reduction to the simple case}\label{JBN5new} 

In this section, we make the reduction to the simple case, proving~\ref{thm:reduction}~Theorem. As intimated at the end of the last chapter, the key point to prove is that the character~$\epsilon_{\mathfrak M_{t}} \epsilon_{\mathfrak M_{t}^i}$ is independent of~$t$ (see~\ref{prop:mainprop}~Proposition). Note that the character~$\epsilon_{\mathfrak M_{t}} \epsilon_{\mathfrak M_{t}^i}$ is the character~$\chi_i$ appearing in the statement of~\ref{thm:reduction}~Theorem. While we have a description of it as a permutation character and, through careful analysis of this permutation, give a recipe by which one could compute it, we do not here compute it precisely; we only check that it is independent of~$t$.

There is one further subtlety which should be remarked upon. In Section~\ref{JBN4new}, we began with a pair of cuspidal representations~$(\rho,\pi)$ and built from them a cover of a type, without starting from types for~$\rho$ and~$\pi$. In this section, we begin just with a cuspidal representation~$\pi$ of~$\SpFV$ and a cuspidal type~$\lambda$ for it, and use this to define certain cuspidal representations~$\rho$ of general linear groups, and maximal simple types~$\tilde\lambda$ for them. The cover obtained in Section~\ref{JBN4new} is then indeed a cover of~$\tilde\lambda\otimes\lambda$ but this is only clear because the (semi)simple characters in~$\lambda$ and~$\tilde\lambda$ are suitably related. Thus we take great care to set up the notation in this section.

\subsection{}\label{5.1new}
We first review the notation that we need. This is the notation as in paragraph~\ref{2.5new} so that it differs slightly from the notation of the previous chapter. In particular, objects in the symplectic space~$V$ do not have the subscript~$V$; instead, the corresponding objects in~$X$ (which we have yet to define) will have the subscript~$X$.

Throughout this and the following paragraphs, we fix a cuspidal representation~$\pi$ of~$\SpFV=\Sp_F(V)$. We have the following data.
\subsubsection*{In the symplectic space~$V$.}
\begin{itemize}
\item A maximal skew semisimple stratum~$[\L, -, 0, \beta]$ in~$\End_F(V)$ and a skew semisimple character~$\theta$ of~$H^1=H^1(\beta, \Lambda)$ such that~$\theta$ occurs in~$\pi$. 
\item The irreducible representation~$\eta$ of~$J^1=J^1(\beta, \Lambda)$ containing~$\theta$ and the~$p$-primary beta-extension~$\underline \kappa$ of~$\eta$ to~$J=J(\beta, \Lambda)$.
\item A cuspidal representation~$\tau$ of~$\HP(\Lambda_{\oE}) = P(\Lambda_{\oE})/ P^1(\Lambda_{\oE})$ such that~$\pi$ is induced from~$\lambda = \underline\kappa \otimes \tau$. 
\end{itemize}

The stratum~$[\L, -, 0, \beta]$ can be written (uniquely) as an orthogonal direct sum of skew simple strata~$[\L^j, -, 0, \beta^j]$ in~$\End_F(V^j)$, for~$j=0, \ldots, l$, with the convention that~$\beta^0=0$. The data above then give us the following data in the spaces~$V^j$.
\begin{itemize}
\item Skew simple characters~$\theta_j$ of~$H^1_j=H^1(\beta^j,\Lambda^j)$, which are the restriction of~$\theta$.
\item The irreducible representation~$\eta_j$ of~$J^1_j=J^1(\beta^j, \Lambda^j)$ containing~$\theta_j$ and the~$p$-primary beta-extension~$\underline{\kappa}_j$ of~$\eta_j$ to~$J_j=J(\beta^j, \Lambda^j)$.
\item The cuspidal representations~$\tau_j$ of~$\HP(\Lambda^j_{\oEj})$ such that, via the isomorphism~$\HP(\Lambda_{\oE})\simeq\prod_{j=0}^l \HP(\Lambda^j_{\oEj})$, we have~$\tau = \otimes_{j=0}^l \tau_j$. 
\item The representation~$\lambda_j = \underline{\kappa}_j \otimes \tau_j$ of~$J_j$. 
\end{itemize}

Note that, writing~$G_j=\Sp_F(V^j)$,  the representation~$\pi_j=\cInd_{J_j}^{G_j}\lambda_j$ is a cuspidal representation. A priori, it is not determined uniquely by the representation~$\pi$, but it is determined by our choice of data~$([\L,-,0,\beta],\theta)$ such that~$\pi$ contains~$\theta$.

\medskip

We now fix~$i \in \{0, \cdots, l\}$ and choose an~$F$-vector space~$W$ whose dimension is divisible by the degree~$[E^i:F]$. We then have the following data.
\subsubsection*{In the vector space~$W$.}
\begin{itemize}
\item A maximal simple stratum~$[\L_W, -, 0, \beta_W]$ in~$\End_F(W)$, together with a field isomorphism 
$E^i=F[\b^i] \rightarrow F[\b_W]=E_W$ fixing~$F$ and taking~$\b^i$ to~$\b_W$.
\item The simple character~$\tilde\vartheta_W$ of~$\tilde H^1_{W}=\tilde{H}^1(\beta_W,\L_W)$ which is the transfer of the square~$(\tilde\theta_i)^2$ of the unique self-dual simple character of~$\tilde H^1(\beta^i,\Lambda^i)$ restricting to~$\theta_i$. 
\item The~$p$-primary beta-extension~$\underline{\tilde \kappa}_W$ of~$\tilde\vartheta_W$ to~$\tilde J_{W}=\tilde{J}(\beta_W,\L_W)$ and an irreducible self-dual cuspidal representation~$\tilde \tau_W$ of~$\tilde \HP(\L_{W,\oEW})$, inflated to~$\tilde J_{W}$, where we have written~$\oEW$ for the ring of integers of~$E_W$.
\item A self-dual cuspidal representation~$\rho$ of~$\GL_F(W)$ containing~$\tilde \lambda_W= \underline{\tilde \kappa}_W\otimes \tilde \tau_W$. 
\end{itemize}

These data also induce data in the dual space~$W^\ast$ as follows. By duplicating if necessary, we assume that~$\Lambda_W$ has period divisible by~$4$ and that~$\Lambda_W(-1)\ne \Lambda_W(0)$. (The reason for doing this is to ensure that the self-dual lattice sequence we will obtain conforms to our standard normalization -- see~\ref{rem:normalizesequence}~Remark.) Writing~$\langle\cdot,\cdot\rangle$ for the pairing~$W\times W^\ast\to F$, we define~$\Lambda_W^\ast$ by
\[
\Lambda_W^\ast(r)=\{w^\ast\in W^\ast \mid \langle \Lambda_W(1-r),w^\ast\rangle\subseteq\pF\},\qquad\text{for }r\in\mathbb Z;
\]
then the lattice sequence~$\Lambda_W\oplus\Lambda_W^*$ is self-dual with respect to the natural symplectic structure on~$W\oplus W^*$. We also define~$\beta_W^\ast$ in~$\End_F(W^\ast)$ by
\[
\langle w,\beta_W^\ast(w^\ast)\rangle = - \langle \beta_W(w),w^\ast\rangle, 
\qquad\text{for all }w\in W,\ w^\ast\in W^\ast.
\]
Note that, by the fact that~$[\Lambda^i,-,0,\beta^i]$ is skew, there is a unique isomorphism~$E^i\to F[\beta_W^\ast]$ which takes~$\beta^i$ to~$\beta_W^\ast$.

\medskip

We now use these data to define corresponding data in the larger spaces on which we will have covers (as in Section~\ref{JBN4new}). We define the symplectic space~$X^i = (W \oplus W^{\ast}) \perp V^i$, for which we have the following.
\subsubsection*{In the symplectic space~$X^i$.}
\begin{itemize}
\item The maximal Levi subgroup~$M_i\simeq\GL_F(W) \times \Sp_F(V^i)$ of~$\Sp_F(X^i)$ which stabilizes the decomposition~$X^i = (W\oplus W^\ast) \perp V^i$, and the maximal parabolic subgroup~$P_i=M_iU_i$ which stabilizes the subspace~$W$ (so stabilizes the flag~$W\subseteq W \perp V^i \subset X^i$). 
\item The skew simple stratum~$[\Lambda_X^i , -, 0, \beta_X^i]$ in~$\End_F(X^i)$, where~$\Lambda_X^i=(\Lambda_W\oplus\Lambda_W^\ast)\perp \Lambda$ and~$\beta_X^i$ is the unique skew simple element which stabilizes the decomposition~$X^i = (W \oplus W^{\ast}) \perp V^i$ and acts as~$\beta^i$ on~$V^i$ and as~$\beta_W$ on~$W$; it then acts as~$\beta_W^\ast$ on~$W^\ast$. We identify~$E^i$ with~$F[\beta_X^i]$ via the isomorphism which takes~$\beta^i$ to~$\beta_X^i$.
\item Two further skew simple strata in~$\End_F(X^i)$,
\[
[\mathfrak M^{i}_0 , - , 0, \beta_X^i], \ [\mathfrak M^{i}_1 , -, 0, \beta_X^i],
\]
such that~$\mathfrak b_0(\mathfrak M^{i}_t)$, for~$t=0,1$, are the two maximal self-dual~$\oEi$-orders in the commuting algebra of~$\beta_X^i$ which contain~$\mathfrak b_0(\Lambda_X^i)$.
\item The unique skew simple character~$\theta_X^i$ of~$H^1_{X^i}=H^1(\beta_X^i,\Lambda_X^i)$ that restricts to~$\theta_i$ on~$H^1_{i}$ and to~$\tilde\vartheta_W$ on~$\tilde H^1_{W}$; this is the transfer to~$\Lambda_X^i$ of the skew simple character~$\theta_i$.
\item For~$t=0,1$, the skew simple character~$\theta_{\mathfrak M^{i}_t}$ of~$H^1_{\mathfrak M^{i}_t}$ that is transferred from~$\theta_X^i$; the corresponding irreducible representation~$\eta_{\mathfrak M^{i}_t}$ of~$J^1_{\mathfrak M^{i}_t}$; and the~$p$-primary beta-extension~$\underline\kappa_{\mathfrak M^{i}_t}$ of~$\eta_{\mathfrak M^{i}_t}$ to~$J_{\mathfrak M^{i}_t}$. 
\item An~$\Sp_F(X^i)$-cover~$(J_{P}^i,  \lambda_{P}^i)$ of the pair~$(\tilde J_{W}\times J_{i}, \tilde \lambda_W\otimes \lambda_i)$ in~$M_i$.
\end{itemize}

\medskip

Finally, we define the symplectic space~$X = (W\oplus W^\ast)\perp V= X^i\perp V^{\vee i}$, where~$V^{\vee i} = \mathop{\perp}\limits_{j \ne i} V^j$, for which we have the following.
\subsubsection*{In the symplectic space~$X$.}
\begin{itemize}
\item The maximal Levi subgroup~$M\simeq\GL_F(W) \times \Sp_F(V)$ of~$\Sp_F(X)$ which stabilizes the decomposition~$X = (W\oplus W^\ast)\perp V$, and the maximal parabolic subgroup~$P=MU$ which stabilizes the subspace~$W$.
\item The skew semisimple stratum~$[\Lambda_X,-,0,\beta_X]$, where~$\Lambda_X=\Lambda_X^i \perp\Lambda^{\vee i}$, with~$\L^{\vee i} = \perp_{j \ne i} \L^j$, and~$\beta$ the unique skew semisimple element which stabilises the decomposition~$X = (W\oplus W^\ast)\perp V$ and acts as~$\beta$ on~$V$ and~$\beta_W$ on~$W$ (or, equivalently, acts as~$\beta_X^i$ on~$X^i$ and as~$\beta^{\vee i}=\bigoplus_{j\ne i}\beta^j$ on~$V^{\vee i}$, from which it is clear that the resulting stratum is indeed semisimple). We identify~$E$ with~$F[\beta_X]$ via the isomorphism which takes~$\beta^i$ to~$\beta_X^i$ and~$\beta^j$ to itself, for~$j\ne i$.
\item Two further skew semisimple strata
\[
[\mathfrak M_0 , -, 0, \beta_X], \ [\mathfrak M_1 , -, 0, \beta_X],
\]
where~$\mathfrak M_t=  \mathfrak M_t^i\perp \Lambda^{\vee i}$, for~$t=0,1$; then~$\mathfrak b_0(\mathfrak M_t)$ are the two maximal self-dual~$\oE$-orders in the commuting algebra of~$\beta_X$ which contain
$\mathfrak b_0(\Lambda_X)$. 
\item The unique skew semisimple character~$\theta_X$ of~$H^1_X=H^1(\beta_X,\Lambda_X)$ which restricts to~$\theta$ on~$H^1$ and to~$\tilde\vartheta_W$ on~$\tilde H^1_W$; it is the transfer to~$H^1_X$ of the skew semisimple character~$\theta$, and restricts to~$\theta_X^i$ on~$H^1_{X^i}$.
\item For~$t=0,1$, the skew semisimple character~$\theta_{\mathfrak M_t}$ of~$H^1_{\mathfrak M_t}$ that is transferred from~$\theta_X$; the corresponding irreducible representation~$\eta_{\mathfrak M_t}$ of~$J^1_{\mathfrak M_t}$; and the~$p$-primary beta-extension~$\underline\kappa_{\mathfrak M_t}$ of~$\eta_{\mathfrak M_t}$ to~$J_{\mathfrak M_t}$. 
\item An~$\Sp_F(X)$-cover~$(J_P, \lambda_P)$ of the pair~$(\tilde J_W\times J, \tilde \lambda_W \otimes \lambda)$ in~$M$.
\end{itemize}

\subsection{}\label{5.2new} 
We use the setup in the previous paragraph and come back to the comparison of real parts of reducibility points, as in paragraph~\ref{4.15new}. The comparison of beta-extensions yields, as in~\ref{prop:Phi}~Proposition, characters~$\epsilon_{\mathfrak M_{t}^i}$ and~$\epsilon_{\mathfrak M_{t}}$ for~$t=0,1$. 

We fix~$t=0,1$ and temporarily drop the subscript~$t$. By definition~$\epsilon_{\mathfrak M}(m)$, for~$m \in J_{\mathfrak M}  \cap M$,  is the signature of the permutation~$\Ad m : u \mapsto  m^{-1} u m$ of the quotient~ $J^1_{\mathfrak M} \cap U/ H^1_{\mathfrak M}  \cap U$, isomorphic to the~$\mathbb F_p$-vector space~$\mathfrak J^1_{\mathfrak M}  \cap \mathbb U / \mathfrak H^1_{\mathfrak M}  \cap \mathbb U$, where~$\mathbb U$ is the Lie algebra of~$U$  (see~\ref{prop:Phi}~Proposition and~\ref{prop:lemma3.8}~Proposition). The same holds with~$\epsilon_{\mathfrak M^i} (m)$, for~$m \in J_{\mathfrak M^i} \cap M^i$: it is the signature of the same permutation on~$\mathfrak J^1_{\mathfrak M^i}  \cap \mathbb U_i / \mathfrak H^1_{\mathfrak M^i}  \cap \mathbb U_i$. On the other hand~$\mathbb U$ is isomorphic to~$\mathbb U_i \oplus \Hom_F(V^{\vee i},W)$ in an~$M_i$-equivariant way, and the action of~$(m,y) \in \GL_F(W) \times \Sp_F(V^i)$ on~$\phi \in \Hom_F(V^{\vee i}, W)$ is given by~$\phi \mapsto m \phi$. The associated decompositions of the lattices~$\mathfrak J^1$ and~$\mathfrak H^1$ (as in~\cite[Proposition~7.1.12]{BK}) lead to:

\begin{Lemma}\label{lem:com}  
Let~$(m,y) \in\tilde P(\L_{W,\oEW})\times  P(\L^i_{\oEi})$. Then~$(\epsilon_{\mathfrak M^i}   \epsilon_{\mathfrak M})((m,y) )$  is the signature of the permutation~$\phi \mapsto m \phi$ of  
\[
\mathfrak X  := \mathfrak J^1_{\mathfrak M}  \cap \Hom_F(V^{\vee i}, W )/ \mathfrak H^1_{\mathfrak M}  \cap \Hom_F(V^{\vee i}, W).
\]
\end{Lemma} 

Now the quotient group~$\tilde P(\L_{W,\oEW}) /  \tilde P^1(\L_{W,\oEW})$ is a general linear group~$ \GL_{m_W}(k_W)$ over the finite extension~$k_W=k_{E^i}$ of~$k_F$; this extension depends only on the endo-class of the simple character~$\tilde\vartheta_W$. The lemma actually asserts that the character~$\epsilon_\mathfrak M \epsilon_{\mathfrak M^i}$  is trivial on~$P(\L^i_{\oEi})$ and factors through the signature of the natural left action of~$ \GL_{m_W}(k_W)$ on~$\mathfrak X$.  

\subsection{}\label{5.2bnew} 
Retrieving the subscripts~$t$, our main tool is the following comparison of characters:

\begin{Proposition}\label{prop:mainprop} 
With notation as above, we have
\[
\epsilon_{\mathfrak M_{0}^i} \epsilon_{\mathfrak M_{0}} 
= \epsilon_{\mathfrak M_{1}^i} \epsilon_{\mathfrak M_{1}} . 
\]
This character, as a character of~$\GL_{m_W}(k_W)$, can be written as~$ \chi_i\circ\det_{k_W}$, where~$\chi_i$ is a quadratic or trivial character of~$k_W^\times$ which is independent of the choice of the space~$W$. 
\end{Proposition} 

The independence on the space~$W$ (for a fixed choice of~$i$) is particularly important. We postpone the proof of the Proposition for now and, taking it for granted, deduce~\ref{thm:reduction}~Theorem. 

\subsection{}\label{5.3new} 
Recall that we have written~$\pi=\cInd_J^\SpFV \lambda$ and, for~$j=0,\ldots,l$, we have the cuspidal representation~$\pi_j=\cInd_{J_j}^{G_j}\lambda_j$ of~$G_j=\Sp_F(V^j)$. We have~$\theta_j$, the simple character of~$H_j^1$ contained in~$\lambda_j$, and we write~$\tilde\theta_j$ for the self-dual simple character of~$\tilde H_j^1$ which restricts to~$\theta_j$. Let~$\bs\Theta_j$ be the endo-class of the simple character~$(\tilde\theta_j)^2$, which is a simple character for the stratum~$[\Lambda^j,-,0,2\beta^j]$, and~$k_{\bs\Theta_j}$ for the corresponding extension of~$k_F$. 

Recall that, for an endo-class~$\bs\Theta$ and a character~$\chi$ of the multiplicative group of the corresponding finite field~$k_{\bs\Theta}$, we have:
\begin{itemize} 
\item $\Jord(\pi)$, the Jordan set of~$\pi$ (see paragraph~\ref{2.1new});
\item $\IJord(\pi,\bs\Theta)$, the inertial Jordan set of~$\pi$ relative to~$\bs\Theta$, which is the multiset of pairs~$([\rho],m)$, for~$(\rho,m)\in\Jord(\pi)$ such that~$\rho$ has endo-class~$\bs\Theta$;
\item $\IJord(\pi,\bs\Theta)_\chi$, the $\chi$-twisted inertial Jordan set of $\pi$ relative to $\bs\Theta$, which is the multiset of pairs~$([\rho]_\chi, m)$ with~$(\rho, m) \in \Jord(\pi,\bs\Theta)$.
\end{itemize}
Recall here that, if~$\rho$ contains a maximal simple type~$(\tilde J,\tilde\lambda)$, then~$[\rho]_\chi$ denotes the inertial class of cuspidal representations containing~$(\tilde J,\tilde\lambda\otimes\chi\circ\det)$ (see paragraph~\ref{2.4new}). Also, when~$\chi$ is the trivial character we just write~$\IJord(\pi,\bs\Theta)$.

We restate~\ref{thm:reduction}~Theorem in a refined form:

\begin{Theorem}\label{thm:reductiondetail}
Fix~$i$ with~$0\le i\le l$, and let~$\chi_i$ be the character of~$k_{\bs\Theta_i}^\times$ 
such that~$\chi_i\circ \det_{k_{\bs\Theta_i}}$ is the twisting character in~\ref{prop:mainprop}~Proposition. We have an equality of multisets
\[
\IJord(\pi,\bs\Theta_i) = \IJord(\pi_i,\bs\Theta_i)_{\chi_i}.
\] 
\end{Theorem}

\begin{proof}
This is now just a matter of putting together the previous results. Let~$\rho$ be a cuspidal representation with endo-class~$\bs\Theta_i$ and use the notation of paragraph~\ref{5.1new} so that~$\rho$ is a representation of~$\GL_F(W)$ containing the maximal simple type~$\tilde\lambda_W=\underline{\tilde \kappa}_W\otimes \tilde \tau_W$. The values of~$m$, if any, for which~$([\rho],m)\in\IJord(\pi_i,\bs\Theta_i)$ can then be computed from~\eqref{eqn:realparts2}: more precisely, they are
\begin{equation}\label{eqn:realpartswithpii}
\left| \frac{r_0 ( \epsilon_{\mathfrak M^i_{0}} (\tilde \tau_W\otimes \tau_i))\pm r_1 ( \epsilon_{\mathfrak M^i_{1}} (\tilde \tau_W\otimes \tau_i))}{t(\rho)}\right|-1,
\end{equation}
whenever these integers are strictly positive, together with positive integers less than this and of the same parity.

Now we consider the inertial class~$[\rho]_{\chi_i}$. The cuspidal representations in this class contain the maximal simple type~$\tilde\lambda_W\otimes\chi_i\circ\det=\underline{\tilde \kappa}_W\otimes (\tilde \tau_W\otimes\chi_i\circ\det)$. Then, using~\eqref{eqn:realparts}, we see that the values of~$m$ for which~$([\rho]_{\chi_i},m)\in\IJord(\pi,\bs\Theta_i)$ are
\[
\left| \frac{r_0 ( \epsilon_{\mathfrak M_{0}} ((\tilde \tau_W\otimes\chi_i\circ\det)\otimes \tau_i))\pm r_1 ( \epsilon_{\mathfrak M_{1}} ((\tilde \tau_W\otimes\chi_i\circ\det)\otimes \tau_i))}{t(\rho)}\right|-1,
\]
whenever these integers are strictly positive, together with positive integers less than this and of the same parity. But~\ref{prop:mainprop}~Proposition says that these are precisely the same integers as those in~\eqref{eqn:realpartswithpii} (recall that all the characters here are quadratic or trivial), and the result follows.
\end{proof} 

\begin{Remark}\label{rem:bothappear}
As we have seen in the proof, the pairs~$([\rho],m)$ which appear in~$\IJord(\pi)$ are determined by the values of~$r_t=r_t(\epsilon_{\mathfrak M^i} (\tilde \tau_W\otimes \tau_i))$. Denote by~$\rho'$ the other self-dual cuspidal in the inertial class~$[\rho]$. If~$\rho$ and~$\rho'$ are of opposite parity, say~$\rho$ is of symplectic type and~$\rho'$ is of orthogonal type, then we also recover this part of the full Jordan set~$\Jord(\pi)$: if~$m$ is even then it is~$(\rho,m)$ which appears in~$\Jord(\pi)$, while if~$m$ is odd then it is~$(\rho',m)$.
 
Suppose now that~$\rho,\rho'$ are of the same parity and~$([\rho],m)$ appears in~$\IJord(\pi)$. Then~$\rho$ and~$\rho'$ both appear \emph{with the same multiplicities} in~$\Jord(\pi)$ if and only if~$r_0r_1=0$. Thus in this case we also recover this part of the full Jordan set. Both~$\rho$ and~$\rho'$ appear with \emph{some} multiplicity in~$\Jord(\pi)$ if and only if~$|r_0-r_1|>t(\rho)$; when~$\rho,\rho'$ are both of orthogonal type, this condition simplifies to~$r_0\ne r_1$, since the reducibility points must be integers in this case.

The situations in which~$\rho,\rho'$ have the same parity are examined more closely from the Galois point of view in Section~\ref{JBN3new}.
\end{Remark}

It remains now to prove~\ref{prop:mainprop}~Proposition, which will take up the remainder of this section.

\subsection{}\label{5.5new} 
In this and the next few paragraphs, we define and study an auxiliary lattice sequence which will be needed for the calculations. Let~$\Lambda_W$ and~$\Lambda_Y$ be~$\oF$-lattice sequences in finite dimensional~$F$-vector spaces~$W$ and~$Y$ respectively, with the same~$\oF$-period~$e$. We define an~$\oF$-lattice sequence~$\mathcal C= \mathcal C ( \Lambda_Y, \Lambda_W)$ in the vector space~$C= \Hom_F(Y, W)$ by 
\[
\mathcal C (t) = \{g \in C \mid g \Lambda_Y (i) \subseteq \Lambda_W (i+t) \text{ for all }i\in\mathbb  Z\}, \qquad\text{for }t \in \mathbb Z. 
\]
We call the \emph{jumps} of~$\Lambda_Y$ those integers~$i$ such that~$\Lambda_Y(i) \ne \Lambda_Y(i+1)$ (and similarly for any lattice sequence). The set of jumps of~$\Lambda_Y$ is also the image of~$Y\setminus\{0\}$ by the valuation map attached to~$\Lambda_Y$, given by~$\val_Y(y) = \max \{ k \in \mathbb Z \mid y \in \Lambda_Y(k)\}$, for~$y \in Y\setminus\{0\}$. 

We make the following assumptions:
\begin{enumerate}
\item\label{assumption.i}
The set of jumps of~$\Lambda_W$ is equal to~$a_W + s_W\mathbb Z$ and the set of jumps of~$\Lambda_Y$ is equal to~$a_Y + s_Y\mathbb Z$. 
\item\label{assumption.ii}
The orders~$\mathfrak a(\Lambda_W)$ and~$\mathfrak a(\Lambda_Y)$ are principal orders, in other words non-zero quotients~$\Lambda_W(i) / \Lambda_W(i+1)$ are all isomorphic, and the same for~$\Lambda_Y$. In particular there is an element~$\Pi_W \in \mathfrak a (\Lambda_W)$ (resp.~$\Pi_Y \in \mathfrak a(\Lambda_Y)$) such that~$\Pi_W ( \Lambda_W(i)) = \Lambda_W (i+1)$ (resp.~$\Pi_Y (\Lambda_Y(i)) = \Lambda_Y(i+1)$) whenever~$i$ is a jump of~$\Lambda_W$ (resp. of~$\Lambda_Y$)~\cite[\S5.5]{BK}. 
\end{enumerate}

\begin{Lemma}\label{lem:jump} 
The set of jumps of~$\mathcal C$ is equal to~$(a_W - a_Y) + \gcd (s_W, s_Y )\mathbb Z$. Moreover, the quotient spaces~$\mathcal C(i) / \mathcal C(i+1)$ that are non-zero are all isomorphic as~$\kF$-vector spaces, and their common dimension is 
\[
c = c ( \Lambda_Y, \Lambda_W) = \frac{\gcd (s_W, s_Y)}{e} \, \dim_F W \dim_F Y.
\] 
\end{Lemma} 

\begin{proof}
Proving that the set of jumps is contained in the given~$\mathbb Z$-coset is straightforward using only assumption~\ref{assumption.i}. Now we use assumption~\ref{assumption.ii} and remark that~$\Pi_W$ and~$\Pi_Y$ satisfy~$\Pi_W (\Lambda_W(i)) = \Lambda_W(i+s_W)$ and~$\Pi_Y (\Lambda_Y(i)) = \Lambda_Y(i+s_Y)$, for any integer~$i$. For any~$\phi \in C$ we check that: 
\[
\val_\mathcal C ( \Pi_W \phi) = 
\val_\mathcal C ( \phi ) + s_W , 
\quad 
\val_\mathcal C (\phi \Pi_Y) = 
\val_\mathcal C ( \phi ) + s_Y.
\]
The left (resp. right) multiplication by~$\Pi_W$ (resp.~$\Pi_Y$) is thus an isomorphism of~$\oF$-modules from~$\mathcal C(t)$ onto~$\mathcal C(t +s_W )$ (resp.~$\mathcal C(t +s_Y)$) whence the isomorphy. 

To compute the dimension we use the generalized index notation~$[A:B]$ for two lattices~$A$ and~$B$ in a same finite dimensional vector space:~$[A:B]$ is just the ordinary quotient of~$[A:X]$ and~$[B:X]$ for any lattice~$X$ contained in~$A$ and~$B$. 

The common~$\oF$-period~$e$ is a multiple of~$s_W$ and~$s_Y$, say~$e = r_W s_W= r_Y s_Y$. Write~$s = \gcd (s_W, s_Y )$ and pick integers~$n,m$ such that~$s = n s_W+ m s_Y$. We have, for any integer~$k$, 
\begin{align*} 
[ \mathcal C(k) : \mathcal C(k+s) ]  &=
[\mathcal C(k) : \mathcal C(k+n s_W) ]\ [ \mathcal C(k+n s_W) : \mathcal C(k+n s_W+ m s_Y) ] \\
& = [\mathcal C(k) : \mathcal C(k+s_W) ]^n\ [\mathcal C(k) : \mathcal C(k+s_Y) ]^m \\
& =[\mathcal C(k) : \varpi_F \mathcal C(k ) ]^\frac{n}{r_W} \ [\mathcal C(k) : \varpi_F \mathcal C(k) ]^\frac{m}{r_Y} \\
&  = [\mathcal C(k) : \varpi_F \mathcal C(k )]^\frac{s}{e} 
\end{align*}
whence the result. 
\end{proof}

\subsection{}\label{5.6new} 
We will need to determine the effect on~$\mathcal C (\Lambda_Y,\Lambda_W)$ of a shift in indices on~$\Lambda_W$. We further assume the following.

\begin{Notation}\label{assumptionW}
\begin{enumerate}
\item\label{assumptionW.i}
The space~$W$ is an~$E_W$-vector space for some finite extension~$E_W$ of~$F$, with ramification index~$e_W$ and residue field~$k_W$ of cardinality~$q_W$.
\item\label{assumptionW.ii} 
We fix two~$\oEW$-lattice sequences~$\Lambda_{W,0}$ and~$\Lambda_{W,1}$ in~$W$ with the same underlying lattice chain of period~$1$ over~$E_W$ (so that~$s_{W,0} = s_{W,1}= e/e_W$) 
and with jumps at~$a_{W,0} =0$ and~$a_{W,1} = \frac{e}{2e_W}$ respectively. 
\end{enumerate}
\end{Notation}

We write~$s_Y = \frac{e}{r_Y}$ and put~$\mathcal C_t = \mathcal C (\Lambda_Y,\Lambda_{W,t})$, for~$t=0, 1$. The sets of jumps of~$\mathcal C_0$,~$\mathcal C_1$ are respectively 
\[
-a_Y + \gcd \left( \frac{e}{r_Y}, \frac{e}{e_W} \right) \mathbb Z 
\quad \text{ and } \quad 
\frac{e}{2e_W} -a_Y +\gcd\left( \frac{e}{r_Y}, \frac{e}{e_W} \right) \mathbb Z; 
\]
they are the same when~$\frac{e}{2e_W}$ divides~$\gcd\left(\frac{e}{r_Y}, \frac{e}{e_W}\right)$. We get the following, where~$\val_2$ is the~$2$-adic valuation of an integer.

\begin{Lemma}\label{lem:samejumps} 
$\mathcal C_0$ and~$\mathcal C_1$ have the same jumps if and only if~$\val_2(e_W) < \val_2(r_Y)$. Otherwise the jumps of~$\mathcal C_0$ and~$\mathcal C_1$ are shifted by~$\frac 1 2 \, \gcd ( \frac{e}{r_Y}, \frac{e}{e_W} )$. 
\end{Lemma} 

\subsection{}\label{5.7new} 
We now observe that the group~$\GL_{m_W}(k_W)$ acts on the quotients~$\mathcal C_t(i) / \mathcal C_t(i+1)$ by left multiplication, where~$m_W = \dim_{E_W} W$. These actions commute with the left action of~$E_W^\times$ and with the right action of~$\Pi_Y$ so, on the non-zero quotients, they are all equivalent and the corresponding permutations of the non-zero sets~$\mathcal C_t(i) / \mathcal C_t(i+1)$ all have the same signature. 

In the same fashion the non-zero quotients~$\mathcal C_t(i) / \mathcal C_t(i+1)$ are isomorphic left modules over~$\mathfrak a_0(\Lambda_{W,t,\oEW})/ \mathfrak a_1(\Lambda_{W, t,\oEW}) \simeq M_{m_W}(k_W)$. The latter is a simple algebra hence those modules have composition series with~$d$ simple quotients all isomorphic to the natural module~$k_W^{m_W}$. The determinant of the action of~$g \in \GL_{m_W}(k_W)$ on any such module is thus~$\left(\det_{k_W} g \right)^d$ and the signature of the corresponding permutation is ~$\left((\det_{k_W} g)^{\frac{q_W-1}{2}} \right)^d$ by~\ref{lem:sign}~Lemma. The associated character of~$\GL_{m_W}(k_W)$ is then trivial if and only if~$d$ is even. Now~\ref{lem:jump}~Lemma 
gives us:
\[
d = \frac{1}{m_W [k_W:k_F]} \ \frac{\gcd \left(\frac{e}{r_Y}, \frac{e}{e_W}\right)}{e} \, \dim_F W \dim_F Y
\]
Since~$\dim_F W = e_W [k_W:k_F] m_W$, we conclude: 

\begin{Lemma}\label{lem:computed} 
The signature of the natural left action of~$\GL_{m_W}(k_W)$ on the non trivial quotients~$\mathcal C_t(i) / \mathcal C_t(i+1)$ is the trivial character if and only if 
\[
d = \dfrac{e_W}{\lcm ( r_Y , e_W )} \, \dim_F Y
\]
is \emph{even}; otherwise it is the unique character of~$\GL_{m_W}(k_W)$ of order two. In particular: 
\begin{itemize}
\item this signature only depends on~$e_W$, not on~$W$ itself; 
\item when~$\mathcal C_0$ and~$\mathcal C_1$ do not have the same jumps, we have~$d \equiv \dim_F Y \pmod{2}$.
\end{itemize}
\end{Lemma} 

\subsection{}\label{5.4new} 
We return to the notation of paragraphs~\ref{5.1new},~\ref{5.2new} but, for now, we drop the subscript~$t$ so that~$\mathfrak M$ denotes either of the orders~$\mathfrak M_0$ or~$\mathfrak M_1$. We first detail the structure of the~$\mathfrak b_0(\mathfrak M)$-bimodule~$\mathfrak J^1_{\mathfrak M} \cap \mathbb U / \mathfrak H^1_{\mathfrak M} \cap \mathbb U$, isomorphic to~$J^1_{\mathfrak M} \cap U / H^1_{\mathfrak M} \cap U$ by the Cayley map, or equivalently by~$Y \mapsto 1 + Y$. (Recall that~$\mathbb U$ denotes the Lie algebra of~$U$.) We use the inductive definition of the orders~$\mathfrak J_{\mathfrak M}$ and~$\mathfrak H_{\mathfrak M}$ given in~\cite[\S3.2]{S4}. 

We have, for some~$u \ge 1$, a sequence~$(\gamma_0 = \beta, \gamma_1, \cdots, \gamma_u =0)$ and a strictly increasing sequence of integers~$0 < r_0 < \cdots < r_{u-1}=n:=v_{\mathfrak M}(\beta)$ such that, for~$0\le v \le u-1$, the stratum~$[\mathfrak M, -, r_v-1, \gamma_v ]$ is semi-simple and the stratum~$[\mathfrak M, -, r_v , \gamma_v ]$ is equivalent to~$[\mathfrak M, -, r_v , \gamma_{v+1} ]$. Using the inductive definition and writing~$[Z]$ for the image in the Grothendieck group of a~$\mathfrak b_0(\mathfrak M)$-bimodule~$Z$, we find that: 
\begin{equation}\label{eqn:induction}
[\mathfrak J^1_{\mathfrak M} / \mathfrak H^1_{\mathfrak M} ]
= [\mathfrak a^{n/2}_{\mathfrak M} / \mathfrak a^{(n/2)+}_{\mathfrak M} ] 
- 
[\mathfrak b_{\gamma_0,\mathfrak M}^{r_0/2} / \mathfrak b_{\gamma_0,\mathfrak M}^{(r_0/2)+}] 
+ \sum_{v=1}^{u-1} \left( [
\mathfrak b_{\gamma_v,\mathfrak M}^{r_{v-1}/2} / \mathfrak b_{\gamma_v,\mathfrak M}^{(r_{v-1}/2)+}] 
- [\mathfrak b_{\gamma_v,\mathfrak M}^{r_{v }/2} / \mathfrak b_{\gamma_v,\mathfrak M}^{(r_{v }/2)+}]
\right),
\end{equation}
where~$\mathfrak b^{r/2}_{\gamma,\mathfrak M}$ is shorthand for the intersection of~$\mathfrak a_{r/2}(\mathfrak M)$ with the centraliser of~$\gamma$.

From~\cite[Proposition~3.4]{S4}, we may choose the elements~$\gamma_v$ so that the decomposition~$X = V \perp (W \oplus W^\ast)$ is subordinate to all strata considered above; in particular we can take intersections with~$\mathbb U$ in every term in the above equality. Then the value~$\epsilon_\mathfrak M (m)$ of the quadratic character~$\epsilon_\mathfrak M$ can be calculated as the product of the signatures of the permutation~$\Ad m$ on each resulting quotient. 

\subsection{}\label{5.8new} 
We now begin the proof of~\ref{prop:mainprop}~Proposition. Recall that, by~\ref{lem:com}~Lemma, the character~$\epsilon_{\mathfrak M^i} \epsilon_{\mathfrak M}$ is given by the signature of the permutation~$\phi\mapsto m\phi$ on 
\[
\mathfrak X  := \mathfrak J^1_{\mathfrak M}  \cap \Hom_F(V^{\vee i}, W )/ \mathfrak H^1_{\mathfrak M}  \cap \Hom_F(V^{\vee i}, W),
\]
for~$m\in\tilde P(\Lambda_{W,\oEW})$.

The space~$\Hom_F(V^{\vee i}, W)$ decomposes as a direct sum~$\oplus_{j \ne i} \Hom_F(V^j, W)$. Moreover, each~$V^j$ in turn decomposes as a direct sum~$V^j = \perp_{s=1}^{s_j} Y^{j,s}$ of subspaces~$Y^{j,s}$ for which the assumptions of \S \ref{5.5new}--\ref{5.7new} are satisfied, and such that the resulting decomposition of~$V$ is subordinate to~$[\Lambda, -, 0, \beta]$. Precisely: 
\begin{itemize}
\item if~$\beta_j$ is non-zero, we take a direct sum of lines over~$E_j$ that splits the lattice sequence~$\Lambda^j$ as in~\cite[\S5.3~Lemma]{BK2}. 
\item If~$\beta_j=0$, the reductive quotient of the maximal parahoric subgroup~$P(\Lambda^j)$ is isomorphic to the direct product of at most two symplectic groups over~$k_F$, whence a decomposition of~$V^j$ as an orthogonal sum of at most two symplectic spaces satisfying the conditions required. 
\end{itemize}
The action of~$\GL_{m_W}(k_W)$ on~$\mathfrak X$ then decomposes as a direct sum over~$j, s$ of actions on 
\[
\mathfrak X^{j,s} = \mathfrak J^1_{\mathfrak M} \cap \mathbf Y^{j,s} / \mathfrak H^1_{\mathfrak M} \cap \mathbf Y^{j,s},
\]
where~$\mathbf Y^{j,s} = \Hom_F( Y^{j,s}, W)$. 

Using~\cite[Proposition~3.4]{S4} and~\cite[\S5.3~Corollary]{BK2}, we may choose the elements~$\gamma_v$ for~\eqref{eqn:induction} so that the decomposition~$X = \perp_{j,s} Y^{j,s} \perp (W \oplus W^\ast)$ is subordinate to all strata considered. The action of~$\GL_{m_W}(k_W)$ then decomposes further along~\eqref{eqn:induction} into pieces that fit the hypotheses of~\ref{lem:computed}~Lemma, namely pieces of the following forms: 
\begin{align*}
Q_1 &= [\mathfrak a^{n/2}_{\mathfrak M} \cap \mathbf Y^{j,s} / \mathfrak a^{(n/2)+}_{\mathfrak M} \cap\mathbf Y^{j,s} ] , \\
Q_2 &=[\mathfrak b_{\gamma_0,_{\mathfrak M} }^{r_0/2} \cap \mathbf Y^{j,s} / \mathfrak b_{\gamma_0,_{\mathfrak M} }^{(r_0/2)+} \cap \mathbf Y^{j,s} ], \\
Q_3 &=[\mathfrak b_{\gamma_v,_{\mathfrak M} }^{r_{v-1}/2} \cap \mathbf Y^{j,s} / \mathfrak b_{\gamma_v,_{\mathfrak M} }^{(r_{v-1}/2)+}\cap \mathbf Y^{j,s} ] - [\mathfrak b_{\gamma_v,_{\mathfrak M} }^{r_{v }/2} \cap \mathbf Y^{j,s} / \mathfrak b_{\gamma_v,_{\mathfrak M} }^{(r_{v}/2)+}\cap \mathbf Y^{j,s} ] . 
\end{align*}

\subsection{}\label{5.9new} 
At last we come to the point, which is not actually to compute the character~$\epsilon_{\mathfrak M^i} \epsilon_{\mathfrak M}$, but rather to prove that this character does not depend on the maximal self-dual order~$\mathfrak M$. In our setting there are exactly two choices for~$\mathfrak M$, with a given period~$e$ and duality invariant~$d=1$. Indeed, the lattice chain underlying the self-dual lattice sequence~$\Lambda_X \cap (W \oplus W^{\ast})$ is the disjoint union of two self-dual lattice chains, one containing a self-dual lattice and its multiples, the other containing a non-self-dual lattice (whose dual is~$\pEW$ times it) and its multiples. Let~$\mathfrak M_0$ and~$\mathfrak M_1$ be the two possible choices and write~$ \Lambda_{W,0}=\mathfrak M_0\cap W$ and~$ \Lambda_{W,1}=\mathfrak M_1 \cap W$. According to~\cite[Lemma~6.7]{S5}, the sets of jumps of~$ \Lambda_{W,0}$ and~$ \Lambda_{W,1}$ are~$\frac{e}{e_W} \mathbb Z$ and~$\frac{e}{2e_W} + \frac{e}{e_W} \mathbb Z$ respectively, and all results in paragraphs~\ref{5.5new}--\ref{5.7new} apply. 

We can thus compare~$\epsilon_{\mathfrak M_{0}^i} \epsilon_{\mathfrak M_{0}}$ and 
$ \epsilon_{\mathfrak M_{1}^i} \epsilon_{\mathfrak M_{1}}$ term by term.

\subsubsection*{{\bf Term~$\bs Q_{\bs 1}$}} 
We apply paragraphs~\ref{5.5new}--\ref{5.7new}, replacing~$Y$ by~$Y^{j,s}$, $\Lambda_Y$ by~$\Lambda\cap Y^{j,s}$ and using~$\Lambda_{W,t}=\mathfrak M_t\cap W$ as above, for~$t=0,1$. We remark that~$\dim_F Y^{j,s}$ is always even. Hence, by~\ref{lem:computed}~Lemma, the signature on~$Q_1$ is trivial unless~$\mathcal C_0$ and~$\mathcal C_1$ have the same jumps, so give the same signature. 

\subsubsection*{{\bf Term~$\bs Q_{\bs 2}$}}
We actually have~$\gamma_0 = \beta$, hence this term is zero if the centralizer of~$\beta$ does not intersect~$\Hom_F(V^j, W)$. This condition holds under the assumptions of~\ref{prop:mainprop}~Proposition because~$j\ne i$. 

\subsubsection*{{\bf Term~$\bs Q_{\bs 3}$}}
Since~$\mathfrak M_0$ and~$\mathfrak M_1$ have the same intersection with~$V$, we may and do choose the same sequence~$(\gamma_v, r_v)$ for both. We may also scale all our lattice sequences to make the period big enough so that all numbers~$r_v / 2$ are integers. Now~$Q_3$ is zero unless the centralizer of~$\gamma_v$ intersects~$\Hom_F(V^j, W)$, which we now assume. We then apply paragraphs~\ref{5.5new}--\ref{5.7new} over~$F[\gamma_v]$. 

If the lattice sequences~$\mathcal C_0$ and~$\mathcal C_1$ have the same jumps we have the equality we want. Otherwise, they are shifted by half a period (\ref{lem:samejumps}~Lemma) and the integer~$d$ given by~\ref{lem:computed}~Lemma is equal to~$\dim_{F[\gamma_v]} Y^{j,s}$. If~$d$ is even we are also done. Otherwise we have~$\beta_j \ne 0$ and~$s_\L = e/e_j$, and the period of~$\mathcal C_0$ and~$\mathcal C_1$ is~${e}/{\lcm (e_W, e_j)}$. 

Since~$Q_3$ is the difference of two terms~$[\mathcal C_t(a)] -[\mathcal C_t(b)]$ in the lattice sequence~$\mathcal C_t$, for~$t=0,1$, over~$F[\gamma_v]$, the values of~$Q_3$ for~$t=0$ and~$t=1$ will be the same on condition that the difference~$a-b$ is a multiple of half the period. This is what we will now prove. 

In the notation of~\eqref{eqn:induction}, we let~$h\ge 1$ be the smallest integer such that the centralizer of~$\gamma_h$ intersects~$\Hom_F(V^j, W)$, so that we only need to consider terms with~$v\ge h$. If~$h = u$ there is nothing to do. Otherwise, we need to examine the values of~$r_h$ and~$r_{h-1}$ more closely, in terms of the \emph{normalized critical exponents}~$k_0^F[\gamma_v]$ (see~\cite[pp.~129,141--2]{S4}). We use~\cite[p.~141]{S4}, case~(ii) for~$r_{h-1}= -k_0(\gamma_{h-1}, \mathfrak M)$ (the unnormalized critical exponent relative to~$\mathfrak M$) and case~(i) for~$r_{h}= -k_0(\gamma_{h}, \mathfrak M)$, to get
\[
r_{h-1} = v_{F[\gamma_h ]}(c) \frac{e(\mathfrak M|\oF)}{e(F[\gamma_h]/F)} , 
\qquad 
r_{h}= -k_0^F(\gamma_h) \frac{e(\mathfrak M|\oF)}{e(F[\gamma_h]/F)} , 
\]
for some element~$c$ in~$F[\gamma_h]$, so that
\[
\frac{r_{h-1}}{2} - \frac{r_{h}}{2} =
\frac{e(\mathfrak M|{\mathfrak o}_{F[\gamma_h]})}{2} \left( 
v_{F[\gamma_h ]}(c) + k_0^F(\gamma_h ) 
\right).
\]
This is indeed an integer multiple of the half-period of jumps
\[
\frac{e(\mathfrak M|{\mathfrak o}_{F[\gamma_h]})}{2} 
\left(\frac{\lcm (e_W, e_j)}{e(F[\gamma_h ]/F)}\right)^{-1} 
\]
since the last term is the inverse of an integer. 

For~$v > h$ we use \cite[p.~141]{S4} case~(i) again and get: 
\[
\frac{r_{v-1}}{2} - \frac{r_{v}}{2} = 
\frac{e(\mathfrak M|{\mathfrak o}_{F[\gamma_v ]})}{2} 
\left( -k_0^F(\gamma_{v-1}) 
\frac{e(F[\gamma_v]/F)}{e(F[\gamma_{v-1} ]/F)} 
+ k_0^F(\gamma_{v}) 
\right)
\]
This is a multiple of the half-period of jumps if and only if 
\begin{align*}
& \frac{ -k_0^F(\gamma_{v-1}) e(F[\gamma_{v} ]/F) 
+ k_0^F(\gamma_{v})e(F[\gamma_{v-1} ]/F) }{
e(F[\gamma_{v-1} ]/F)} \ 
\frac{ \lcm (e_W, e_j) }{
e(F[\gamma_{v} ]/F) }
\\
& \qquad = \left( -k_0^F(\gamma_{v-1}) + k_0^F(\gamma_{v}) \frac{ 
e(F[\gamma_{v-1} ]/F) }{
e(F[\gamma_{v} ]/F)} \right) \ 
\frac{ \lcm (e_W, e_j) }{
e(F[\gamma_{v-1} ]/F) }
\end{align*}
is an integer, which is the case because~$e(F[\gamma_{v} ]/F)$ divides~$e(F[\gamma_{v-1} ]/F)$ (see~\cite[2.4.1]{BK1}). 

\medskip

Putting this together, we obtain the character~$\epsilon_{\mathfrak M^i} \epsilon_{\mathfrak M}$ as a product of signatures, each of them only depending on~$e_W$ by~\ref{lem:computed}~Lemma, hence our character only depends on~$e_W$, not on~$W$. Furthermore, the extension~$E_W$ is isomorphic to~$F[\beta_i]$, hence~$e_W$ is equal to~$e(F[\beta_i]/F)$, independent of the choice of~$W$. This completes the proof of~\ref{prop:mainprop}~Proposition, whence of~\ref{thm:reductiondetail}~Theorem.

\section{The simple case}\label{JBN6new}
 
In this section we prove~\ref{thm:simple}~Theorem. Recalling that, by~\ref{thm:jt}~Theorem, the parameters of the Hecke algebra of our cover are those in the Hecke algebra of a finite reductive group, we are required to analyse these Hecke algebras. Fortunately, these are described by the work of Lusztig~\cite{Lbook} and have been computed in our cases in~\cite{LS}. One subtlety is that the twisting characters~$\epsilon_{\mathfrak M_{t}}$ give rise to involutions which we have not computed explicitly so remain unknown. Fortunately, the numerics are such that an exact description of these involutions is not needed.

\subsection{}\label{6.1new} 
Let~$\pi$ be a simple cuspidal representation of~$\SpFV$ in the sense of paragraph~\ref{2.6new}. Since the case of depth zero representations is already dealt with in~\cite{LS}, we assume moreover that~$\pi$ has positive depth. Thus~$\pi$ contains a skew \emph{simple} character~$\theta$ of~$H^1=H^1(\beta,\Lambda)$, for some maximal skew simple stratum~$[\Lambda,-,0,\beta]$, with~$\beta\ne 0$, and~$E=F[\beta]$ is a field. We write~$\bs\Theta$ for the endo-class of the unique self-dual simple character~$\tilde\theta$ which restricts to~$\theta$. We retain all the notation of paragraph~\ref{5.1new} so interpret simplicity as meaning that~$l=1$ and drop the index~$1$ for notation. We will be considering the space~$X=X^1$, while varying the self-dual cuspidal representation~$\rho$ of~$\GL_F(W)$ (and the space~$W$). Note that we have~$E_W\simeq E$ so we will identify them.

For a self-dual cuspidal representation~$\rho$ of some~$\GL_F(W)$, recall that we write~$\deg(\rho)=\dim_F(W)$ and~$s_\pi(\rho)$ for the unique non-negative real number such that the normalized induced representation~$\nu^s \rho \times \pi$ is reducible. Then the description of the Jordan set in paragraph~\ref{2.1new} shows that, in order to prove~\ref{thm:simple}~Theorem, the equality we must prove is
\begin{equation}\label{eqn:needsum}
\sum_{\rho} \left\lfloor s_\pi(\rho)^2\right\rfloor \deg(\rho) = 2N,
\end{equation}
where the sum runs over all self-dual cuspidal representations~$\rho$ with endo-class~$\bs\Theta(\rho)=\bs\Theta^2$.

\subsection{}\label{6.2new} 
Recall that we have~$\pi=\cInd_J^G\lambda$, with~$\lambda=\underline\kappa \otimes \tau$ and that~$\rho$ contains the maximal simple type~$\tilde \lambda_W= \underline{\tilde \kappa}_W\otimes \tilde \tau_W$ and has unramified twist number~$t(\rho)={\frac{\dim_F W}{e(\bs\Theta)}}$, where we have written~$e(\bs\Theta)=e(\bs\Theta^2)=e(E/F)$ since it depends only on the endo-class. Moreover, by~\ref{prop:realpartsprop}~Proposition, we have that the real parts of the reducibility points of the normalized induced representation~$\nu^s \rho \times \pi$ are the elements of the set 
\[
\left\{ \pm \frac{r_0 + r_1}{2t(\rho)} , \pm \frac{r_0 - r_1}{2t(\rho)}\right\}, 
\]
where, for~$t=0,1$, the integers~$r_t= r_t (\epsilon_{\mathfrak M_{t}} (\tilde \tau_W\otimes \tau))$ comes from the quadratic relations in the finite Hecke algebra~$\mathscr H(\HP ( \mathfrak M_{t ,\oE} ), \epsilon_{\mathfrak M_{t}} (\tilde\tau_W \otimes\tau))$ as in~\eqref{eqn:emphasis}. 

\begin{Remark}
It will be crucial to note that the character~$\epsilon_{\mathfrak M_{t}}$ depends only on the dimension~$\deg(\rho)=\dim_F(W)$, and not on the representation~$\rho$ itself.
\end{Remark}

The contribution to the sum~\eqref{eqn:needsum} of the inertial class~$[\rho]$ (that is, writing~$\rho'=\nu^{\pi{\mathrm i}/t(\rho)\log(q)}\rho$ for the other self-dual representation in the inertial class, the combined contributions of~$\rho$ and~$\rho'$) is
\[
\left\lfloor \left(\frac{r_0 + r_1}{2t(\rho)}\right)^2\right\rfloor + \left\lfloor \left(\frac{r_0 - r_1}{2t(\rho)}\right)^2\right\rfloor.
\]
From results of Lusztig (see~\cite[\S8]{LS} and also paragraph~\ref{6.5new} below), the numbers~$r_t/t(\rho)$ are either both integers or both half-integers so that this simplifies to
\begin{equation}\label{eqn:suminertial}
\left\lfloor \frac{r_0^2 + r_1^2}{2t(\rho)^2}\right\rfloor.
\end{equation}

\subsection{}\label{6.2bnew} 
In order to prove~\eqref{eqn:needsum} we will need to recall Lusztig's parametrisation of cuspidal representations of classical groups, and the computation of the parameter~$r_t$ in the Hecke algebra~$\mathscr H(\HP(\mathfrak M_{t ,\oE} ), \epsilon_{\mathfrak M_{t}} (\tilde\tau_W \otimes\tau))$. We follow the description in~\cite[\S\S$2,3,6$ and, especially,~$7$]{LS}, to which we refer for details and references for the assertions made here.

In almost all cases, we have
\[
\mathscr H(\HP(\mathfrak M_{t ,\oE} ), \epsilon_{\mathfrak M_{t}} (\tilde\tau_W \otimes\tau)) \simeq
\mathscr H(\HP^\so(\mathfrak M_{t ,\oE} ), \epsilon_{\mathfrak M_{t}} (\tilde\tau_W \otimes\tau^\so)),
\]
where~$\tau^\so$ is an irreducible component of the restriction~$\tau_{|\HP^\so(\Lambda_{\oE})}$, and it is here that we will perform our calculations. In the exceptional cases we have~$r_t=0$ and it will turn out that this matches the formula one would obtain by following the recipe for computing the parameters in the connected component~$\HP^\so(\mathfrak M_{t ,\oE})$. Thus we will assume first that the calculation is to be done in~$\HP^\so(\mathfrak M_{t ,\oE})$ and then, in paragraph~\ref{6.5bnew}, we will treat the exceptional cases.

\subsection{}\label{6.3new} 
Since~$P(\Lambda_{\oE})$ is the normaliser of a maximal parahoric subgroup of the centraliser~$G_E$, we can decompose
\[
\HP^\so(\Lambda_{\oE}) = \HP^{(0)}(\Lambda_{\oE})\times \HP^{(1)}(\Lambda_{\oE})
\]
as a product of two connected classical groups over~$k_E^\so$ (the residue field of the fixed points~$E_\so$ in~$E$ under the involution on~$A$). We have a similar decomposition of~$\HP^\so(\mathfrak M_{t ,\oE})$ with, moreover,
\[
\HP^{(1)}(\mathfrak M_{0 ,\oE}) = \HP^{(1)}(\Lambda_{\oE}) \text{ and }
\HP^{(0)}(\mathfrak M_{1 ,\oE}) = \HP^{(0)}(\Lambda_{\oE}),
\]
and the Levi subgroup
\[
\tilde\HP(\Lambda_{W,\oE})\times \HP^{(t)}(\Lambda_{\oE}) \subseteq \HP^{(t)}(\mathfrak M_{t ,\oE}).
\]
We choose an irreducible component~$\tau^\so$ of the restriction~$\tau_{|\HP^\so(\Lambda_{\oE})}$ and write it as~$\tau^{(0)}\otimes\tau^{(1)}$. Writing the character~$\epsilon_{\mathfrak M_{t}}$ as~$\epsilon_{\mathfrak M_{t}}^W\otimes \epsilon_{\mathfrak M_{t}}^{(0)}\otimes\epsilon_{\mathfrak M_{t}}^{(1)}$, we have isomorphisms of Hecke algebras
\[
\mathscr H(\HP^\so(\mathfrak M_{t ,\oE} ), \epsilon_{\mathfrak M_{t}} (\tilde\tau_W \otimes\tau^\so)) \simeq
\mathscr H( \HP^{(t)}(\mathfrak M_{t ,\oE}), \epsilon_{\mathfrak M_{t}}^W\tilde\tau_W\otimes \epsilon_{\mathfrak M_{t}}^{(t)}\tau^{(t)}),
\]
and it is in this Hecke algebra that we compute the parameter~$r_t$.

\subsection{}\label{6.4new} 
We now fix~$t=0$ or~$1$ so drop the sub/superscript~$t$ from our notations for now. Thus we have:
\begin{itemize}
\item a connected classical group~$\HP^\so(\mathfrak M_{\oE})$ over~$k_E^\so$, with Levi subgroup~$\tilde\HP(\Lambda_{W,\oE})\times \HP(\Lambda_{\oE})$ and~$\tilde\HP(\Lambda_{W,\oE})\simeq\GL_m(k_E)$, where~$m=\dim_E(W)$;
\item a self-dual cuspidal representation~$\tilde\tau_W\otimes\tau$ of~$\tilde\HP(\Lambda_{W,\oE})\times \HP(\Lambda_{\oE})$;
\item a character~$\epsilon_{\mathfrak M}^W\otimes \epsilon_{\mathfrak M}$ of~$\tilde\HP(\Lambda_{W,\oE})\times \HP(\Lambda_{\oE})$ of order at most two, which depends on~$m=\dim_E(W)$ but not on~$\tilde\tau_W$.
\end{itemize}
By Green's parametrization (and after fixing an isomorphism~$\tilde\HP(\Lambda_{W,\oE})\simeq\GL_m(k_E)$), the cuspidal representation~$\tilde\tau_W$ corresponds to an irreducible monic polynomial~$Q\in k_E[X]$ of degree~$m$; moreover, this polynomial is~$k_E/k_E^\so$-self-dual, that is
\[
Q(X)=(\overline{Q(0)})^{-1} X^{\deg Q} \overline Q(1/X),
\]
where~$x\mapsto\overline x$ is the automorphism of~$k_E$ with fixed field~$k_E^\so$, extended to~$k_E[X]$ coefficient\-wise (see~\cite[\S7.1]{LS}). Since a cuspidal representation~$\tilde\tau_W$ of~$\tilde\HP(\Lambda_{W,\oE})$ is self-dual if and only if~$\tilde\tau_W\epsilon_{\mathfrak M}^W$ is cuspidal self-dual, twisting by~$\epsilon_{\mathfrak M}^W$ induces an involution on the set of irreducible~$k_E/k_E^\so$-self-dual monic polynomials of degree~$m$. We denote this involution by~$\sigma_{m,W}$; it is either trivial, or given by~$Q(X)\mapsto (-1)^{\deg Q} Q(-X)$. 

Similarly, by Lusztig's parametrization, the cuspidal representation~$\tau$ lies in a \emph{rational Lusztig series}~$\mathcal E(\mathbf{s})$ corresponding to (the rational conjugacy class of) a semisimple element~$\mathbf{s}$ of the dual group of~$\HP(\Lambda_{\oE})$. Since its series contains a cuspidal representation, this semisimple element~$\mathbf{s}$ has characteristic polynomial of a particular form, namely
\[
P_{\mathbf{s}}(X)=\prod_P P(X)^{a_P},
\]
where the product is over all irreducible~$k_E/k_E^\so$-self-dual monic polynomials and the integers~$a_P$ satisfy certain combinatorial constraints (see~\cite[(7.2) and~\S7.7]{LS}); more precisely, we have:
\begin{itemize}
\item $\sum_{P} a_P\deg(P)$ is the dimension of the space~$\mathcal V$ on which the dual group of~$\HP(\Lambda_{\oE})$ naturally acts;
\item if either~$k_E\ne k_E^\so$ or~$P(X)\ne (X\pm 1)$, then~$a_P=\frac 12(b_P^2+b_P)$, for some non-negative integers~$b_P$;
\item if~$P(X)=X\pm 1$ then, writing~$a_+:=a_{(X-1)}$ and~$a_-:=a_{(X+1)}$, there are integers~$b_+^{},b_-^{}\ge 0$ such that
\begin{enumerate}
\item\label{i} if~$\HP(\Lambda_{\oE})$ is odd special orthogonal then~$a_+=2(b_+^2+b_+^{})$
 and~$a_-=2(b_-^2+b_-^{})$,
\item\label{ii} if~$\HP(\Lambda_{\oE})$ is symplectic then~$a_+=2(b_+^2+b_+^{})+1$
 and~$a_-=2b_-^2$,
\item\label{iii} if~$\HP(\Lambda_{\oE})$ is even special orthogonal then~$a_+=2b_+^2$ and~$a_-=2b_-^2$,
\end{enumerate}
and, in case~\ref{iii}, the~$(\pm 1)$-eigenspace in~$\mathcal V$ is an even-dimensional orthogonal space of type~$(-1)^{b_\pm}$, and the same in case~\ref{ii} for the~$(-1)$-eigenspace only.
\end{itemize}
As above, twisting by the character~$\epsilon_{\mathfrak M}$ will induce a degree-preserving involution on the set of irreducible~$k_E/k_E^\so$-self-dual monic polynomials. If the character~$\epsilon_{\mathfrak M}$ is trivial then this involution is trivial. If the character~$\epsilon_{\mathfrak M}$ is quadratic then, by~\cite[Proposition~8.26]{CE}, twisting by~$\epsilon_{\mathfrak M}$ induces a bijection between rational Lusztig series
\[
\mathcal E(\mathbf{s}) \xrightarrow{\ \sim\ } \mathcal E(-\mathbf{s}),
\]
and the involution is given by~$P(X)\mapsto (-1)^{\deg P} P(-X)$. In either case, we denote by~$\sigma_{m,G}$ the involution induced by twisting by~$\epsilon_{\mathfrak M}$. (Note that this is a degree-preserving involution on the set of \emph{all} irreducible~$k_E/k_E^\so$-self-dual monic polynomials; the subscript~$m$ is included to indicate that the involution depends on~$m$.) The characteristic polynomial corresponding to the cuspidal representation~$\tau\epsilon_{\mathfrak M}$ is then
\[
\prod_P P(X)^{a_{\sigma_{m,G}(P)}}.
\]

Putting together our two involutions, we get an involution on the set of irreducible~$k_E/k_E^\so$-self-dual monic polynomials of degree~$m$ given by
\[
\sigma_m = \sigma_{m,G}\circ\sigma_{m,W}.
\]

\subsection{}\label{6.5new} 
Recall that the Hecke algebra~$\mathscr H( \HP(\mathfrak M_{\oE}), \epsilon_{\mathfrak M}^W\tilde\tau_W\otimes \epsilon_{\mathfrak M}\tau)$ is generated by an element~$\mathcal T$ satisfying a quadratic relation
\[
(\mathcal T-q^r\omega)(\mathcal T+\omega)=0,
\]
where~$q$ is the cardinality of the residue field of~$k_F$. The work of Lusztig, explicated in~\cite[\S7]{LS}, allows one to write down explicitly the parameter~$r$ in terms of the characteristic polynomials of the previous paragraph, as follows.

Let~$Q(X)$ be the irreducible~$k_E/k_E^\so$-self-dual monic polynomial of degree~$m$ corresponding to~$\tilde\tau_W$, and let~$P_s(X)=\prod_{P} P(X)^{a_P}$ be the monic polynomial corresponding to~$\tau$, where the~$a_P$ are as described in the previous paragraph. Writing~$f$ for the degree of the extension~$k_E/k_F$, one gets the following values:
\begin{itemize}
\item if~$k_E=k_E^\so$ and~$\sigma_1(Q)=X-1$ then 
\[
\frac rf = \begin{cases}  
2b_{+}&\text{ if~$\HP$ is even special orthogonal,} \\
2b_{+}+1 &\text{ otherwise;}
\end{cases}
\]
\item if~$k_E=k_E^\so$ and~$\sigma_1(Q)=X+1$ then 
\[
\frac rf = \begin{cases} 2b_{-}+1 &\text{ if~$\HP$ is odd special orthogonal,} \\ 
2b_{-}&\text{ otherwise;} 
\end{cases}
\]
\item if~$k_E\ne k_E^\so$ or~$m$ is even then
\[
\frac rf=(2b_{\sigma_m(Q)}+1)\frac{m}2.
\]
\end{itemize}
Note that, since~$t(\rho)=mf$, the number~$r/t(\rho)$ is a half-integer, as asserted above. Moreover,~$r/t(\rho)$ is an integer precisely when~$E/E_\so$ is ramified and~$E/F$ is a maximal extension (i.e. of degree~$\dim_F(W)$); in particular, this depends only on the polynomial~$Q$ (that is, on~$\tilde\tau_W$, so on the representation~$\rho$) and not on either the representation~$\tau$ or on the involution~$\sigma_1$.

\subsection{}\label{6.5bnew}
In this paragraph, we treat the exceptional cases, where we do \emph{not} have an isomorphism
\begin{equation} \label{eqn:isoH}
\mathscr H(\HP(\mathfrak M_{t ,\oE} ), \epsilon_{\mathfrak M_{t}} (\tilde\tau_W \otimes\tau)) \simeq
\mathscr H(\HP^\so(\mathfrak M_{t ,\oE} ), \epsilon_{\mathfrak M_{t}} (\tilde\tau_W \otimes\tau^\so)),
\end{equation}
According to the description in~\cite[\S6.3]{MS}, this occurs precisely when:
\begin{itemize}
\item $E/E_\so$ is ramified;
\item $\dim_E(W)=1$, so that~$\tilde\tau_W$ is a character of order at most~$2$;
\item and either~$\HP(\Lambda_{\oE})=\HP^\so(\Lambda_{\oE})$ or~$\epsilon_{\mathfrak M_{t}}\tau_{|\HP^\so(\Lambda_{\oE})}$ is reducible.
\end{itemize}
We remark that~$\epsilon_{\mathfrak M_{t}}\tau_{|\HP^\so(\Lambda_{\oE})}$ is reducible if and only if~$\tau_{|\HP^\so(\Lambda_{\oE})}$ is reducible.

In these cases, writing~$\HP^\so(\mathfrak M_{t ,\oE}) = \HP^{(0)}(\Lambda_{\oE})\times \HP^{(1)}(\Lambda_{\oE})$, there is one value of~$t$ for which~$\HP^{(t)}(\Lambda_{\oE})$ is an even special orthogonal group (for the other it is a symplectic group) and it is precisely for this value of~$t$ that we do not have an isomorphism~\eqref{eqn:isoH} and we get parameter~$r_t=1$.

As above, we write~$\tau^{(0)}\otimes\tau^{(1)}$ for an irreducible component of~$\tau_{|\HP^\so(\Lambda_{\oE})}$, and write~$\epsilon_{\mathfrak M_{t}}$ as~$\epsilon_{\mathfrak M_{t}}^W\otimes \epsilon_{\mathfrak M_{t}}^{(0)}\otimes\epsilon_{\mathfrak M_{t}}^{(1)}$. Writing~$P_s(X)=\prod P(X)^{a_P}$ for the polynomial corresponding to the cuspidal representation~$\tau^{(t)}$, the fact that it does \emph{not} extend to the full even orthogonal group implies, by~\cite[Proposition~7.9]{LS}, that~$\pm 1$ are \emph{not} roots of~$P_s$, that is,~$a_+=a_-=0$. 

Since~$\tilde\tau_W$ is a character of order at most~$2$, the corresponding polynomial is~$Q(X)=X\pm 1$. In particular, since we have~$b_+=b_-=0$, the formulae of paragraph~\ref{6.5new} are still valid, since they too give~$r_t=0$. Thus those formulae are valid in every case.

\subsection{}\label{6.6new} 
Finally, using the formulae of paragraph~\ref{6.5new}, we return to computing the contribution~\eqref{eqn:suminertial}, so we retrieve the sub/super\-scripts~$t$. We have:
\begin{itemize}
\item an irreducible~$k_E/k_E^\so$-self-dual monic polynomial~$Q(X)$, corresponding to the cuspidal representation~$\tilde\tau_W$;
\item for~$t=0,1$, a polynomial~$\prod_P P(X)^{a^{(t)}_P}$ corresponding to the cuspidal representation~$\tau^{(t)}$;
\item for~$t=0,1$, an involution~$\sigma_m^{(t)}$ on the set of irreducible~$k_E/k_E^\so$-self-dual monic polynomials of degree~$m$.
\end{itemize}
Suppose first that either~$k_E\ne k_E^\so$ or~$m$ is even; then we get
\begin{align*}
\left\lfloor \frac{r_0^2 + r_1^2}{2t(\rho)^2}\right\rfloor &= 
\left\lfloor \frac{(2b_{\sigma_m^{(0)}(Q)}^{(0)}+1)^2 + (2b_{\sigma_m^{(1)}(Q)}^{(1)}+1)^2}{8} \right\rfloor  \\
&= \left\lfloor \tfrac 12 b_{\sigma_m^{(0)}(Q)}^{(0)} \left(b_{\sigma_m^{(0)}(Q)}^{(0)}+1\right) + \tfrac 12 b_{\sigma_m^{(1)}(Q)}^{(1)} \left(b_{\sigma_m^{(1)}(Q)}^{(1)}+1\right) +\tfrac 12 \right\rfloor \\
&= a_{\sigma_m^{(0)}(Q)}^{(0)}+a_{\sigma_m^{(1)}(Q)}^{(1)}.
\end{align*}
If~$k_E=k_E^\so$ then one of the groups~$\HP^{(t)}(\mathfrak M_{t ,\oE})$ is symplectic while the other is orthogonal. Here we can treat each case, each polynomial~$X\pm 1$, and each possibility for the involutions~$\sigma_1^{(t)}$, separately. Up to permuting~$\{0,1\}$ we are in one of the following two cases:
 
\textbf{If~$\HP^{(0)}(\Lambda_{\oE})$ is odd special orthogonal and~$\HP^{(1)}(\Lambda_{\oE})$ is symplectic}, then the contribution of~$\sigma_1^{(1)}(X - 1)$ is
\[
\left\lfloor \frac{ \left(2b_\zeta^{(0)}+1\right)^2 +\left(2b_+^{(1)}+1\right)^2}2\right\rfloor 
= 2 b_\zeta^{(0)}\left(b_\zeta^{(0)}+1\right) + 2 b_+^{(1)} \left( b_+^{(1)}+1 \right)+1 
= a_\zeta^{(0)} + a_+^{(1)},
\]
where~$\zeta$ is the sign defined by~$\sigma_1^{(0)}\sigma_1^{(1)}(X-1)=X-\zeta$; and the contribution of~$\sigma_1^{(1)}(X + 1)$ is
\[
\left\lfloor \frac{ \left(2b_{-\zeta}^{(0)}+1\right)^2 +\left(2b_-^{(1)}\right)^2}2\right\rfloor 
= 2 b_{-\zeta}^{(0)}\left(b_{-\zeta}^{(0)}+1\right) + 2 \left( b_-^{(1)} \right)^2
= a_{-\zeta}^{(0)} + a_-^{(1)}.
\]
In particular, the sum of the contributions of~$X\pm 1$ is
\[
a_+^{(0)}+a_-^{(0)} + a_+^{(1)}+a_-^{(1)}.
\]

\textbf{If~$\HP^{(0)}(\Lambda_{\oE})$ is even special orthogonal and~$\HP^{(1)}(\Lambda_{\oE})$ is symplectic}, then the contribution of~$\sigma_1^{(1)}(X - 1)$ is
\[
\left\lfloor \frac{ \left(2b_\zeta^{(0)}\right)^2 +\left(2b_+^{(1)}+1\right)^2}2\right\rfloor 
= 2 \left(b_\zeta^{(0)}\right)^2 + 2 b_+^{(1)} \left( b_+^{(1)}+1 \right)
= a_\zeta^{(0)} + a_+^{(1)} - 1,
\]
where~$\zeta$ is again the sign defined by~$\sigma_1^{(0)}\sigma_1^{(1)}(X-1)=X-\zeta$; and the contribution of~$\sigma_1^{(1)}(X + 1)$ is
\[
\left\lfloor \frac{ \left(2b_{-\zeta}^{(0)}\right)^2 +\left(2b_-^{(1)}\right)^2}2\right\rfloor 
= 2 \left(b_{-\zeta}^{(0)}\right)^2 + 2 \left( b_-^{(1)} \right)^2
= a_{-\zeta}^{(0)} + a_-^{(1)}.
\]
In this second case, the sum of the contributions of~$X\pm 1$ is
\[
a_+^{(0)}+a_-^{(0)} + a_+^{(1)}+a_-^{(1)} - 1; 
\]
the term~$-1$ reflects the fact that the sum of the dimensions of the spaces on which the dual groups of~$\HP^{(t)}(\Lambda_{\oE})$ act naturally is~$1$ more than the sum of the dimensions of the spaces on which the groups~$\HP^{(t)}(\Lambda_{\oE})$ act naturally. Note also that this latter sum of dimensions is precisely~$\dim_E(V)$, where we recall that~$V$ is the symplectic space on which our group~$\SpFV$ acts.

\subsection{}\label{6.7new} 
Having computed all the contributions to the sum~\eqref{eqn:needsum} in the previous paragraph, we can now sum them over all possible~$Q$, noting that, if the cuspidal representation~$\rho$ corresponds to the polynomial~$Q$, then~$\deg(\rho)=[E:F]\deg(Q)$. If~$k_E\ne k_E^\so$ this is straightforward and we obtain
\begin{align*}
\sum_{\rho} \left\lfloor s_\pi(\rho)^2\right\rfloor \deg(\rho) &= 
[E:F]\sum_m \left( m\sum_{\deg(Q)=m} \left(a_{\sigma_m^{(0)}(Q)}^{(0)}+a_{\sigma_m^{(1)}(Q)}^{(1)}\right)\right) \\
& = [E:F]\left( \sum_P a_P^{(0)}\deg(P) + \sum_P a_P^{(1)}\deg(P)\right)
= [E:F]\dim_E(V) = 2N,
\end{align*}
as required. Here the penultimate equality occurs because each group~$\HP^{(t)}(\Lambda_{\oE})$ is a unitary group (whose dual group is then a unitary group acting naturally on a space of the same dimension), and the sum of the dimensions of the spaces on which they act is~$\dim_E(V)$.

If~$k_E=k_E^\so$ then we need to be a little more careful with the polynomials~$X\pm 1$ (that is, the~$k_E/k_E^\so$-self-dual monic polynomials of degree~$1$), as described at the end of the previous paragraph. If one of the~$\HP^{(t)}(\Lambda_{\oE})$ is \emph{even} special orthogonal (and the other symplectic) then we get that~$\sum_{\rho} \left\lfloor s_\pi(\rho)^2\right\rfloor \deg(\rho)$ is
\begin{align*}
[E:F]\left(\sum_{m\ge 2} \left( m\sum_{\deg(Q)=m} \left(a_{\sigma_m^{(0)}(Q)}^{(0)}+a_{\sigma_m^{(1)}(Q)}^{(1)}\right)\right) + a_+^{(0)}+a_-^{(0)} + a_+^{(1)}+a_-^{(1)} - 1 \right)
\\
 = [E:F]\left( \sum_P a_P^{(0)}\deg(P) + \sum_P a_P^{(1)}\deg(P)-1\right)
= [E:F]\dim_E(V) = 2N,
\end{align*}
where the penultimate equality uses the fact that the dual of a symplectic group acts naturally on a space of dimension~$1$ greater, while the dual of an even special orthogonal group acts naturally on a space of the same dimension.

On the other hand, if one of the~$\HP^{(t)}(\Lambda_{\oE})$ is \emph{odd} special orthogonal (and the other symplectic) then we get the same sum except without the term~$-1$, and the penultimate equality uses the fact that the dual of an odd special orthogonal group acts naturally on a space of dimension~$1$ smaller, while the dual of a symplectic group acts naturally on a space of dimension~$1$ greater.

This completes the proof of~\eqref{eqn:needsum}, whence of~\ref{thm:simple}~Theorem.

\subsection{}\label{6.8new} 
The results in this section not only prove~\ref{thm:simple}~Theorem but also give an algorithm to compute the inertial Jordan set of a positive depth simple cuspidal representation of~$G$. (The case of depth zero is treated already in~\cite{LS}.) Moreover,~\ref{cor:together}~Corollary then gives the inertial Jordan set for any cuspidal representation of~$G$.

Indeed, suppose~$\pi$ is a simple cuspidal representation of~$G$, induced from a cuspidal type~$\lambda=\underline\kappa\otimes\tau$. With the usual notation, let~$\tau^\so$ be any irreducible component of the restriction of~$\tau$ to the maximal parahoric subgroup~$P^\so(\Lambda_{\oE})$. Then~$\tau^\so$ is the inflation of a representation~$\tau^{(0)}\otimes\tau^{(1)}$, with each~$\tau^{(t)}$ a cuspidal representation of a finite reductive group over~$k_E^\so$. These each appear in some rational Lusztig series and we consider the set~$\mathcal Q^{(t)}$ of monic irreducible polynomials dividing the characteristic polynomial (over~$k_E$) of the corresponding semisimple conjugacy class, for~$t=0,1$, all of which are~$k_E/k_E^\so$-self-dual. For each~$m\in\deg(\mathcal Q^{(t)})$, we compute the signature character~$\epsilon_{\mathfrak M_t}$, and thus deduce the involution~$\sigma_m^{(t)}$ as in paragraph~\ref{6.4new}. We set
\[
\mathcal Q=\left\{\sigma_m^{(t)}(Q)\mid Q\in\mathcal Q^{(t)},\ \deg(Q)=m,\ t=0,1\right\}.
\]

Now let~$\bs\Theta$ be the endo-class of the self-dual simple character lifting any skew simple character in~$\pi$ and let~$Q\in\mathcal Q$. We put~$n=\deg(Q)\deg(\bs\Theta)$ and let~$\tilde\theta$ be the unique (up to conjugacy) \emph{m-simple} character in~$\GL_n(F)$ with endo-class~$\bs\Theta^2$ (in the language of~\cite{BHMemoir}, for example). Let~$\underline{\tilde\kappa}$ be the~$p$-primary extension of~$\tilde\theta$, a representation of a group~$\tilde J$. The group~$\tilde J/\tilde J^1$ is then a finite general linear group of rank~$\deg(Q)$ over~$k_E$, and we let~$\tilde\tau_Q$ be the unique cuspidal representation in the Lusztig series corresponding to a semisimple conjugacy class with characteristic polynomial~$Q$. Write~$[\rho_Q]$ for the inertial class of cuspidal representations of~$\GL_n(F)$ containing~$\underline{\tilde\kappa}\otimes\tilde\tau_Q$. 

The inertial classes in~$\left\{[\rho_Q]\mid Q\in\mathcal Q\right\}$ are precisely the inertial classes which will appear in~$\IJord(\pi)$. In order to compute the multiplicities with which~$[\rho_Q]$ appears, we follow the recipe of paragraph~\ref{6.5new} to compute the corresponding Hecke algebra parameters~$r_0$ and~$r_1$, whence the real parts of the reducibility points~$|r_0\pm r_1|/(\deg(Q)[k_E:k_F])$ and the multiplicities from M\oe glin's criterion. In the case that~$k_E\ne k_E^\so$ or~$m=\deg(Q)>1$, this is straightforward, with the real parts of the reducibility points given by
\[
\frac{b_{\sigma_m^{(0)}(Q)}^{(0)}+b_{\sigma_m^{(1)}(Q)}^{(1)}+1}2 \quad\text{ and }\quad 
\frac{b_{\sigma_m^{(0)}(Q)}^{(0)}-b_{\sigma_m^{(1)}(Q)}^{(1)}}2,
\]
where~$a_P^{(t)}=\frac 12b_P^{(t)}(b_P^{(t)}+1)$ is the power to which~$P$ divides the characteristic polynomial corresponding to~$\tau^{(t)}$. By construction of~$\mathcal Q$, the first of these is certainly greater than~$\frac 12$. In the case~$k_E=k_E^\so$ and~$\deg(Q)=1$ (so that~$Q$ is~$X\pm 1$) there is no such simple universal formula, and instead one must proceed in a case-by-case analysis as in paragraph~\ref{6.6new}. We leave this as an exercise to the reader; a similar calculation is done in~\cite[section~8]{LS}.

\section{Galois parameters}\label{JBN3new}

In this section we study self-duality in terms of Galois parameters with a view, in particular, to understanding the ambiguities in our results in terms of the local Langlands correspondence. 

\subsection{}\label{3.3new}
We denote by~$\bar F$ a fixed separable closure of~$F$ and by~$W_F$ the absolute Weil group of~$F$ (with similar notation for intermediate fields). We would like to explore the self-dual irreducible representations~$\sigma$ of~$W_F$, with a view to determining its parity (that is, whether it is symplectic or orthogonal); in particular, we would like to know when the self-dual irreducible representation~$\sigma'$ which is an unramified twist of (and not isomorphic to)~$\sigma$ has the same parity as~$\sigma$, since it is in this case that we have ambiguity. For now, we do not require~$p$ to be odd.

Let~$\chi$ be an unramified character of~$W_F$. Then~$(\chi \sigma)^\vee$ is isomorphic to~$\chi^{-1} \sigma^\vee$, so,~$\sigma$ being self-dual,~$\chi \sigma$ is self-dual if and only if~$\chi^2 \sigma \simeq \sigma$. 

We let~$t(\sigma)$ be the number of unramified characters~$\eta$ of~$W_F$ such that~$\eta \sigma \simeq \sigma$ -- such characters form a cyclic group. We deduce that the only unramified character twist~$\sigma'$ of~$\sigma$ which is self-dual but not isomorphic to~$\sigma$ is obtained as~$\chi \sigma$, where~$\chi$ is an unramified character of order~$2t(\sigma)$. (If~$r= \val_2(t(\sigma))$, any unramified character of order~$2^{r+1}$ would do equally well.) 

Let~$E$ be the unramified extension of~$F$ in~$\bar F$ of degree~$t(\sigma)$. Then~$\sigma$ is induced from a representation~$\tau$ of~$W_E$; the restriction of~$\sigma$ to~$W_E$ is the direct sum of the conjugates of~$\tau$ under~$\Gal(E/F)$, which are pairwise inequivalent. As~$\sigma$ is self-dual,~$\tau^\vee$ is one of those conjugates. 

Assume first that~$\tau$ is self-dual -- which, we remark, is necessarily true if~$t(\sigma)$ is odd. Since~$t(\tau)=1$, the unramified twist~$\tau'$ of~$\tau$ which is self-dual but not isomorphic to~$\tau$ has the form~$\chi \tau$, where~$\chi$ is the order~$2$ unramified character of~$W_E$, and it has the same parity as~$\tau$. Since induction for self-dual representations preserves the parity, we deduce that~$\sigma$ and~$\sigma'$ share the same parity too. 

Assume then that~$\tau$ is not self-dual. Then~$\tau^\vee$ is necessarily isomorphic to~$\tau^\gamma$, where~$\gamma$ is the order~$2$ element of~$\Gal(E/F)$. Let~$\tilde E = E^\gamma$, so that~$E/\tilde E$ is quadratic, and let~$T$ be the (irreducible) representation of~$W_{\tilde E}$ induced from~$\tau$. As~$\tau^\vee \simeq \tau^\gamma$, we see that~$T$ is self-dual. Its restriction to~$W_E$ is~$\tau \oplus \tau^\vee$, with~$\tau$ not isomorphic to~$\tau^\vee$, so the~$W_E$-invariant bilinear forms on the space of~$T$ form a space of dimension~$2$, with a line of alternating forms and a line of symmetric ones. Each of these lines is invariant under~$\Gal(E/F)$, one offering the trivial representation, the other the order~$2$ character~$\omega$ of~$\Gal(E/F)$. The self-dual unramified twist~$T'$ of~$T$ which is not isomorphic to~$T$ is~$T' = \eta T$ where~$\eta$ is unramified of order~$4$, so that~$T' \otimes T' \simeq \omega T \otimes T$. From the previous analysis, we deduce that if~$T$ is symplectic then~$T'$ is orthogonal and conversely:~$T$ and~$T'$ have different parities. By induction again we see that~$\sigma$ and~$\sigma'$ have different parities. 

\subsection{}\label{3.4new}
Let us look at some special cases. Assume first that~$\sigma$ is tame. Then~$t(\sigma) = \dim(\sigma)$. Introducing~$E$ and~$\tau$ as in paragraph~\ref{3.3new}, we have that~$\tau$ is a character, regular under the action of~$\Gal(E/F)$. If~$\tau$ were self-dual it would have order~$1$ or~$2$, but any character of~$E^\times$ of order~$1$ or~$2$ factors through~$N_{E/F}$, hence can be regular under the action of~$\Gal(E/F)$ only if~$t(\sigma) = \dim(\sigma)=1$, so~$E=F$. Thus, apart from quadratic characters of~$W_F$, \emph{tame} self-dual irreducible representations~$\sigma$ of~$W_F$ have even dimension, and we can apply the discussion of paragraph~\ref{3.3new} to them, concluding that~$\sigma$ and~$\sigma'$ have different parities. 

\subsection{}\label{3.5new}
We now assume that~$\sigma$ is \emph{not tame}, but we concentrate on our case of interest: that is, we assume from now on that~\emph{$p$ is odd}. We want in that case to spot when~$\sigma$ and~$\sigma'$ have the same parity, and then try to say whether they are orthogonal or symplectic. 

Let us first analyse~$\sigma$. Its restriction to the wild ramification subgroup~$\Pp_F$ of~$W_F$ is non-trivial, since~$\sigma$ is not tame. Let~$\gamma$ be an irreducible component of this restriction -- so that~$\gamma$ is not the trivial character of~$\Pp_F$ -- and~$S=S_\gamma$ its stabilizer in~$W_F$. Then by Clifford theory~$\sigma$ is induced from the representation of~$S$ on the isotypical component~$\V(\gamma)$ of~$\gamma$ in the space~$\V$ of~$\sigma$. 

Now by assumption~$\sigma$ is self-dual, and so is its restriction to~$\Pp_F$. But~$\Pp_F$ is a pro-$p$-group and~$p$ is odd, so no non-trivial irreducible representation of~$\Pp_F$ is self-dual, and we see that~$\gamma^\vee$ is not isomorphic to~$\gamma$. Thus there is~$g$ in~$W_F\setminus S$ with~$^g \gamma$ isomorphic to~$\gamma^\vee$; the coset~$gS$ is the same for all possible choices of~$g$, and~$g^2$ belongs to~$S$, so~$\tilde S = S \cup gS$ is a subgroup of~$W_F$ containing~$S$ as an index 2 subgroup. 

To get~$\sigma$, we can first induce~$\V(\gamma)$ from~$S$ to~$\tilde S$, and then from~$\tilde S$ to~$W_F$. We shall prove now that~$\Ind_S^{\tilde S} \V(\gamma)$ is self-dual; its parity is then inherited by~$\sigma$. This reduces the problem to understanding the parity of~$\Ind_S^{\tilde S} \V(\gamma)$.

\subsection{}\label{3.7new}
To prove that~$\Ind_S^{\tilde S} \V(\gamma)$ is self-dual, we take an abstract viewpoint: 

\begin{Proposition}\label{prop:IndHGrho}
Let~$\G$ be a group with a subgroup~$\H$ of index 2, and let~$g \in \G\setminus\H$. Let~$(\rho, \V)$ be an irreducible representation of~$\H$. Assume that~$\rho$ is not self-dual, but that~$\rho^\vee$ is equivalent to~${}^g \rho$. Then~$\Ind_\H^\G \rho$ is irreducible and self-dual. If~$\dim \V$ is odd, then~$\Ind_\H^\G \rho$ is symplectic if and only if its determinant is trivial. 
\end{Proposition}

\begin{proof}
Since~${}^g \rho$ is not isomorphic to~$\rho$, the induced representation~$\Ind_\H^\G \rho$ is irreducible, and it is self-dual because~$(\Ind_\H^\G \rho)^\vee$ is isomorphic to~$\Ind_\H^\G \rho^\vee$ hence to~$\Ind_\H^\G { }^g\rho$, itself isomorphic to~$\Ind_\H^\G \rho$. If~$\Ind_\H^\G \rho$ is symplectic, then clearly its determinant is trivial. To prove the converse statement when~$\dim \V$ is odd, we need to analyse the situation carefully. 

Since~$\rho^\vee$ is equivalent to~${}^g\rho$, there is a non-degenerate bilinear form~$\Phi: \V \times \V \to \mathbb C$ such that 
\[
\Phi(hv, ghg^{-1} v') = \Phi(v,v') \ \text{ for all } h \in \H, v, v' \in \V. 
\]
It is unique up to scalar. We claim that the form~$\Psi$, defined by~$\Psi(v,v')= \Phi(v', g^2v)$ for~$v,v'$ in~$\V$, is proportional to~$\Phi$. Indeed, for~$v,v'\in\V$ and~$h\in\H$, we find
\[
\Psi(hv, ghg^{-1} v') = \Phi( ghg^{-1} v', g^2hv) = \Phi(v', g^2v) = \Psi(v,v').
\]
Writing~$\Psi= \lambda \Phi$ with~$\lambda \in \mathbb C^\times$, we compute 
\[
\Phi(v',v)=\Phi(g^2v',g^2v)= \Psi(v,g^2v') = \lambda\Phi(v,g^2v')= \lambda \Psi(v',v) = \lambda^2 \Phi(v',v) 
\]
for~$v, v'$ in~$\V$ so that~$\lambda^2=1$. We shall see that the parity of~$\Ind_\H^\G \rho$ is governed by the scalar~$\lambda$. 

On the space~$\V \oplus \V$ equipped with the representation~$\rho \oplus {}^g \rho$, there is an~$\H$-invariant symplectic form~$f$, unique up to scalar, which we can take to be 
\[ 
f : ((v_1, v_2), (w_1,w_2)) \longmapsto \Phi(v_1, w_2) - \Phi(w_1, v_2). 
\]
The space of~$\Ind_\H^\G \rho$ can be taken as~$\V \oplus \V$ where~$\H$ acts as~$\rho \oplus {}^g \rho$ and~$g$ acts via 
\[
g(v_1, v_2) = (v_2, g^2 v_1).
\]
Since~$\Phi (v_2, g^2 v_1) = \Psi (v_1, v_2) = \lambda \Phi (v_1, v_2)$, we get that~$g$ acts on~$f$ by multiplication by~$-\lambda$, so~$\Ind_\H^\G \rho$ is symplectic if and only if~$\lambda=-1$. 

Let us choose a basis~$(e_1, \dots, e_d)$ of~$\V$, where~$d=\dim\rho$. Then~$\Phi (v_1, v_2) = ({}^t \mathbf{x}_1) H \mathbf{x}_2$, for~$v_1, v_2\in\V$ with coordinates given by~$\mathbf{x}_1, \mathbf{x}_2\in\mathbb C^d$ respectively, and~$H$ the~$d \times d$ Gram matrix of~$\Phi$ in the basis. If~$M_{g^2}$ is the matrix of~$\rho(g^2)$ we get~$\Phi (v_2, g^2 v_1) = ({}^t \mathbf{x}_2) H M_{g^2} \mathbf{x}_1$, from which we deduce that~$H M_{g^2} = \lambda ({}^tH)$, which implies that~$\det \rho(g^2)= \lambda^d$. 

Now~$\det(\Ind_\H^\G \rho)$ is an order~$2$ character of~$\G$ which is trivial on~$\H$; in fact it is given by 
\[
(\det \rho \circ \Ver ) \ \omega^{\dim \rho} 
\]
where~$\Ver : \G \longmapsto \G^\text{ab} \longmapsto \H^\text{ab}$ is the transfer and~$\omega$ is the non-trivial character of~$\G$ trivial on~$\H$. In this special case where~$\H$ has index~$2$ in~$\G$, the transfer map~$\Ver$ is trivial on~$\H$ and sends~$g$ to~$g^2$, so~$\det(\Ind_\H^\G \rho (g)) = (-\lambda)^d$. 

When~$d$ is odd, we find that~$\Ind_\H^\G \rho$ is symplectic if and only if its determinant is trivial, as desired.
\end{proof}

\begin{Remark}
When~$d$ is even,~$\Ind_\H^\G \rho$ always has trivial determinant, regardless of its parity. Determining the parity amounts to computing the scalar~$\lambda$. 
\end{Remark} 

\subsection{}\label{3.8new} 
We revert to the context of paragraphs~\ref{3.3new}--\ref{3.5new}. We want to spot the cases where~$\sigma$ and~$\sigma'$ (in the notation of paragraph~\ref{3.3new}) have the same parity, and in those cases possibly apply~\ref{prop:IndHGrho}~Proposition to determine that parity. For that we have to analyse the situation further. 
 
It is known (see~\cite[1.3~Proposition]{BHMemoir}) that~$\gamma$ extends to a representation~$\Gamma$ of~$S=S_\gamma$, and we can even impose that~$\det \Gamma$ have order a power of~$p$; then~$\Gamma$ is unique up to twist by an unramified character of~$S$, of order a power of~$p$. Since~${}^g\gamma$ is equivalent to~$\gamma^\vee$, we see that~${}^g \Gamma$ is equivalent to~$\chi \Gamma^\vee$ where~$\chi$ is an unramified character of~$S$ of order a power of~$p$. Such a~$\chi$ has a unique square root~$\eta$ with order a power of~$p$ and replacing~$\Gamma$ with~$\eta^{-1} \Gamma$, we may -- and do -- assume that~${}^g \Gamma \simeq \Gamma^\vee$. This now specifies~$\Gamma$ completely. 
 
As a representation of~$S$, the space~$\V(\gamma)$ is a tensor product~$\Gamma \otimes \delta$, where~$\delta$ is an irreducible representation of~$S$ trivial on~$\Pp_F$, well-defined up to isomorphism. Since~${}^g \V(\gamma) \simeq \V(\gamma)^\vee$ as representations of~$S$, we get that~${}^g \delta \simeq \delta^\vee$. 
 
Let~$K$ be the fixed field of~$S$, and~$\tilde K$ that of~$\tilde S$; thus the extension~$K/ \tilde K$ is quadratic, in particular tame. Writing~$d=\dim \delta$, the representation~$\delta$ is induced from a character~$\alpha$ of the unramified degree~$d$ extension~$K_d$ of~$K$ in~$\bar F$, with~$\alpha$ tamely ramified and regular under the action of~$\Gal (K_d/ K)$; this character~$\alpha$ is determined up to the action of~$\Gal (K_d/ K)$. 

In those terms, we try to see when~$\sigma$ and~$\sigma'$ have the same parity; that is, writing~$\sigma = \Ind \tau$ as in paragraph~\ref{3.3new}, where~$\tau$ is a representation of~$W_E$ with~$E/F$ unramified of degree~$t(\sigma)$, we want to know if~$\tau$ is self-dual. Note that~$t(\V(\gamma))=d$, so~$t(\sigma)= d f(K/F)$, where~$f(K/F)$ is the inertia degree of~$K/F$. The extension~$K_d/E$ is totally tamely ramified, and we can take~$\tau$ to be~$\Ind (\Gamma\otimes \alpha)$ where the induction is from~$W_{K_d}$ to~$W_E$ (and we first restrict~$\Gamma$ from~$S$ to~$W_{K_d}$). 

\subsection{}\label{3.9new} 
The following result describes when~$\sigma,\sigma'$ have the same parity.
\begin{Proposition}\label{prop:indparity}
Let~$\sigma$ be a self-dual irreducible representation of~$W_F$. Assume~$\sigma$ is not tame, and adopt the above notation. Then the following are equivalent:
\begin{enumerate}
\item\label{prop:indparity.i} $\sigma$ and~$\sigma'$ have the same parity; 
\item\label{prop:indparity.ii} $K/\tilde K$ is ramified and~$d=1$. 
\end{enumerate}
When these conditions are satisfied,~$\sigma$ and~$\sigma'$ are symplectic if and only if the character~$\alpha$ is ramified. 
\end{Proposition}
 
\begin{Remark}
When~$d=1$, we see that~$\alpha$ is a tame character of~$K^\times$ which satisfies~$^g\alpha = \alpha^{-1}$. If~$K/\tilde K$ is ramified,~$g$ acts trivially on the residue field of~$K$, and~$\alpha_{|U_K}$ has order~$1$ or~$2$. In that case, let~$\varpi$ be a uniformizer of~$K$ with~$\varpi^2 \in \tilde K$; then the condition~$^g\alpha = \alpha^{-1}$ translates into~$\alpha(-\varpi^2)=1$: either~$\alpha$ is unramified of order~$1$ or~$2$ or~$\alpha_{|\tilde K^\times}$ is the quadratic character~$\omega_{K/\tilde K}$ defining~$K$. 
\end{Remark}
 
\begin{proof}
To prove the proposition, we need to see when~$\tau = \Ind_{W_{K_d}}^{W_E} (\Gamma\otimes \alpha)$ is self-dual. The restriction of~$\Gamma\otimes \alpha$ to~$\Pp_F$ is~$\gamma$, so~$\tau$ can be self-dual only if there is~$h$ in~$W_E$ such that~$^h \gamma \simeq {}^g \gamma$ -- that is~$h \in gS$, or equivalently~$W_E \cap S \ne W_E \cap \tilde S$. Recalling that~$E$ is the maximal unramified extension of~$F$ in~$K_d$, we see that the fixed field of~$W_E \cap S$ is~$K_d$; if~$K/\tilde K$ were unramified, the fixed field of~$W_E \cap \tilde S$ would also be~$K_d$, so that~$\tau$ could not be self-dual. 

Thus, if~$\tau$ is self-dual then~$K/\tilde K$ is ramified and we take~$g$ in~$W_E \cap \tilde S$. Reasoning as in paragraph~\ref{3.5new} and using~\ref{prop:IndHGrho}~Proposition, we see that~$\tau$ is self-dual if and only if~$\Gamma\otimes \alpha$ induces to a self-dual representation of~$W_{\tilde K_d}$, where~$\tilde K_d$ is the fixed field of~$g$ in~$K_d$ (so that~$K_d/ \tilde K_d$ is quadratic ramified); in particular we then have~$^g(\Gamma\otimes \alpha) \simeq ( \Gamma\otimes \alpha )^\vee$. Since~$^g \Gamma \simeq \Gamma ^\vee$ by construction, this implies~$^g \alpha = \alpha^{-1}$ and since~$K_d/ \tilde K_d$ is ramified,~$g$ acts trivially on the residue field of~$K_d$ so~$\alpha_{|U_{K_d}}$ has order 1 or 2 and regularity with respect to~$\Gal(K/ K_d)$ implies~$d=1$. Thus if~$\tau$ is self-dual then~$d=1$, which proves~\ref{prop:indparity.i}~$\Rightarrow$~\ref{prop:indparity.ii}. 
 
Conversely if~\ref{prop:indparity.ii} is satisfied then~$\tau$ is self-dual if and only if~$^g \alpha = \alpha^{-1}$ by the above analysis, which gives \ref{prop:indparity.ii}~$\Rightarrow$~\ref{prop:indparity.i}.
 
Assume finally that conditions~\ref{prop:indparity.i} and~\ref{prop:indparity.ii} are satisfied. Using again~\ref{prop:IndHGrho}~Proposition, we have to check whether the determinant of~$\Ind_{W_K}^{W_{\tilde K}} (\Gamma\otimes \alpha)$ -- which by self-duality has order~$1$ or~$2$ -- is trivial. Seeing that determinant as a character of~$\tilde K^\times$ (via class field theory), it is equal to 
\[
\nu = \det(\Gamma\otimes \alpha)_{| \tilde K^\times} (\omega_{K/\tilde K})^{\dim \gamma}.
\]
But~$\det \Gamma$ has order a power of~$p$ and~$p$ is odd, and~$\alpha$ has order at most~$4$ (cf. the remark above) so we find~$\nu = (\alpha_{|\tilde K^\times}\omega_{K/\tilde K})^{\dim \gamma}$. If~$\alpha$ is unramified then~$\nu = \omega_{K/\tilde K}$ (since~${\dim \gamma}$ is odd) is non-trivial; if~$\alpha$ is ramified then~$\alpha_{|\tilde K^\times} = \omega_{K/\tilde K}$ by the remark and~$\nu$ is trivial. The final claim of the proposition now follows from~\ref{prop:IndHGrho}~Proposition.
\end{proof} 
 
\subsection{}\label{3.10new} 
Now we interpret the conditions of~\ref{prop:IndHGrho}~Proposition in terms of the cuspidal representation~$\rho$ of~$\GL_n(F)$, with~$n=\dim\sigma$, which corresponds to~$\sigma$ under the Langlands correspondence. To describe this representation~$\rho$ we will use the machinery of the construction of cuspidal representations as in~\S\ref{JBN1new}.
 
Assume~$\sigma$ is not tame, i.e.~$\rho$ is not of depth zero. Then~$\rho$ contains a simple character~$\tilde\theta$, belonging to a set of simple characters built using an element~$\beta \in \GL_n(F)$ which generates a field~$F[\beta]$. We have~$n = d [F[\beta]:F]$, so that~$d$ is determined by~$\rho$. Moreover the extension~$K/F$ which appears above in the discussion on the construction of~$\sigma$ is isomorphic to the maximal tame subextension~$L/F$ of~$F[\beta]/F$ (see~\cite[Tame Parameter Theorem]{BHMemoir}). 
 
When~$\rho$ -- equivalently~$\sigma$ -- is self-dual, we can choose~$\beta$ such that the self duality comes from an automorphism~$x \mapsto \bar x$ of~$F[\beta]$, sending~$\beta$ to~$-\beta$, and~$\tilde\theta$ to~$\tilde\theta^{-1}$ (see~\cite[Theorem~1]{BlSp2N}). That automorphism induces an order~$2$ automorphism of~$L$; let~$\tilde L$ be its fixed field. 
 
\begin{Proposition}\label{prop:sameext}
The extensions~$K/\tilde K$ and~$L/ \tilde L$ are isomorphic. 
\end{Proposition}
 
Thus condition~\ref{prop:indparity.ii} in~\ref{prop:indparity}~Proposition can be translated in terms of~$\rho$. See below (paragraph~\ref{3.11new}) for a translation of the last assertion of \textit{loc.~cit.}. 
 
\begin{proof}
The proof relies on the compatibility of tame lifting of simple characters with the induction process for Weil group representations~\cite{BH,BH1,BHMemoir}. Choose an isomorphism~$\iota$ of~$K/F$ onto~$L/F$. 
 
The representation~$\sigma_K$ of~$W_K$ on~$\V(\sigma)$ corresponds to a (cuspidal) representation~$\rho_L$ of~$\GL_m(L)$, where~$m=d [F[\beta]:L]$; the simple character~$\tilde\theta_L$ appearing in~$\rho_L$ is an~$L/F$-lift of~$\tilde\theta$ and~$L/F$ is the maximal tame extension such that~$\tilde\theta$ has a lift to~$\GL_{[F[\beta]:L]}(L)$. 
 
If~$K'$ is intermediate between~$F$ and~$K$, and~$L' = \iota(K')$, then~$\sigma_{K'} = \Ind_{W_K}^{W_{K'}} \V(\sigma)$ corresponds to a (cuspidal) representation~$\rho_{L'}$ of~$\GL_{m'}(L')$, with~$m'=d [F[\beta]:L']$, and the simple character~$\tilde\theta_{L'}$ appearing in~$\rho_{L'}$ is an~$L'/F$-lift of~$\tilde\theta$ and lifts to~$\tilde\theta_L$ in~$L/L'$. But~$\tilde K$ is the maximal intermediate subfield~$K'$ such that~$\sigma_{K'}$ is self-dual. Because the Langlands correspondence is compatible with taking contragredients, the field~$\iota(\tilde K)$ is the maximal field~$L'$ intermediate between~$F$ and~$L$ such that~$\tilde\theta_{L'}$ is self-dual (i.e. conjugate to~$\tilde\theta_{L'}^{-1}$ in~$\GL_{[F[\beta]:L']}(L')$). Thus~$\iota(\tilde K) = \tilde L$. 
\end{proof}
 
\subsection{}\label{3.11new} 
Now assume that~$d=1$ and~$L/ \tilde L$ (or equivalently~$K/\tilde K$) is ramified. We want to express the condition that~$\alpha$ is ramified in~\ref{prop:indparity}~Proposition in terms of~$\rho$. For that we have to review a little bit the construction of~$\rho$ from~$\tilde\theta$ from Section~\ref{JBN1new}, whose notation we use.  
 
We also continue with the notation for~$\rho$ introduced in the previous paragraph. Recall that, since~$d=1$, we have~$n = [F[\beta]:F]$. The simple character~$\tilde\theta$ is a character of~$\tilde H^1$ and we have the open subgroups~$\tilde J^1$,~$\tilde J$ of~$\GL_n(F)$. We write~$\tilde\eta$ for the unique irreducible representation of~$\tilde J^1$ containing~$\tilde\theta$. Then~$\tilde J=U_{F[\beta]}\tilde J^1$ and, by the Types Theorem~\cite[7.6 Theorem]{BHMemoir}, there is a unique beta-extension~$\tilde\kappa$ such that~$\tr\tilde\kappa$ is constant on the roots of unity of~$F[\beta]$ of order prime to~$p$ which are regular for the action of~$\Gal(K_{nr}/F)$, where~$K_{nr}$ is the maximal unramified extension of~$F$ in~$K$. Moreover, the same result gives that~$\rho$ contains the representation~$\tilde\kappa\otimes\omega \alpha$ of~$J$, where~$\alpha$ is seen as a character of~$J/J^1=\simeq U_{F[\beta]}/U^1_{F[\beta]}\simeq U_K/U^1_K$ and~$\omega$ is the order 2 character of~$U_{F[\beta]}$. 
 
Thus we conclude that~$\rho$ is symplectic when~$\rho$ contains~$\tilde\kappa$, and is orthogonal when~$\rho$ contains~$\omega \tilde\kappa$. 

\subsection{}\label{3.12new} 
We have discussed at length above the ambiguity between~$\sigma$ and~$\sigma'$ inherent to our method -- of course when~$\sigma$ and~$\sigma'$ have different parities it is the orthogonal one that features. 

Let us now briefly mention a few favourable circumstances when our methods do allow us to determine completely the parameter of a cuspidal representation~$\pi$ of~$\Sp_{2N}(F)$. 

Since the parameter~$\phi$ of~$\pi$ is orthogonal of dimension~$2N+1$, one irreducible component must have odd dimension. But in our case where~$p$ is odd, the only irreducible orthogonal representations of~$W_F$ with odd dimension are the four quadratic characters of~$W_F$. Thus at least one of them, say~$\omega$, has to occur in the parameter, and if the Jordan block it belongs to is~$(\omega, m)$ then~$m$ has to be congruent to~$1\pmod 4$, to yield an odd-dimensional contribution to~$\phi$; the contribution to the determinant is then~$\omega$. We then see that if we know all other components, then we can decide between~$\omega$ and~$\omega'$ by taking into account the condition~$\det \phi = 1$. To know all the other components~$\sigma$, it is necessary that for each of them,~$\sigma$ and~$\sigma'$ have different parities. We conclude that it will be rather rare that we determine~$\phi$ without ambiguity. 

Let us give just a few examples in low dimension. See~\cite{LS} for a discussion of depth zero cases. 

\underline{$N=1$,~$\SL_2(F)$} 

The parameter is either~$\rho \oplus \omega$ with~$\rho$ irreducible orthogonal of dimension 2 and~$\omega = \det \rho$, or~$\omega_1 \oplus \omega_2 \oplus \omega_3$ where the~$\omega_i$'s are the non-trivial quadratic characters of~$W_F$. In terms of homomorphisms~$W_F \longrightarrow \SO_3(\mathbb C) \simeq \PGL_2(\mathbb C)$, the second case corresponds to a triply imprimitive representation of~$W_F$, the first case to a simply imprimitive one~\cite{BHbook}. In the first case, our methods allow us to determine~$\rho$ only if it is induced from the quadratic unramified extension of~$F$ (i.e., in fact, when~$\omega$ is unramified of order~$2$). 

\underline{$N=2$,~$\Sp_4(F)$} 

There has to be a quadratic character~$\omega$ of~$W_F$ occurring with Jordan block~$(\omega, 1)$ only. If another quadratic character~$\eta$ occurs, the Jordan block can be~$(\eta, 1)$ or~$(\eta, 3)$. In the latter case~$\phi= \omega \oplus \eta \oplus \eta\otimes \St_3$ and the determinant condition implies that~$\omega$ is trivial and consequently that~$\eta$ is not trivial. If our computation shows that both~$1$ and the non-trivial quadratic unramified character~$\omega_{nr}$ occur, then the parameter is necessarily~$\phi= 1 \oplus \omega_{nr} \oplus \omega_{nr}\otimes \St_3$; if, on the contrary, our method gives that a ramified quadratic character~$\eta$ occurs, then we cannot distinguish between~$\eta$ and~$\eta' = \eta \omega_{nr}$. 

Let us look at the case where two distinct characters~$\omega$,~$\eta$ occur with Jordan blocks~$(\omega,1)$ and~$(\eta, 1)$ only. Then a third character,~$\nu$ say, must also occur and~$\phi= \omega \oplus \eta \oplus \nu \oplus \rho$ where~$\rho$ is irreducible orthogonal of dimension two. The determinant of~$\rho$ is the quadratic character~$\omega_{E/F}$ defining the extension from which~$\rho$ is induced so that the determinant condition on~$\phi$ is~$\omega \eta \nu \omega_{E/F} =1$. 

When~$E/F$ is unramified, there is no ambiguity in~$\rho$ in our computation, and the parameter is 
\[
\phi= 1 \oplus \mu \oplus \mu' \oplus \rho
\]
where~$\mu$,~$\mu'$ are the two ramified quadratic characters of~$W_F$. 

When~$E/F$ is ramified, the parameter could be 
\[
\phi= 1 \oplus \omega_{nr} \oplus \omega_{nr}\omega_{E/F} \oplus \rho, 
\qquad \text{ or } \qquad
\phi= 1 \oplus \omega_{nr} \oplus \omega_{nr}\omega_{E/F} \oplus \rho'
\]
and we cannot resolve the ambiguity between~$\rho$ and~$\rho'$. 

Finally if there is only one quadratic character~$\omega$ of~$W_F$ occurring in~$\phi$, we can compute~$\omega$, and thus determine~$\phi$ completely, only if the other components (necessarily even-dimensional) offer no ambiguity. 

We hope to come back to the case of~$\Sp_4(F)$ in a sequel to this paper, where a refinement of our methods will allow a more complete determination of~$\phi$.

\section{Langlands correspondence and ramification}\label{JBN7new}

In this final section we interpret our results on the endoscopic transfer map in terms of the Langlands correspondence for~$\SpFV$. In particular, we prove a Ramification Theorem for the symplectic group~$\SpFV$, giving a bijection between self-dual endo-classes and self-dual orbits of irreducible representations of the wild inertia group~$\Pp_F$ which is simultaneously compatible (in a suitable sense) with the Langlands correspondence for symplectic groups over~$F$ in all dimensions.

\subsection{}\label{7.1new}
We first recall the Ramification Theorem for general linear groups, from~\cite[8.2~Theorem]{BH2} (see also~\cite[6.3~Theorem]{BHMemoir}). Recall that~$\CE(F)$ denotes the set of endo-classes over~$F$. We write~$W_F\backslash\Irr(\Pp_F)$ for the set of~$W_F$-orbits of irreducible representations of~$\Pp_F$. By abuse of notation, we will identify such an orbit with the direct sum of the inequivalent irreducible representations in the orbit; thus, for~$\gamma$ an irreducible representation of~$\Pp_F$ with stabiliser~$S$, we identify its~$W_F$-orbit~$[\gamma]$ with~$\bigoplus_{W_F/S}{}^g\gamma$. In particular, we can then talk of the dimension of an orbit.

Given an irreducible representation of~$W_F$, by Mackey theory its restriction to~$\Pp_F$ is a multiple of a single~$W_F$-orbit of irreducible representations, so we get a natural map~$\Irr(W_F)\to W_F\backslash\Irr(\Pp_F)$, which is surjective. 

\begin{Theorem}\label{thm:ramGL}
There is a unique bijection~$\CE(F)\to W_F\backslash\Irr(\Pp_F)$,~$\bs\Theta\mapsto[\gamma(\bs\Theta)]$, which is compatible with the local Langlands correspondence:
\[
\xymatrix {\displaystyle\smash{\bigcup_{n\ge 1}}\Cusp(\GL_n(F)) \ar^{\qquad\sim}[r]\ar@{->>}[d] & \Irr(W_F) \ar@{->>}[d] \\
\CE(F) \ar^{\sim\quad}[r]& \mathcal W_F\backslash\Irr(\Pp_F)
}
\]
Moreover we have~$\deg(\bs\Theta)=\dim[\gamma(\bs\Theta)]$.
\end{Theorem}

\subsection{}\label{7.2new}
Now we consider how this bijection behaves with respect to duality. Recall that we write~$\CE^{\mathsf{sd}}(F)$ for the set of self-dual endo-classes; that is, those endo-classes~$\bs\Theta$ for which there is a self-dual simple character~$\tilde\theta$ with endo-class~$\bs\Theta$. If the endo-class is non-trivial then~$\tilde\theta$ is associated to a skew simple stratum~$[\Lambda,-,0,\beta]$ and the associated field~$E=F[\beta]$ has degree~$n$ over~$F$ and is equipped with a Galois involution with fixed field~$E_\so$. If~$\bs\Theta$ is the trivial endo-class then we have~$E=E_\so=F$.

It will be useful to have the following result, which guarantees the existence of self-dual cuspidal representations of general linear groups with given (self-dual) endo-class.

\begin{Lemma}\label{lem:existSD}
Let~$\bs\Theta$ be a self-dual endo-class and~$E/E_\so$ as above. Let~$m$ be an integer which is
\begin{enumerate}
\item odd, if~$E/E_\so$ is unramified quadratic,
\item $1$ or even, if~$E/E_\so$ is ramified quadratic,
\item even, if~$E=F$,
\end{enumerate}
and put~$n=m\deg(\bs\Theta)$. Then there are (at least) two inequivalent orthogonal self-dual cuspidal representations of~$\GL_n(F)$ with endo-class~$\bs\Theta$, and two inequivalent symplectic self-dual cuspidal representations of~$\GL_n(F)$ with endo-class~$\bs\Theta$.
\end{Lemma}

Note that, in the case that~$E=F$ (so~$\bs\Theta$ is trivial) and~$m=1$, there are four inequivalent self-dual (cuspidal) representations of~$\GL_1(F)$ with endo-class~$\bs\Theta$ but all four are orthogonal; they are the four quadratic characters.

\begin{proof}
Suppose first that~$\bs\Theta$ is non-trivial. Let~$\til\theta$ be a self-dual simple character with endo-class~$\bs\Theta$, as above, with associated skew simple stratum~$[\Lambda,-,0,\beta]$ and~$E=F[\beta]$. Then any transfer (in the sense of simple characters) of~$\tilde\theta$ is also self-dual, by~\cite[Corollary~2.13]{S4}.

Let~$m$ be an integer as in the hypotheses of the lemma and let~$f$ be a non-degenerate skew-hermitian form on an~$m$-dimensional~$E$-vector space~$V$ such that the associated unitary group (a group over~$E_\so$) is quasi-split. We write~$\oEo$ for the ring of integers of~$E_\so$ and~$\pE^\so$ for its unique maximal ideal, with~$k_E^\so$ the residue field. Fix~$\lambda_\so$ an~$F$-linear form on~$E_\so$ such that~$\{e\in E_\so\mid\lambda_\so(e\oEo)\subseteq\pF\}=\pE^\so$, and consider the form~$h=\lambda_\so\circ\tr_{E/E_\so}\circ f$ on~$V$. Thinking of~$V$ as an~$n$-dimensional~$F$-vector space, this is a nondegenerate alternating form. We take the transfer~$\tilde\theta_m$ of~$\tilde\theta$ to the unique (up to conjugacy) self-dual~$\oE$-lattice chain~$\Lambda_m$ on~$V$ such that~$\Lambda_m(0)\ne\Lambda_m(1)$. Thus~$\tilde\theta$ is a self-dual simple character of endo-class~$\bs\Theta$. 

Denote by~$\tilde{\underline\kappa}_m$ the unique~$p$-primary extension of~$\tilde\theta_m$, and denote by~$\tilde J_m$ the group on which it lives; then, by uniqueness,~$\tilde{\underline\kappa}_m$ is self-dual (that is, invariant under the involution~$\sigma$ defining the symplectic group~$\Sp_F(V)$). Now~$\tilde{\underline\kappa}_m$ extends to a representation~$\tilde{\underline\CK}_m$ of~$E^\times\tilde J_m$ with determinant a power of~$p$ and any two such extensions differ by an unramified character of order a power of~$p$. In particular,~$\tilde{\underline\CK}_m\circ\sigma$ is another such extension so has the form~$\tilde{\underline\CK}_m\otimes\chi$, for~$\chi$ unramified of order a power of~$p$. Since~$p$ is odd,~$\chi$ has a unique square root~$\chi'$ of order a power of~$p$, and then we can replace~$\tilde{\underline\CK}_m$ by~$\tilde{\underline\CK}_m\otimes\chi'$, which is self-dual.

Now we consider the quotient~$\tilde J_m/\tilde J^1_m \simeq \tilde P(\Lambda_{m,\oE})/\tilde P^1(\Lambda_{m,\oE})\simeq \GL_m(k_E)$. The involution~$\sigma$ also acts here, with fixed points a unitary group if~$E/E_\so$ is unramified and a symplectic group if~$E/E_\so$ is ramified (in the latter case, it is symplectic rather than orthogonal because~$\Lambda_m(0)\ne\Lambda_m(1)$); the action of~$\sigma$ is conjugate to the map transpose-inverse-$\Gal(k_E/k_E^\so)$-conjugate. The conditions on~$m$ are then precisely those required for the existence of a~$\Gal(k_E/k_E^\so)$-self-dual cuspidal representation~$\tilde\tau$ of~$\GL_m(k_E)$ (that is, such that the Galois conjugate of~$\tilde\tau$ is equivalent to~$\tilde\tau^\vee$) -- see~\cite[Theorem~7.1]{Adler} in the case~$k_E=k_E^\so$ and~\cite[Corollary~5.8]{Ka} in the case~$k_E\ne k_E^\so$. 

Let~$\omega$ be a quadratic character of~$E$, necessarily tame since~$p$ is odd. We also write~$\omega$ for the character of~$k_E^\times$ induced by restricting~$\omega$; then the representation~$\tilde\tau\omega$ is also~$\Gal(k_E/k_E^\so)$-self-dual. We inflate~$\tilde\tau\omega$ to~$\tilde J_m$ and extend to a representation~$\tilde\CT_\omega$ of~$E^\times\tilde J_m$ by setting~$\tilde\CT_\omega(\varpi_E)=\omega(\varpi_E)Id_{\tilde\tau\omega}$, for~$\varpi_E$ a fixed uniformizer of~$E$ such that~$\overline\varpi_E=(-1)^{e(E/E_\so)}\varpi_E$, where~$x\mapsto\overline x$ denotes the generator of~$\Gal(E/E_\so)$. This representation~$\tilde\CT_\omega$ is then self-dual, that is, equivalent to~$\tilde\CT_\omega\circ\sigma$.

Finally, the representation~$\rho_\omega=\cInd_{E^\times\tilde J}^{\GL_n(F)}\tilde{\underline\CK_m}\otimes\tilde\CT_\omega$ is then irreducible and cuspidal, and equivalent to~$\rho_\omega\circ\sigma$. Since the involution~$\sigma$ is a conjugate of the involution transpose-inverse, by a theorem of Gelfand--Kazhdan~\cite[Theorem~2]{GeK}, the representation~$\rho_\omega\circ\sigma$ is equivalent to~$\rho_\omega^\vee$. 

Thus we have constructed four self-dual cuspidal representations~$\rho_\omega$ of~$\GL_m(F)$ with endo-class~$\bs\Theta$, and it remains only to see that two are orthogonal and two symplectic. Note that~$\rho_\omega$ and~$\rho_{\omega'}$ are unramified twists of each other if and only if~$\omega^{-1}\omega'$ is the unramified quadratic character~$\omega_{nr}$. If either~$m>1$ or~$E/E_\so$ is unramified then, by~\ref{prop:indparity}~Proposition and~\ref{prop:sameext}~Proposition, the representations~$\rho_\omega$ and its self-dual unramified twist~$\rho_{\omega\omega_{nr}}$ have opposite parities so we are done. (Note that, writing~$L$ for the maximal tame subextension of~$E/F$ and~$L_\so$ for that of~$E_\so/F$, we have that~$L/L_{\so}$ is ramified if and only if~$E/E_{\so}$ is ramified, since~$p$ is odd.)

On the other hand, if~$m=1$ and~$E/E_\so$ is ramified then we are in the situation of paragraph~\ref{3.11new}, and the argument there explains that one pair~$\rho_\omega,\rho_{\omega\omega_{nr}}$ consists of two orthogonal representations, while the other pair consists of two symplectic representations, as required.

We are left with the case that~$\bs\Theta$ is the trivial endo-class and~$m$ is even. The existence of self-dual cuspidal depth zero representations is~\cite[Theorem~7.1]{Adler} and the argument that there are (at least) two orthogonal and two symplectic is formally exactly as in the previous case, with~$\tilde{\underline\CK}$ the trivial representation.
\end{proof}

We say that an orbit~$[\gamma]$ in~$W_F\backslash\Irr(\Pp_F)$ is self-dual if it is self-dual when considered as a representation of~$\Pp_F$; that is, if there is~$g\in W_F$ such that~$\gamma^\vee\simeq{}^g\gamma$. We write~$\left(W_F\backslash\Irr(\Pp_F)\right)^{\mathsf{sd}}$ for the set of self-dual orbits. Then we have:

\begin{Proposition}\label{prop:endosd}
The bijection of~\ref{thm:ramGL}~Theorem restricts to a bijection
\begin{equation}\label{eqn:endosd}
\CE^{\mathsf{sd}}(F)\to \left(W_F\backslash\Irr(\Pp_F)\right)^{\mathsf{sd}}.
\end{equation}
\end{Proposition}

\begin{proof}
Let~$\gamma$ be an irreducible representation of~$\Pp_F$ and put~$n=\dim[\gamma]$. Suppose that~$[\gamma]$ is a self-dual orbit and let~$g\in W_F$ be such that~${}^g\gamma\simeq \gamma^\vee$. Then, as in paragraph~\ref{3.8new}, there is a unique irreducible representation~$\Gamma$ of the stabilizer~$S$ of~$\gamma$ such that~$\det\Gamma$ has order a power of~$p$ and~${}^g\Gamma\simeq\Gamma^\vee$. Then the representation~$\Ind_S^{W_F}\Gamma$ is irreducible self-dual so the corresponding cuspidal representation~$\rho$ of~$\GL_n(F)$ is also self-dual. By~\cite[2.2~Corollary]{BlSp2N} (see also~\cite[p.10]{GKS}),~$\rho$ contains a simple character with self-dual transfer to~$\GL_{2n}(F)$, so the endo-class~$\bs\Theta(\rho)$, which corresponds to~$[\gamma]$ by~\ref{thm:ramGL}~Theorem, is self-dual.

Conversely, let~$\bs\Theta$ be a self-dual endo-class and put~$n=\deg(\bs\Theta)$. By the lemma, there is a self-dual cuspidal representation~$\rho$ of~$\GL_n(F)$ with endo-class~$\bs\Theta$. Then the corresponding irreducible representation of~$W_F$ is self-dual so the orbit in its restriction to~$\Pp_F$ is also self-dual, as required. 
\end{proof}

\subsection{}\label{7.3new}
We now introduce the notion of \emph{wild parameter}.

\begin{Definition}
A \emph{wild parameter (over~$F$)} is a finite-dimensional semisimple complex representation~$\CV$ of~$\Pp_F$ such that~$^g\CV\simeq \CV$, for all~$g\in W_F$. We write~$\Psi(F)$ for the set of equivalence classes of wild parameters over~$F$, and~$\Psi_n(F)$ for the set of equivalence classes of~$n$-dimensional wild parameters over~$F$.

Equivalently, we can think of an element of~$\Psi_n(F)$ as the~$\GL_n(\BC)$-conjugacy class of a homomorphism~$\psi:\Pp_F\to\GL_n(\BC)$ for which there exists~$A\in\GL_n(\BC)$ such that~$\psi\circ\Ad g=\Ad A\circ\psi$, for all~$g\in W_F$.
\end{Definition}

Thus a finite-dimensional semisimple complex representation~$\CV$ of~$\Pp_F$ is a wild parameter if and only if, when we decompose it into its isotypic components~$\CV=\bigoplus_{\gamma\in \Irr(\Pp_F)}\CV(\gamma)$, we have
\[
\dim \CV(\gamma)=\dim \CV({}^g\gamma), \text{ for all }g\in W_F.
\]
Therefore a wild parameter is equivalent to
\[
\bigoplus_{W_F\backslash \Irr(\Pp_F)} m_{[\gamma]}[\gamma],
\]
where we are thinking of the orbit~$[\gamma]$ as the sum over the~$W_F$-conjugates of~$\gamma\in\Irr(\Pp_F)$, and~$m_{[\gamma]}\in\BZ_{\ge 0}$.

Equivalently, the~$n$-dimensional wild parameters are precisely the restrictions to~$\Pp_F$ of the Langlands parameters for~$\GL_n(F)$; that is, writing~$\Phi_n(F)$ for the set of admissible homomorphisms~$\phi:W_F\times\SL_2(\BC)\to\GL_n(\BC)$ up to conjugacy, and~$\Phi(F)=\bigcup_{n\ge 1}\Phi_n(F)$, the natural map
\[
\Phi(F)\to\Psi(F)
\]
induced by~$\phi\mapsto\phi_{|\Pp_F}$ is surjective. Indeed, by taking direct sums one need only check that, for any~$\gamma\in\Pp_F$, there is a Langlands parameter~$\phi$ whose restriction to~$\Pp_F$ is isomorphic to~$[\gamma]$. This, however, follows from the discussion in paragraph~\ref{3.8new}:~$\gamma$ extends to a representation~$\Gamma$ of its stabiliser~$S_\gamma$ by~\cite[1.3~Proposition]{BHMemoir}, and then~$\Ind_{S_\gamma}^{W_F}\Gamma$ is the required Langlands parameter (with trivial~$\SL_2(\BC)$ action).

Recall from paragraph~\ref{2.7new} that an endo-parameter of degree~$n$ over~$F$ is a formal sum
\[
\sum_{\bs\Theta\in\CE} m_{\bs\Theta}\bs\Theta, \quad m_{\bs\Theta}\in\BZ_{\ge 0},
\quad\text{such that}\quad
\sum_{\bs\Theta\in\CE} m_{\bs\Theta}\deg(\bs\Theta) = n.
\]
We write~$\CEE_n(F)$ for the set of endo-parameters of degree~$n$ over~$F$. Then the Ramification Theorem for~$\GL_n$ (\ref{thm:ramGL}~Theorem) together with the compatibility of the Langlands correspondence with parabolic induction immediately give:

\begin{Theorem}\label{thm:ramGLSS}
The bijection of~\ref{thm:ramGL}~Theorem induces, for each~$n$, a bijection~$\CEE_n(F)\to\Psi_n(F)$ which is compatible with the Langlands correspondence:
\[
\xymatrix{
\Irr(\GL_n(F)) \ar^{\quad\sim}[r]\ar@{->>}[d] & \Phi_n(F) \ar@{->>}[d] \\
\CEE_n(F) \ar^{\sim}[r]& \Psi_n(F)
}
\]
\end{Theorem}

\subsection{}\label{7.4new}
Now we turn to the case of the symplectic group~$\SpFV$ and recall Arthur's local Langlands correspondence in this case. 

We denote by~$\Phi(\SpFV)$ the set of Langlands parameters for~$\SpFV$, that is, the set of conjugacy classes of homomorphisms~$\phi:W_F\times\SL_2(\BC)\to\SO_{2N+1}(\BC)$ such that the representation obtained by composing with the natural inclusion map~$\iota:\SO_{2N+1}(\BC)\into\GL_{2N+1}(\BC)$ is semisimple. 

We denote by~$\Phi^{\mathsf{disc}}(\SpFV)$ the set of \emph{discrete} Langlands parameters, that is, those which cannot be conjugated into a proper parabolic subgroup of~$\SO_{2N+1}(\BC)$; equivalently,~$\iota\circ\phi$ is a direct sum of inequivalent irreducible orthogonal representations of~$W_F\times\SL_2(\BC)$ and has determinant~$1$. Thus, given~$\phi$ a discrete Langlands parameter, the representation~$\iota\circ\phi$ decomposes as a multiplicity-free direct sum
\begin{equation}\label{eqn:disc}
\bigoplus_{i\in I} \sigma_i\otimes\St_{m_i},
\end{equation}
where~$\St_{m}$ denotes the unique~$m$-dimensional irreducible algebraic representation of~$\SL_2(\BC)$, for~$m\ge 1$, and the~$\sigma_i$ are irreducible self-dual representations of~$W_F$, such that
\begin{itemize}
\item $\sum_{i\in I} m_i\dim(\sigma_i)=2N+1$,
\item $\sigma_i$ is symplectic if~$m_i$ is even and orthogonal if~$m_i$ is odd,
\item $\prod_{i\in I}\det(\sigma_i)^{m_i}=1$.
\end{itemize}
We say that a discrete Langlands parameter~$\phi$ is \emph{cuspidal} if, whenever~$\sigma\otimes\St_m$ is a subrepresentation of~$\iota\circ\phi$ and~$m>2$, the representation~$\sigma\otimes\St_{m-2}$ is also a subrepresentation of~$\iota\circ\phi$. We denote by~$\Phi^{\mathsf{cusp}}(\SpFV)$ the set of cuspidal Langlands parameters.

As usual, for~$\phi$ a discrete Langlands parameter, we denote by~$\Ss_\phi$ the group of connected components of the centralizer in~$\SO_{2N+1}(\BC)$ of the image of~$\phi$. This is a finite product of copies of the cyclic group of order~$2$; if~$\iota\circ\phi$ decomposes as in~\eqref{eqn:disc}, then~$\Ss_\phi$ has order~$2^{\#I-1}$.

\begin{Theorem}[{\cite[Theorems~1.5.1 and~2.2.1]{Arthur},~\cite[Theorem~1.5.1]{Mo5}}]\label{thm:LangSp}
Suppose that~$F$ is of characteristic zero. 
There is a natural surjective map from the set of discrete series representations of~$\SpFV)$ to~$\Phi^{\mathsf{disc}}(\SpFV)$ with finite fibres, characterised by an equality of stable distributions via transfer to~$\GL_{2N+1}(\BC)$. Moreover:
\begin{itemize}
\item the fibre of~$\phi\in\Phi^{\mathsf{disc}}(\SpFV)$ is in bijection with the set of characters of~$\Ss_\phi$;
\item the fibre~$\Pi_\phi$ of~$\phi\in\Phi^{\mathsf{disc}}(\SpFV)$ contains a cuspidal representation of~$\SpFV$ if and only if~$\phi$ is cuspidal, in which case~$\Pi_\phi\cap\Cusp(\SpFV)$ is in bijection with the set of \emph{alternating} characters of~$\Ss_\phi$.
\end{itemize}
\end{Theorem}

We do not recall the definition of \emph{alternating} character (see~\cite[\S1.5]{Mo5}) but only recall that if, for~$\phi$ a cuspidal Langlands parameter as in~\eqref{eqn:disc}, we set~$I_0=\{\sigma_i\mid \sigma_i\text{ is orthogonal}\}$, then there are~$2^{\#I_0-1}$ alternating characters of~$\Ss_\phi$. (Note that~$I_0$ is non-empty, since one of the~$\sigma_i$ must be a quadratic character, so this makes sense.) In particular, the~$L$-packet of a cuspidal Langlands parameter~$\phi$ consists only of cuspidal representations if and only if~$m_i=1$, for all~$i\in I$ (in the description~\eqref{eqn:disc}; that is, each self-dual irreducible representation of~$W_F$ which appears in~$\phi$ is orthogonal and appears with multiplicity at most one. In this case, we say that~$\phi$ is \emph{regular}.

\subsection{}\label{7.5new}
We say that a wild parameter~$\CV$ is \emph{self-dual} if it is self-dual as a representation of~$\Pp_F$, in which case~$\det(\CV)$ is trivial (since~$p$ is odd).

Given a self-dual wild parameter, we would like to see that there is a unique choice of orthogonal structure on it. This is indeed a special case of the following result on the existence and uniqueness of orthogonal structures on self-dual representations of groups of odd order.

\begin{Proposition}\label{prop:uniqueorth}
Let~$\G$ be a finite group of odd order and let~$\CV$ be a finite dimensional complex representation of~$\G$. If~$\CV$ is self-dual, then~$\CV$ is orthogonal: there is on~$\CV$ a~$\G$-invariant non-degenerate symmetric bilinear form; moreover such a form is unique up to the action of~$\Aut_\G(\CV)$.
\end{Proposition}

In other words, a self-dual representation of~$\G$ is underlying a unique (up to isomorphism) orthogonal representation.

\begin{proof}
As~$\G$ has odd order, the only self-dual irreducible representation of~$\G$ is the trivial representation~$1_\G$. For an irreducible representation~$\gamma$ of~$\G$, let~$\CV(\gamma)$ be the~$\gamma$-isotypic component of~$\CV$, and put~$\CV_\gamma=\Hom_\G(\gamma,\CV)$, so that~$\CV(\gamma)$ decomposes canonically as~$\gamma\otimes \CV_\gamma$. Then~$\CV$ is self-dual if and only if~$\CV_\gamma$ and~$\CV_{\gamma^\vee}$ have the same dimension for all~$\gamma$.

Assume~$\CV$ is self-dual. For any~$\G$-invariant non-degenerate symmetric bilinear form on~$\CV$, we can write~$\CV$ as the orthogonal direct sum of its subspaces~$\CV(1_\G)$ and~$\CV(\gamma)\oplus \CV(\gamma^\vee)$, for~$\gamma$ running through a set of representatives of the non-trivial irreducible representations up to contragredient. On~$\CV(1_\G)$, where~$\G$ acts trivially, there is a non-degenerate symmetric bilinear form, unique up to the action of~$\Aut(\CV(1_\G))$. 

Therefore, for existence and uniqueness, it is enough to consider the case where~$\CV= \CV(\gamma) \oplus \CV(\gamma^\vee)$, for some non-trivial~$\gamma$. Then the dual of~$\CV(\gamma)$ is~$\gamma^\vee \otimes (\CV_\gamma)^*$, whereas the dual of~$\CV({\gamma^\vee})$ is~$\gamma \otimes (\CV_{\gamma^\vee})^*$. An isomorphism~$j:\CV\to \CV^\vee$ (that is, a self-duality on~$\CV$) is the direct sum of~$Id_\gamma \otimes i$ and~$Id_{\gamma^\vee} \otimes i'$, where~$i$ is an isomorphism of~$\CV_\gamma$ onto~$(\CV_{\gamma^\vee})^*$, and~$i'$ an isomorphism of~$\CV_{\gamma^\vee}$ onto~$(\CV_\gamma)^*$. The self-duality is orthogonal if and only if~$i$ and~$i'$ are transpose to each other.

Obviously there exists then an orthogonal structure on~$\CV$, and moreover all such structures are given by the choice of~$i$ (with~$i'$ its transpose). Since~$\Aut_\G(\CV)$, which is the product~$\Aut(\CV_\gamma)\times\Aut(\CV_{\gamma^\vee})$, acts transitively on the set of~$i$, we have uniqueness too.
\end{proof}

\subsection{}\label{7.6new}
Now let~$\CV$ be an~$n$-dimensional self-dual wild parameter over~$F$. By~\ref{prop:uniqueorth}~Proposition,~$\CV$ then carries a~$\Pp_F$-invariant nondegenerate symmetric bilinear form, unique up to the action of~$\Aut_{\Pp_F}(\CV)$. Thus we can regard~$\CV$ as a homomorphism~$\psi:\Pp_F\to \SO(\CV)\simeq \SO_{n}(\mathbb C)$.

For~$\gamma\in\Irr(\Pp_F)$, we write~$\CV[\gamma]$ for the component of~$\CV$ corresponding to the orbit of~$\gamma$ under~$W_F$; that is~$\CV[\gamma]=\sum_{g\in W_F} \CV({}^g\gamma)$. We consider the stabiliser in~$\SO(\CV)$ of the self-dual decomposition
\[
\CV=\bigoplus_{W_F\backslash \Irr(\Pp_F)} \CV[\gamma]
\]
and say that~$\CV$ is \emph{discrete} if this stabiliser is contained in no proper Levi subgroup of~$SO(\CV)$. Equivalently, the self-dual parameter~$\CV$ is discrete if and only if every orbit~$[\gamma]$ in the support of~$\CV$ (that is, such that~$\CV[\gamma]$ is non-zero) is self-dual. 

We write~$\Psi_n^{\mathsf{sd}}(F)$ for the set of equivalence classes of discrete self-dual~$n$-dimensional wild parameters over~$F$. Note that the restriction to~$\Pp_F$ of any discrete Langlands parameter for~$G$ is a discrete self-dual wild parameter of dimension~$2N+1$, which explains the nomenclature.

Recall also that we have the set~$\CEE_n^{\mathsf{sd}}(F)$ of self-dual endo-parameters of degree~$n$ over~$F$, which consists of those endo-parameters of degree~$n$ with support in the set~$\CE^{\mathsf{sd}}(F)$ of self-dual endo-classes. Then we have the following Ramification Theorem for~$\SpFV$.

\begin{Theorem}\label{thm:ramG}
The bijection~\eqref{eqn:endosd} induces, for each~$N\ge 1$, a bijection~$\CEE_{2N}^{\mathsf{sd}}(F)\to\Psi_{2N+1}^{\mathsf{sd}}(F)$ which, when~$F$ is of characteristic zero, is compatible with the Langlands correspondence for cuspidal representations of~$\SpFV$:
\[
\xymatrix{
\Cusp(\SpFV) \ar@{->>}[r]\ar@{->>}[d] & \Phi^{\mathsf{cusp}}(G) \ar@{->>}[d] \\
\CEE_{2N}^{\mathsf{sd}}(F) \ar^{\sim\ }[r]& \Psi_{2N+1}^{\mathsf{sd}}(F)
}
\]
\end{Theorem}

The induced bijection~$\CEE_{2N}^{\mathsf{sd}}(F)\to\Psi_{2N+1}^{\mathsf{sd}}(F)$ is not as obvious as in the case of general linear groups. If we denote the bijection~\eqref{eqn:endosd} by~$\bs\Theta\mapsto[\gamma(\bs\Theta)]$ then the induced map is
\begin{equation}\label{eqn:wildL}
\sum_{\bs\Theta}m_{\bs\Theta}\bs\Theta \mapsto 1_{\Pp_F}\oplus\bigoplus_{\bs\Theta} m_{\bs\Theta}[\gamma(\bs\Theta^2)].
\end{equation}
We remark also that this Theorem asserts that the restriction map~$\Phi^{\mathsf{cusp}}(G)\to\Psi_{2N+1}^{\mathsf{sd}}(F)$ is surjective, so that every discrete self-dual wild parameter of dimension~$2N+1$ occurs as the restriction of not only some discrete Langlands parameter for~$G$ but of some cuspidal parameter. In fact, we show that it occurs as the restriction of a \emph{regular} parameter (i.e. one whose~$L$-packet consists only of cuspidal representations).

\begin{proof}
Since the only irreducible self-dual representation of~$\Pp_F$ is the trivial representation (so the only odd-dimensional self-dual class~$[\gamma]$ is that of the trivial representation), while the squaring map on endo-classes is a bijection (since~$p$ is odd), it is clear that~\eqref{eqn:wildL} defines a bijection. Its compatibility with the Langlands correspondence is now just a reinterpretation of~\ref{thm:endoparameter}~Theorem, using~\ref{thm:ramGLSS}~Theorem.

It remains to prove that the vertical maps are surjective. We prove that the map on the right is surjective, and then surjectivity on the left follows. So let~$\CV=\bigoplus m_{[\gamma]}[\gamma]$ be a~$(2N+1)$-dimensional self-dual wild parameter (where the sum is over the~$W_F$ orbits in~$\Irr(\Pp_F)$ as usual). We will define a regular Langlands parameter~$\sigma=\bigoplus \sigma[\gamma]$ for~$\SpFV$ such that~$\sigma[\gamma]$ restricts to~$\CV[\gamma]=m_{[\gamma]}[\gamma]$.

Let~$\gamma\in\Irr(\Pp_F)$ be a non-trivial representation. If~$m_{[\gamma]}=0$ then we put~$\sigma[\gamma]=\{0\}$ so assume~$m_{[\gamma]}>0$, in which case the orbit~$[\gamma]$ is self-dual. Let~$\bs\Theta$ be corresponding (self-dual) endo-class and let~$E/E_\so$ be the quadratic extension associated to a skew simple stratum which has a simple character with endo-class~$\bs\Theta$. We pick non-negative integers~$m_1,m_2$ with~$m_1+m_2=m_{[\gamma]}$ such that:
\begin{enumerate}
\item $m_1,m_2$ are odd or~$0$ if~$E/E_\so$ is unramified;
\item $m_1,m_2$ are even or~$1$ if~$E/E_\so$ is ramified.
\end{enumerate}
For~$i=1,2$ we put~$n_i=m_i\deg(\bs\Theta)$. Then, by~\ref{lem:existSD}~Lemma, there exist inequivalent orthogonal self-dual cuspidal representations~$\rho_1,\rho_2$ of~$\GL_{n_1}(F),\GL_{n_2}(F)$ respectively, both with endo-class~$\bs\Theta$. Let~$\sigma_1,\sigma_2$ denote the corresponding Langlands parameters, which are orthogonal and, put~$\sigma[\gamma]=\sigma_1\oplus\sigma_2$; then the restriction of~$\sigma[\gamma]$ to~$\Pp_F$ is~$\CV[\gamma]$, as required, by~\ref{thm:ramGL}~Theorem.

Finally, put~$m=m_{[1_{\Pp_F}]}-1$, which is even. By~\ref{lem:existSD}~Lemma, there is an orthogonal self-dual depth zero cuspidal representation of~$\GL_m(F)$, and let~$\delta$ be the corresponding representation of~$W_F$. We put~$\sigma[1_{\Pp_F}]=\delta\oplus\omega$, where~$\omega=\det(\delta)\prod_{[\gamma]\ne[1_{\Pp_F}]}\det\sigma[\gamma]$.

Then~$\sigma=\bigoplus_{[\gamma]}\sigma[\gamma]$ is a regular cuspidal Langlands parameter for~$G$ which restricts to~$\CV$.
\end{proof}

\subsection{}\label{7.7new}
In the proof of~\ref{thm:ramG}~Theorem we saw that, for any self-dual wild parameter~$\CV$ of odd dimension, there is a regular Langlands parameter for~$\SpFV$ which restricts to~$\CV$. As well as this, one can (in general) cook up other examples of Langlands parameters which restrict to~$\CV$ and are highly irregular. Since we find it amusing, we include here a description of how to find a highly irregular Langlands parameter which restricts to~$\CV$.

We begin with the following observation, which is just the translation of~\ref{lem:existSD}~Lemma (with~$m=1$) to Galois representations. Suppose~$\gamma\in\Irr(\Pp_F)$ is non-trivial with self-dual~$W_F$-orbit. Then there are four self-dual representations of~$W_F$ whose restriction to~$\Pp_F$ is~$[\gamma]$, two of which are orthogonal and two of which are symplectic. We write~$\sigma_{\gamma,1},\sigma_{\gamma,2}$ for the two orthogonal ones, and~$\sigma_{\gamma,3},\sigma_{\gamma,4}$ for the two symplectic ones.

Now we decompose~$\CV=\bigoplus m_{[\gamma]}[\gamma]$ as above. As before, we will define a Langlands parameter~$\sigma=\bigoplus \sigma[\gamma]$, with~$\sigma[\gamma]_{|\Pp_F}=m_{[\gamma]}[\gamma]$. We will obtain a parameter which is not regular whenever either~$m_{[1_{\Pp_F}]}> 3$ or~$m_{[\gamma]}> 1$, for some non-trivial self-dual~$[\gamma]$. 

By Lagrange's~$4$-squares theorem, we can find non-negative integers such that
\[
4m_{[\gamma]} + 2 = a_1^2+a_2^2+a_3^2+a_4^2.
\]
Moreover, two of the~$a_i$ are even and the other two odd. We label them so that~$a_1,a_2$ are even and~$a_3,a_4$ are odd and, when~$m_{[\gamma]}=2$, we take the solution with~$a_1=a_2=0$. Then we set
\[
\sigma[\gamma]=\bigoplus_{i=1}^4 \sigma_{\gamma,i}\otimes\left(\St_{a_i-1}\oplus\St_{a_i-3}\oplus\cdots\right),
\]
where we understand that we ignore the terms on the right where~$a_i\le 1$.

Finally, write~$\omega_1=\prod_{[\gamma]\ne[1_{\Pp_F}]}\det\sigma[\gamma]$, which is a quadratic character, and let~$\omega_2,\omega_3,\omega_4$ denote the other three quadratic characters. Again, there are non-negative integers such that
\[
m_{[1_{\Pp_F}]} = a_1^2+a_2^2+a_3^2+a_4^2.
\]
Since~$m_{[1_{\Pp_F}]}$ is odd, there is exactly one~$a_i$ which has opposite parity to the other three, and we choose our numbering so that this is~$a_1$. Then we take
\[
\sigma[1_{\Pp_F}] = \bigoplus_{i=1}^4 \omega_i\otimes\left(\St_{2a_i-1}\oplus\St_{2a_i-3}\oplus\cdots\oplus\St_1\right),
\]
where, again, we ignore the terms for which~$a_i=0$.


\def\cprime{$'$}

\end{document}